\documentclass{eptcs-modified}

\usepackage{amsmath}
\usepackage{amssymb}
\usepackage{mathabx}
\usepackage{url}
\usepackage{subcaption}
\usepackage{amsthm}
\usepackage{graphicx}
\usepackage{float}
\usepackage{amscd}
\usepackage[usenames,dvipsnames,svgnames,table]{xcolor}
\usepackage{ifthen}
\usepackage{xargs}
\usepackage{cancel}

\usepackage{xcolor}
\usepackage{lmodern}
\usepackage{enumerate}
\usepackage{tikz}
\usepackage{tikzit}
\usepackage{theoremref}
\usetikzlibrary{arrows}
\tikzstyle{box}=[fill={rgb,255: red,228; green,228; blue,228}, draw=black, shape=rectangle]
\tikzstyle{reducible}=[->, dashed]
\tikzstyle{strictreducible}=[->]
\tikzstyle{nonreducible}=[->, draw=red]
\tikzstyle{uncomp}=[draw=red, <->]
\tikzstyle{thmref}=[opacity=0,inner sep=2mm]
\newcommand{\st}{{}\,:\,{}}
\newcommand{\effectivewadge}{\leq_{\mathsf{EW}}}
\newcommand{\constructionone}[1]{\mathcal{A}^{\sqsubseteq}(#1)}
\newcommand{\constructiontwo}[1]{\mathcal{A}^{\!\incomparable}(#1)}
\newcommand{\defas}{:=}
\newcommand{\function}[3]{{#1}:{#2}\rightarrow {#3}}
\newcommand{\partialfunction}[3]{{{#1}} :\subseteq {#2}\rightarrow {#3}}
\newcommand{\multifunction}[3]{{#1}: {#2}\rightrightarrows{#3}}
\newcommand{\partialmultifunction}[3]{{{#1}} :\subseteq {#2} \rightrightarrows {#3}}
\newcommand{\repmap}[1]{\delta_{#1}}
\newcommand{\str}[1]{\langle #1 \rangle}
\newcommand{\pair}[1]{\langle #1 \rangle}
\newcommand{\join}{\oplus}
\newcommand{\incomparable}{~|~}
\newcommand{\length}[1]{\left|{#1}\right|}
\newcommand{\body}[1]{\left[{#1}\right]}
\newcommand{\concat}{^\smallfrown}
\newcommand{\X}{\mathcal{X}}

\newcommand{\dom}{\operatorname{dom}}
\newcommand{\nats}{\mathbb{N}}
\newcommand{\Baire}{\nats^\nats}
\newcommand{\baire}{\nats^{<\nats}}
\newcommand{\Cantor}{2^\nats}
\newcommand{\cantor}{2^{<\nats}}
\newcommand{\weireducible}{\le_{\mathrm{W}}}
\newcommand{\strictlyweireducible}{<_\mathrm{W}}
\newcommand{\weiequiv}{\equiv_{\mathrm{W}}}
\newcommand{\weiincomparable}{~|_{\mathrm{W}~}}
\newcommand{\sweireducible}{\le_{\mathrm{sW}}}
\newcommand{\strongweireducible}{\le_{\mathrm{sW}}}

\newcommand{\sweiequiv}{\equiv_{\mathrm{sW}}}
\newcommand{\parallelization}[1]{\widehat{#1}}
\newcommand{\firstOrderPart}[1]{{}^1{#1}}
\newcommand{\ustar}[1]{#1^{u*}}
\newcommand{\boldfacePi}{\mathbf{\Pi}}
\newcommand{\boldfaceDelta}{\mathbf{\Delta}}
\newcommand{\boldfaceSigma}{\mathbf{\Sigma}}
\newcommand{\boldfaceGamma}{\mathbf{\Gamma}}
\newcommand{\wadgereducible}{\leq_{\mathbf{W}}}
\newcommand{\strictwadgereducible}[1]{<_{\mathbf{W}}}
\newcommand{\wadgeequiv}[1]{\equiv_{\mathbf{W}}}

\newcommand{\disjointunion}[2]{\bigsqcup_{{#1}}{#2}}
\newcommand{\range}{\mathsf{range}}
\newcommand{\mflim}{\mathsf{lim}}
\newcommand{\CBaire}{\mathsf{C}_{\Baire}}
\newcommand{\CCantor}{\mathsf{C}_{\Cantor}}
\newcommand{\UCBaire}{\mathsf{UC}_{\Baire}}
\newcommand{\wf}{\mathsf{WF}}
\newcommand{\PiCA}{\boldfacePi^1_1\mathsf{-CA}_0}
\newcommand{\choice}[2]{{#1\text{-}}\mathsf{C}_{#2}}
\newcommand{\cnats}{\mathsf{C}_\nats}
\newcommand{\uchoice}[2]{#1\text{-}\mathsf{UC}_{#2}}
\newcommand{\illfounded}{\mathcal{IF}}
\newcommand{\wellfounded}{\mathcal{WF}}
\newcommand{\D}[2]{\mathrm{deg}^{#1}(#2)}
\newcommand{\pathconnected}[4]{{#1}\!\leftrightsquigarrow_{#3}^{#4}\!{#2}}
\newcommand{\isg}{\mathsf{IS}}
\newcommand{\sg}{\mathsf{S}}

\newcommand{\connected}{\mathbf{Con}\text{-}}
\newcommand{\subgraph}{\subseteq_{\mathbf{s}}}
\newcommand{\inducedsubgraph}{\subseteq_{\mathbf{is}}}
\newcommand{\infinitelycycles}{\bigotimes_{i\geq 3}\cycle{i}}
\newcommand{\infinitelycopies}[1]{\overset{\infty}{\bigotimes}#1}
\newcommand{\quot}[1]{\lq\lq{#1}\rq\rq}
\newcommand{\id}{\mathsf{id}}
\newcommand{\allgraphrelations}{\subseteq_{(\mathbf{i})\mathbf{s}}}

\newcommand{\egraph}{EGr}
\newcommand{\graph}{Gr}
\newcommand{\R}{\mathsf{R}}
\newcommand{\elle}{\mathsf{L}}
\newcommand{\cycle}[1]{C_{#1}}
\newcommand{\ray}[1]{L_{#1}}
\newcommand{\complete}[1]{K_{#1}}
\newcommand{\embeddingray}[1]{\R\text{-}\mathsf{Emb}_{#1}}
\newcommand{\eembeddingray}[1]{e\R\text{-}\mathsf{Emb}_{#1}}

\newcommand{\jump}[2]{{#1}^{(#2)}}
\newcommand{\connectedunion}[1]{\bigodot_{#1}}
\newcommand{\disconnectedunion}[1]{\bigotimes_{#1}}
\newcommand{\findis}[2]{\mathsf{IS\text{-}Copy}^{#1}_{#2}}
\newcommand{\finds}[2]{\mathsf{S\text{-}Copy}^{#1}_{#2}}
\newcommand{\lpo}{ \mathsf{LPO} }
\newcommand{\tree}{\mathbf{Tr}}
\newcommand{\findisEC}[1]{\mathsf{IS\text{-}Copy}_{#1}^{e  \chi}}
\newcommand{\findisEE}[1]{\mathsf{IS\text{-}Copy}_{#1}^{e  e}}
\newcommand{\findisCC}[1]{\mathsf{IS\text{-}Copy}_{#1}^{\chi  \chi}}
\newcommand{\findisCE}[1]{\mathsf{IS\text{-}Copy}_{#1}^{\chi  e}}
\newcommand{\findsEC}[1]{\mathsf{S\text{-}Copy}_{#1}^{e  \chi}}
\newcommand{\findsEE}[1]{\mathsf{S\text{-}Copy}_{#1}^{e  e}}
\newcommand{\findsCC}[1]{\mathsf{S\text{-}Copy}_{#1}^{\chi  \chi}}
\newcommand{\findsCE}[1]{\mathsf{S\text{-}Copy}_{#1}^{\chi  e}}
\newcommand{\repspacegraphs}{\mathbf{Gr}}
\newcommand{\repspaceegraphs}{\mathbf{EGr}}
\newcommand{\repspacex}{\mathbf{X}}
\newcommand{\setofgraphs}[3]{\{H \in \mathbf{#2} :#1 #3 H\}}

\newcommand{\repspaceallgraphs}{(\mathbf{E})\mathbf{Gr}}

\newtheorem*{theorem*}{Theorem}

\newtheorem{theorem}{Theorem}[section]

\newtheorem{proposition}[theorem]{Proposition}
\newtheorem*{proposition*}{Proposition}

\newtheorem{lemma}[theorem]{Lemma}
\newtheorem{corollary}[theorem]{Corollary}

\theoremstyle{definition}
\newtheorem{definition}[theorem]{Definition}
\newtheorem*{definition*}{Definition}

\theoremstyle{remark}
\newtheorem{open}[theorem]{Question}
\newtheorem{remark}[theorem]{Remark}
\newtheorem*{remark*}{Remark}

\makeatletter

\begin{document}

	\title{Deciding embeddability of graphs}

	\title{Embeddability of graphs and Weihrauch degrees\thanks{Cipriani’s research was partially supported by the Italian PRIN 2017 Grant \lq\lq Mathematical Logic: models, sets, computability\rq\rq. He thanks Alberto Marcone and Manlio Valenti for useful discussions about the topics of this paper.}}

\author{
		Vittorio Cipriani
		\institute{Dipartimento di Scienze Matematiche, Informatiche e Fisiche\\Universit\`a Degli Studi di Udine, Udine, Italy\\}
		\email{vittorio.cipriani17@gmail.com }
		\and
		Arno Pauly
		\institute{Department of Computer Science\\Swansea University, Swansea, UK\\}
		\email{Arno.M.Pauly@gmail.com}
	}
	
	\def\titlerunning{Embeddability of graphs and Weihrauch degrees}
\def\authorrunning{V.~Cipriani \& A.~Pauly}
	
	\maketitle

\begin{abstract}
    We study the complexity of the following computational tasks concerning a fixed countable graph $G$: 1.\ Does a countable graph H provided as input have a(n induced) subgraph isomorphic to $G$? 2.\ Given a countable graph H that has a(n induced) subgraph isomorphic to $G$, find such a subgraph. The frameworks for our investigations are given by effective Wadge reducibility and by Weihrauch reducibility. Our work follows on "Reverse mathematics and Weihrauch analysis motivated by finite complexity theory" (Computability, 2021) by BeMent, Hirst and Wallace, and we answer several of their open questions. 
\end{abstract}

\tableofcontents
\section{Introduction}
In this paper, we study the uniform computational content of the \emph{subgraph problem} and the \emph{induced subgraph problem} via effective Wadge reducibility and via Weihrauch reducibility. The (induced) subgraph problem is parameterized by a countable graph $G$. The input is a countable graph $H$, and the task to decide is whether $G$ is isomorphic to a(n induced) subgraph of $H$ (the decision problem), or to find a(n induced) subgraph of $H$ which is isomorphic to $G$. We classify how the structure of $G$ impacts the difficulty of these computational tasks.

The line of research was initiated by BeMent, Hirst and Wallace \cite{bement2021reverse}, and they showed that the induced subgraph decision problem is $\Sigma^0_1$ for finite graphs, and exhibited some infinite graphs with a $\Sigma^1_1$-complete induced subgraph decision problem. We follow up on this by providing a much more complete picture.

The main contribution of this work is to explore the complexity of countable structures (of which graphs are the most prominent example) from a new angle. We consider the graphs to be genuine infinite objects, rather than restricting our attention to computable graphs coded suitably. A curious phenomenon here is that for an infinite graph, the induced subgraph problem is always $\Sigma^1_1$-complete - however, our proofs proceed in two overlapping cases.

In addition, we also offer some technical contributions to the field of Weihrauch complexity. We introduce the \emph{finite part of problem} as a new tool to investigate particular Weihrauch degrees and relate them to known ones. The Weihrauch degree of finding a subgraph isomorphic to an infinite ray in a given graph turns out to be incomparable to many hithertho investigated Weihrauch degrees, and we prove this via its finite part.

The paper is organized as follows. In \S \ref{sec:background} we give the necessary preliminaries on graph-related concepts, effective Wadge reducibility, computable analysis, and Weihrauch reducibility. Here we also introduce a new concept namely the \emph{finite part of a problem}, which is the strongest problem with finite codomain it can compute. This notion is related to the first-order and deterministic part of a problem introduced respectively by Dzhafarov, Solomon, and Yokoyama \cite{dzafarovsolomonyokoyama} (see also \cite{valentisolda}), and by Goh, the second author and Valenti in \cite{pauly-valenti}. In \S \ref{wadgecomplexityofgraphs}  we give results on the $\Gamma$-completeness, with $\Gamma$ being a lightface class, of sets of (names of) graphs of the form
\begin{equation}
\setofgraphs{G}{\repspaceallgraphs}{\allgraphrelations} \defas \{p \in \dom(\repmap{(E)Gr}) : G \allgraphrelations \repmap{(E)Gr}(p)\}.
\end{equation}
where $G$ is a fixed graph. The results about this section are summarized in Table \ref{SummaryGraph1}, the precise definitions and notations involved are given in due time  (same for the next tables and figures). 

\begin{table}[H]
	\centering
	\begin{tabular}{ |l|l|l| }\hline
	 $G$ finite & $\setofgraphs{G}{\repspaceallgraphs}{\subgraph}$  &  $\Sigma_1^0$-complete \\ \hline
	 $G$ finite & $\setofgraphs{G}{\repspacegraphs}{\inducedsubgraph}$ &  $\Sigma_1^0$-complete \\ \hline

	 $\complete{n}$ & $\setofgraphs{G}{\repspaceegraphs}{\inducedsubgraph}$ &  $\Sigma_1^0$-complete \\  \hline
	 $G$ finite and $G\not\cong \complete{n}$ & $\setofgraphs{G}{\repspaceegraphs}{\inducedsubgraph}$ &  $\Sigma_2^0$-complete \\  \hline

	$G$ c.e.\ and $\complete{\omega}\not\inducedsubgraph G$ & $\setofgraphs{G}{\repspaceallgraphs}{\inducedsubgraph}$ &  $\Sigma_1^1$-complete \\  \hline

	$G$ c.e.\ and $\R\subgraph G$ & $\setofgraphs{G}{\repspaceallgraphs}{\allgraphrelations}$ &  $\Sigma_1^1$-complete \\ \hline

	$\mathsf{T}_{2k+1}$ $(\mathsf{F}_{2k+2})$ & $\setofgraphs{\mathsf{T}_{2k+1} (\mathsf{F}_{2k+2})}{\repspaceallgraphs}{\subgraph}$&  $\Sigma_{2k+1}^0$-complete ($\Pi_{2k+2}^0$-complete)\\ \hline

	\end{tabular}
	\caption{}
	\label{SummaryGraph1}
	\end{table}
In \S \ref{DecisionProblemsGraphs} we turn our attention to Weihrauch reducibility and decision problems: in particular we solve a question left open in \cite[\S 2]{bement2021reverse}. We also seize the opportunity to discuss the interplay between Weihrauch reducibility and reverse mathematics, exhibiting graph-theoretic representatives of $\PiCA$ in the Weihrauch lattice. Figure \ref{SummaryAtThebeginningGraphs} summarizes the results of \S \ref{DecisionProblemsGraphs}.

\begin{figure}[H]
	\centering
	\tikzstyle{every picture}=[tikzfig]
	\begin{tikzpicture}[scale=0.8]

	  \begin{pgfonlayer}{nodelayer}
	  \node [style=box] (lpo) at (0,-1) {$\lpo \sweiequiv \isg_F \sweiequiv  e\isg_{\complete{n}} \sweiequiv  (e)\sg_F$};
	  \node [style=box] (lpojump) at (0,2) {$\lpo' \sweiequiv  e\isg_F  \sweiequiv (e)\sg_{\mathsf{F}_2}$ where  $F\not\cong\complete{n}$ };
	  \node [] (empty) at (0,4) {$\dots$};
  
	  \node [style=box] (lpojumpn) at (0,6) {$\jump{\lpo}{n} \sweiequiv (e)\sg_{\mathsf{G}} $ };
  \node [] (empty2) at (0,8) {$\dots$};
  \node [style=box] (wf) at (0,10) {$\wf\sweiequiv (e)\isg_G \sweiequiv (e)\sg_R$};
  
	  \end{pgfonlayer}
	  \begin{pgfonlayer}{edgelayer}
		  \draw [style=strictreducible] (lpo) to (lpojump);
		  \draw [style=strictreducible] (lpojump) to (empty);
		  \draw [style=strictreducible] (empty) to (lpojumpn);
		  \draw [style=strictreducible] (lpojumpn) to (empty2);
		  \draw [style=strictreducible] (empty2) to (wf);

	  \end{pgfonlayer}
  \end{tikzpicture}

				\caption{Some of the multi-valued functions studied in this paper. Black arrows represent Weihrauch reducibility in the direction of the arrow. Here $F$ represent a finite graph, $G$ an infinite c.e.\ graph, $R$ a c.e.\ graph such that $\R\subgraph R$ and  $\mathsf{G} \in \{ \mathsf{T}_{2k+1}, \mathsf{F}_{2k+2}\}$.}
	\label{SummaryAtThebeginningGraphs}
	\end{figure}
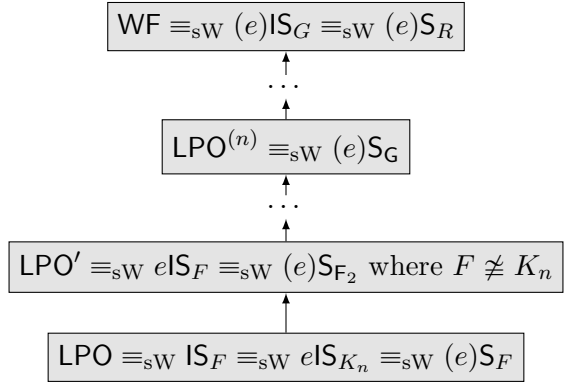
	\vspace{-0.5cm}

	In \S \ref{findsubgraphsandinduced} we introduce \quot{search problems}\ that, fixed a graph $G$ and given in input a graph $H$ such that $G \allgraphrelations H$ output an (induced) subgraph of $H$ that is isomorphic to $G$. In particular, we show that the situation for the induced subgraph relation is more \quot{tidy} (i.e.\ the problems we consider are all Weihrauch equivalent to $\CBaire$, sometimes relatively to some oracle, see \thref{maintheoreminducedsubgraph}), while the situation for the subgraph relation is more intricate. Indeed, we have different infinite graphs such that the corresponding problems for the subgraph relation are Weihrauch equivalent to $\CBaire$ (see \thref{maintheoermsubgraphcase,infinitelycyclescbaire}), others that are computable (see \thref{computablefinitecomponents}) and others that are Weihrauch equivalent to (jumps of) $\mflim$ (see \thref{maintheorem_findubgraphgn}). 

	Of particular interest is the case when $G$ is $\R$: we show that, restricting the domain of the corresponding problem to connected graph, we obtain examples of natural problems having computable finitary part but noncomputable first-order part. Furthermore, even without the domain restriction such problems have the peculiar property of being hard to compute, but weak when they have to compute a problem on their own (see Figure \ref{figurefindray}).

		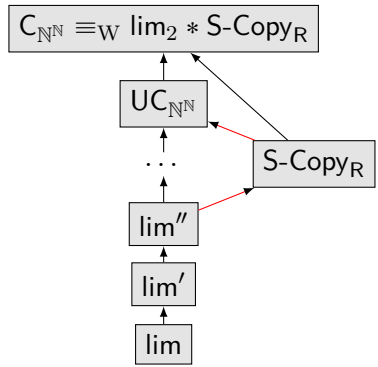
\begin{figure}[H]
			\centering
			\tikzstyle{every picture}=[tikzfig]
			\begin{tikzpicture}[scale=0.8]
		
			  \begin{pgfonlayer}{nodelayer}
			  \node [style=box] (mflim) at (2,0) {$\mflim$};
			  \node [style=box] (mflimjump) at (2,2) {$\mflim'$ };
			  \node [style=box] (mflimjump2) at (2,4) {$\mflim''$ };
		
			  \node [] (empty) at (2,6) {$\dots$};
		  
			  \node [style=box] (ucbaire) at (2,8) {$\UCBaire$ };
		  \node [style=box] (cbaire) at (2,10.5) {$\CBaire\weiequiv \mflim_2*\finds{}{\R}$};
		  \node [style=box] (finds) at (7,6) {$\finds{}{\R}$};

			  \end{pgfonlayer}
			  \begin{pgfonlayer}{edgelayer}
				  \draw [style=strictreducible] (mflim) to (mflimjump);
				  \draw [style=strictreducible] (mflimjump) to (mflimjump2);
				  \draw [style=strictreducible] (mflimjump2) to (empty);
				  \draw [style=strictreducible] (empty) to (ucbaire);
				  \draw [style=strictreducible] (ucbaire) to (cbaire);
		
				  \draw [style=strictreducible] (finds) to (cbaire);
				  \draw [style=nonreducible] (mflimjump2) to (finds);
				  \draw [style=nonreducible] (finds) to (ucbaire);
				  
			  \end{pgfonlayer}
		  \end{tikzpicture}
		  \caption{A summary of the Weihrauch reductions between problems considered in \S \ref{particularcase}. As in Figure \ref{SummaryAtThebeginningGraphs} black arrows represent Weihrauch reducibility in the direction of the arrow while red arrows represent the absence of a Weihrauch reduction in the direction of the arrow.}
		  \label{figurefindray}
		\end{figure}

\section{Background}
\label{sec:background}

 \subsection{Strings, graphs and trees}
\label{sectiongraphs}

Let $\nats^n$ denote the set of finite sequences of natural numbers of length $n$, where the length is denoted by $\length{\cdot}$ (we use the same symbol to denote the cardinality of a set). If $n=0$, $\nats^0=\{\str{}\}$, where $\str{}$ is the empty sequence: in general, given $i_0,\dots,i_{n-1} \in \nats$, we denote by $\str{i_0,\dots,i_{n-1}}$ the string in $\nats^n$ having digits $i_0,\dots,i_{n-1}$. The symbol $\str{\cdot}$ is also used to denote a fixed pairing function from $\baire$ to $\nats$, but there should be no confusion when we are referring to strings or their codes. The set of all finite sequences of natural numbers is denoted by $\baire$. For $\sigma \in \baire$ and $m\leq\length{\sigma}$ let $\sigma[m]=\str{\sigma(0),\dots,\sigma(m-1)}$. Given $\sigma, \tau \in \baire$ we use $\sigma \sqsubseteq \tau$ to say that $\sigma$ is an \emph{initial segment} of $\tau$ (equivalently, $\tau$ an \emph{extension} of $\sigma$), i.e. $\sigma=\tau[m]$ for some $m \leq \length{\sigma}$. We use the symbol $\sqsubset$ in case $\sigma \sqsubseteq \tau$ and $\length{\sigma}<\length{\tau}$, and in case $\sigma \not\sqsubseteq \tau$ and $\tau \not\sqsubseteq \sigma$ we say that $\sigma$ and $\tau$ are \emph{incomparable} ($\sigma \incomparable \tau$). The concatenation of two strings $\sigma,\tau$ is denoted by $\sigma\concat\tau$: whenever it is clear from the context we just write $\sigma\tau$. For $n,i \in \nats$, we denote by $n^i$ the sequence made of $i$ many $n$’s: if in case $i = 1$ we just write $n$ and we use $n^\nats$ to denote the infinite sequence with constant value $n$. 

A \emph{tree} $T$ is a subset of $\baire$ that is closed under initial segments. In case the tree $T$ is a subset of $\cantor$, we  call $T$ a binary tree. 
 Given a tree $T$, we say that $f \in \Baire$ is a path through $T$ if for all $n \in \nats$, $f[n]\in T$ where, as for finite strings, $f[n]=\str{f(0),\dots,f(n-1)}$. We denote by $\body{T}$ the \emph{body} of $T$, that is the set of paths in $T$. We say that a tree $T$ is \textit{ill-founded} iff there exists at least one path in $\body{T}$, \textit{well-founded} otherwise.

 A graph $G$ is a pair $(V, E)$ where $V$ is the set of \emph{vertices} and $E$ is a binary relation on $V \times V$; a pair $(v,w) \in E$ is called an \emph{edge}. In this paper, when we say graph, we always assume that is countable, undirected, and without self-loops: that is, $V\subseteq \nats$ and $E$ satisfies anti-reflexivity and symmetry. 
  Given a graph $G$ we denote the set of vertices and the set of edges respectively as $V(G)$ and $E(G)$.  Given graphs $G$ and $H$, we denote that they are isomorphic with $G\cong H$, and we often say that \lq\lq $G$ is a copy of $H$\rq\rq\ or vice versa. 
  We say that a graph $G$ is \emph{finite}  if $V(G)$ is finite, infinite otherwise. Given $v \in V(G)$, let $\{w \in V(G):(v,w)\in E(G)\}$ the set of \emph{neighbors} of $v$ in $G$, and define the \emph{degree} of $v$ in $G$ as $\D{G}{v}\defas \length{\{w:(v,w)\in E(G)\}}$. 
 
 For a graph $G$ and $n > 0$, a \emph{line segment of length $n$} is a sequence of distinct vertices $v_0,\dots,v_{n}$ of $G$ such that for every $i<n$, $(v_i,v_{i+1}) \in E(G)$ and we call $v_0$ and $v_n$ \emph{endpoints} of the line segment. Given $u,v,w \in V(G)$, we say that \emph{$v$ and $u$ are connected  in $G$ through $w$},  denoted by $\pathconnected{v}{u}{w}{G}$, if there exists a line segment of finite length in $G$, with endpoints $v$ and $u$ containing $w$: instead, the predicate $\pathconnected{v}{u}{\lnot w}{G}$ states that $v$ and $u$ are connected but no line segment of finite length with endpoints $u$ and $v$ contains $w$. If we simply want to say that $v$ and $u$ are connected in $G$, we drop the subscript $w$ (or $\lnot w$). Notice that $\pathconnected{v}{u}{}{G}$ by a line segment of length $1$ is equivalent to writing $(v,w) \in E(G)$. We say that a graph $G$ is \emph{connected}  if $(\forall v,u \in V(G))(\pathconnected{v}{u}{}{G})$.
 
We now define three particular types of graphs, namely \emph{line segments}, \emph{cycles}, and \emph{complete} graphs. We define $\ray{n}$, $\cycle{n}$ and $\complete{n}$ as graphs having the same vertex set $\{i:i < n\}$, where $n >0$ in $\ray{n}$ and $\complete{n}$ and $n\geq 3$ in $\cycle{n}$. The edge sets are respectively: $E(\ray{n})=\{(i,i+1):i< n\}$, $E(\complete{n})=\{(i,j): i \neq j \land i, j < n\}$ and $E(\cycle{n})=E(\ray{n-1}) \cup \{(n-1,0)\}$. It is immediate that $\ray{1}=\complete{1}$ and $\cycle{3}= \complete{3}$.  Notice that $\ray{n}$ and $\complete{n}$  generalize to the infinite case: that is, $\R$ and $\complete{\omega}$ are the graphs having as vertex set $\nats$ and as edge set respectively $\{(i,i+1):i \in \nats\}$ and $\{(i,j): i\neq j \land i,j \in \nats\}$. Another infinite generalization of $\ray{n}$ we use is the \quot{two-way infinite ray}\ $\elle$,  where $V(\elle)=\nats$ and $E(\elle)=\{(0,1)\}\cup \{(2i,2i+2):i \in \nats\} \cup \{(2i+1,2i+3):i \in \nats\}$.

  Given countably many graphs $\{G_i : i \in \nats\}$ we define the \emph{disconnected union}  $\disconnectedunion{i \in \nats}{G_i}$ as follows:
  \[V(\disconnectedunion{i \in \nats}{G_i})\defas \bigcup_{i \in \nats} \{\pair{i,v} : v \in V(G_i)\}\text{ and }E(\disconnectedunion{i \in \nats}{G_i}) \defas \{(\pair{i,v},\pair{i,w}):(v,w) \in  E(G_i)\}.\]

 In case the disconnected union involves at most three graphs $G_0$, $G_1$, $G_2$, we write $G_0 \disconnectedunion{} G_1$ and $G_0 \disconnectedunion{} G_1 \disconnectedunion{} G_2$ respectively. We denote by $\infinitelycopies{G}$ the disconnected union of countably many copies of $G$.
 
 Another operation on graphs we use is the  \emph{connected union} (denoted by $\connectedunion{i \in \nats} G_i$)  in which, intuitively, for every $i$, $G_i$ and $G_{i+1}$ share a unique common vertex, different from the one shared between $G_{i+1}$ and $G_{i+2}$. Formally, given countably many graphs $\{G_i: i \in \nats \}$ (for simplicity assume $\length{V(G_i)} \geq 3$ for every $i$), let $\mathsf{v_i} \defas \min \{v: v \in V(G_i)\}$ and, if $G_i$ is finite, let  $\mathsf{w_i}\defas \max\{v: v \in V(G)\}$, otherwise let $\mathsf{w_i}\defas \min\{v: v \in V(G_i)\setminus \{\mathsf{v}_i\}\}$. Then let,
 
 $$V(\connectedunion{i \in \nats} G_i) \defas V(\disconnectedunion{i \in  \nats} G_i) \
 \setminus \ \big\{ \{ \pair{i, \mathsf{v}_i}: i>0\} \cup \{ \pair{i, \mathsf{w}_i}: i \in \nats  \}\big\}  \bigcup \big \{\pair{\mathsf{w}_i,\mathsf{v}_{i+1}} : i \in \nats \big\}, $$
 $$E(\connectedunion{i \in \nats} G_i) \defas \{(\pair{i,v},\pair{i,w}):v,w \notin \{\mathsf{v}_i, \mathsf{w}_i\} \land (v,w) \in  E(G_i)\} \bigcup$$ $$\big\{ \big( \pair{\mathsf{w}_i, \mathsf{v}_{i+1}}, \pair{i,u} \big): (\mathsf{w}_i, u) \in E(G_i) \lor (\mathsf{v}_{i+1}, u) \in  E(G_{i+1}) \big\}.$$

 \begin{figure}[H]
	 \begin{subfigure}{.5\textwidth}
	   \centering
	   \begin{tikzpicture}[scale=2.0, main/.style = {draw, circle, fill=black,scale=0.3}] 
		 \node[main] (triangle1) {}; 
		 \node[main] (triangle2) [below of=triangle1, yshift=-1.5cm] {}; 
		 \node[main] (triangle3) [right of=triangle1,yshift=-1.2cm, xshift= 1.5cm] {}; 
		 
		 \node[main] (square1) [right of=triangle3,xshift=1.5cm] {}; 
		 \node[main] (square2) [right of= triangle1, xshift=5.5cm] {}; 
		  \node[main] (square3) [below of= square2, yshift=-1.5cm] {}; 
		  \node[main] (square4) [right of= square1, xshift=2cm] {}; 
	 
		  \node[main] (pent1) [right of=square4,xshift=1.5cm,yshift=0.1cm] {}; 
		  \node[main] (pent2) [right of=square2,  xshift=4.3cm] {}; 
		  \node[main] (pent3) [below of=pent2,xshift=-1cm, yshift=-1.5cm] {}; 
		  \node[main] (pent4) [right of=pent1,xshift=1.5cm] {}; 
		  \node[main] (pent5) [right of= pent3,xshift=0.9cm] {}; 
	 
		  \node[] (dots) [right of= pent1,xshift=0.5cm] {$\dots$}; 
 
		 \draw (triangle1) -- (triangle2);
		 \draw (triangle2) -- (triangle3);
		 \draw (triangle1) -- (triangle3);
	 
		 \draw (square1) -- (square2);
		 \draw (square1) -- (square3);
		 \draw (square2) -- (square4);
		 \draw (square3) -- (square4);
	 
		 \draw (pent1) -- (pent2);
		 \draw (pent1) -- (pent3);
		 \draw (pent2) -- (pent4);
		 \draw (pent3) -- (pent5);
		 \draw (pent4) -- (pent5);
		 \end{tikzpicture} 
	 \end{subfigure}
	 \begin{subfigure}{.5\textwidth}
	   \centering
	   \begin{tikzpicture}[scale=2.0, main/.style = {draw, circle, fill=black,scale=0.3}] 
		 \node[main] (triangle1) {}; 
		 \node[main] (triangle2) [below of=triangle1, yshift=-1.5cm] {}; 
		 \node[draw,red, circle, fill=red,scale=0.3] (triangle3) [right of=triangle1,yshift=-1.2cm, xshift= 1.5cm] {}; 
		 
		 \node[draw,red, circle, fill=red,scale=0.3] (square1) [right of=triangle3,xshift=-0.9cm] {}; 
		 \node[main] (square2) [right of= triangle1, xshift=3.2cm] {}; 
		  \node[main] (square3) [below of= square2, yshift=-1.5cm] {}; 
		  \node[draw, red, circle, fill=red,scale=0.3] (square4) [right of= square1, xshift=2.3cm] {}; 
	 
		  \node[draw,red, circle, fill=red,scale=0.3] (pent1) [right of=square4,xshift=-0.9cm] {}; 
		  \node[main] (pent2) [right of=square2,  xshift=2cm] {}; 
		  \node[main] (pent3) [below of=pent2,xshift=-0.9cm, yshift=-1.5cm] {}; 
		  \node[main] (pent4) [right of=pent1,xshift=1.5cm] {}; 
		  \node[main] (pent5) [right of= pent3,xshift=0.9cm] {}; 
		  \node[] (dots) [right of= pent1,xshift=0.5cm] {$\dots$}; 
 
		 \draw (triangle1) -- (triangle2);
		 \draw (triangle2) -- (triangle3);
		 \draw (triangle1) -- (triangle3);
	 
		 \draw (square1) -- (square2);
		 \draw (square1) -- (square3);
		 \draw (square2) -- (square4);
		 \draw (square3) -- (square4);
	 
		 \draw (pent1) -- (pent2);
		 \draw (pent1) -- (pent3);
		 \draw (pent2) -- (pent4);
		 \draw (pent3) -- (pent5);
		 \draw (pent4) -- (pent5);
		 \end{tikzpicture} 
	 \end{subfigure}
	 \caption{On the left side, the disconnected union $\disconnectedunion{i>2}\cycle{i}$ of all cyclic graphs, shown up to $C_5$. On the right side, the connected union $\connectedunion{i>2}\cycle{i}$ of all cyclic graphs, shown up to $C_5$: starting from the left one, the two red vertices denote respectively the vertices $\str{\mathsf{w}_3,\mathsf{v}_4}$ and $\str{\mathsf{w}_4,\mathsf{v}_5}$.}
	 \label{fig:fig}
	 \end{figure}
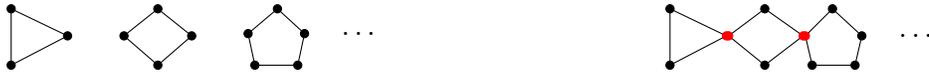
 
 As for the disconnected union, we write $G_0 \connectedunion{} G_1$ and $G_0 \connectedunion{} G_1 \connectedunion{} G_2$ in case the connected union is defined only on two or three graphs. Notice that $\elle \cong \R \connectedunion{} \R$.
 
We have already defined what a tree is: it is easy to notice that a tree is a particular type of graph, and hence when considering a tree, adjusting some detail, we can choose at our convenience if we want to refer to it as a set of finite sequences or as a \emph{graph theoretic tree}. A graph-theoretic tree is a connected graph in which any two vertices are connected by exactly one path: in other words, they are connected graphs that do not contain any cycle. The  graph theoretic tree is \emph{$v$-rooted} if there is a distinguished vertex $v \in V(G)$, namely the \emph{root} of $T$. So, as anticipated, given a  tree $T \in \tree$ we can translate it as a $v$-rooted graph theoretic tree $G$ where  $v=\str{}$, $V(G)=T$ and $E(G)=\{(\sigma,\tau):\sigma\tau \in T\}$. Conversely, we can translate a $v_0$-rooted graph theoretic tree $G$ into a tree $T$, identifying any $v\in V(G)$ with the sequence $\str{v_0,\dots,v_n}$, where $v_0,\dots,v_n$ are the vertices of the unique path from $v_0$ to $v_n$ and $v_n=v$. Notice that both translations are computable (in case we have to translate $T$ into a $v$-rooted tree just let $v=\str{}$). In case we drop the assumption that the graph-theoretic tree is connected, we obtain a \emph{forest}: in other words, a forest, is the disconnected union of countably many graph-theoretic trees.
 
 We give the definition of \emph{subgraph} and the \emph{induced subgraph} relations.
 \begin{definition}
	 \thlabel{definition_induced_subgraph}
	 Given two graphs $G$ and $H$ we say that:
	 \begin{itemize}
		 \item $H$ is a \emph{subgraph} of $G$ if $V(H) \subseteq V(G)$ and $E(H) \subseteq E(G)$;
		 \item  $H$ is an \emph{induced subgraph} of $G$ if $H$ is a subgraph of $G$ and $E(H)= E(G) \cap (V(H) \times V(H))$.
	 \end{itemize}
	 \end{definition}
	 Given graphs $G$ and $H$, we use the following abbreviations
 \begin{itemize}
		 \item $G\subgraph H:\iff(\exists G'\subseteq H)(G'\cong G \land G'\text{ is a subgraph of } H)$;
	 \item $G\inducedsubgraph H:\iff(\exists G'\subseteq H)( G' \cong G \land G'\text{ is an induced subgraph of } H)$.
 \end{itemize}
	 Finally, given a graph $G$ and $V\subseteq V(G)$, we define the \emph{graph induced by $V$ on $G$}, denoted by $G_{\restriction V }$ as the graph having $V(G_{\restriction V})\defas V$ and $E(G_{\restriction V})\defas E(G) \cap (V \times V)$.

\subsection{(Effective) Wadge reducibility}
We assume the reader to be familiar with the Borel and projective hierarchy and their effective counterparts, namely the Kleene and analytical hierarchy: these notions can be found for example in \cite{kechris2012classical,moschovakis}. It is customary to refer to Borel and projective classes as \emph{boldface classes} and their effective counterparts as \emph{lightface classes}.

We give some basic definitions on \emph{Wadge reducibility}. The latter, provides a notion of complexity for sets in topological spaces. Given two topological spaces $X,Y$ and given $A\subseteq X$, $B\subseteq Y$, we say that $A$ is \emph{Wadge reducible} to $B$ (in symbols, $ A \wadgereducible B$) if there is a continuous function $\function{f}{X}{Y}$ such that $x \in A \iff f(x) \in B$. Informally, if a set $A$ is Wadge reducible to a set $B$, it means that $A$ is simpler than $B$. We mention that the $\wadgereducible$ is a quasi-order and the corresponding equivalence classes are called the \emph{Wadge degrees}. Let $\boldfaceGamma$ be a boldface class and $X$ and $Y$ be Polish space with $X$ being zero-dimensional. We call $B \subseteq Y$ $\boldfaceGamma$-\emph{hard} if, for any $A\in \boldfaceGamma(X)$, $A \wadgereducible B$: if in addition $B  \in \boldfaceGamma(Y)$, we say that $B$ is $\boldfaceGamma$-\emph{complete}. Notice that, if $B$ is $\boldfaceGamma$-hard, then the complement of $B$ is $\check{\boldfaceGamma}$-hard. Furthermore, if $B$ is $\boldfaceGamma$-hard and $B \wadgereducible A$, then $A$ is $\boldfaceGamma$-hard as well. All these considerations still hold if we replace hardness with completeness. This gives us a useful technique to show that a set is $\boldfaceGamma$-hard (respectively, $\boldfaceGamma$-complete): \lq\lq take an already known $\boldfaceGamma$-hard ($\boldfaceGamma$-complete) set $A$ and show that $A\wadgereducible B$\rq\rq. 
The effective counterpart of Wadge reducibility (which we call \emph{effective Wadge reducibility}) is defined between subsets of \emph{effective second-countable spaces}, but for the aim of this paper, it suffices to define it on subsets of $\Baire$. So, given $A,B \subseteq \Baire$ we say that $A$ is effectively Wadge reducible to $B$ (in symbols, $A \effectivewadge B$) if there is a computable $\function{f}{\Baire}{\Baire}$ such that $x \in A \iff f(x) \in B$. For a lightface class $\Gamma$, we say that $B\subseteq \Baire$ is $\Gamma$-hard if, for every $A \in \Gamma(\Baire)$, $A \effectivewadge B$. In case $B$ is $\Gamma$-hard and $B \in \Gamma$, then we say that $B$ is $\Gamma$-complete.

We now give some examples of complete sets that are useful in the rest of the paper. All the results in the remaining part of this subsection are stated for lightface classes and with respect to effective Wadge reducibility: on the other hand, all the results still hold replacing \quot{lightface}\ with \quot{boldface}\  and \quot{effective Wadge}\ with \quot{Wadge}. These results can be found, sometimes with different terminology, for example in \cite{moschovakis,kechris2012classical}.

We start defining complete sets of (products of) $\Cantor$ for all levels of the Kleene arithmetical hierarchy. The proof of the next theorem is omitted and exploits ideas by Solecki (see \cite[\S Notes and Hints 23.5(i)]{kechris2012classical}).
\begin{theorem}
	\thlabel{firstcompletesets}
	Let $k>0$. Then,
	\[P_{2k+1} \defas \{ p \in 2^{\nats^{k+1}} :(\exists n_0)(\exists^\infty n_1)\dots(\exists^\infty n_k)(p(n_0,\dots,n_k)=1)\} \text{ is } \Sigma_{2k+1}^0\text{-complete} \]
	\[P_{2k+2} \defas \{ p \in 2^{\nats^{k+1}} : (\exists^\infty n_0)(\exists^\infty n_1)\dots(\exists^\infty n_k)(p(n_0,\dots,n_k) =1)\} \text{ is } \Pi_{2k+2}^0\text{-complete}. \]
\end{theorem}

The following lemma, whose proof is omitted, comes in handy in the next sections of the paper.

		\begin{lemma}
		\thlabel{Aomega}
		Let $n>1$ and let $A$ be a $\Pi_{n}^0$-complete set of $\Baire$. Then, 
		\begin{itemize}
			\item $A_\infty^0 \defas \{(x_i)_{i \in \nats} \in {(\Baire)}^\nats : (\exists i)(x_i \in A)\}$ is $\Sigma_{n+1}^0$-complete
			\item $A_\infty^1 \defas \{(x_i)_{i \in \nats} \in {(\Baire)}^\nats : (\exists^\infty i)(x_i \in A)\}$ is $\Pi_{n+2}^0$-complete.
		\end{itemize}
		\end{lemma}

We summarize some results about the complexity of subsets of trees in the next theorem.  The fact that $\illfounded$ is $\Sigma_1^1$-complete follows from \cite[Theorem 27.1]{kechris2012classical}: the theorem is stated for the boldface case, but its proof works also in the lightface one. If $T\subseteq \cantor$, notice that, by K\"{o}nig's lemma, $T \in \illfounded_2$ if and only if $(\forall n)(\exists \tau \in 2^{n})(\tau \in T)$.

\begin{theorem}
	\thlabel{Complexityresults}
	The set $\illfounded\defas \{T \in \tree : T \text{ is ill-founded}\}$ is $\Sigma_1^1$-complete, while $\wellfounded\defas \{T \in \tree : T \text{ is well-founded}\}$ is $\Pi_1^1$-complete. In contrast, $\illfounded_2\defas \illfounded \cap \tree_2$ is $\Pi_1^0$-complete and $\wellfounded_2\defas \wellfounded \cap \tree_2$ is $\Sigma_1^0$-complete.
\end{theorem}

\subsection{Weihrauch reducibility}
\label{subsec:Weihrauch}
 This subsection introduces some concepts from computable analysis and the theory of represented spaces. The reader is referred to \cite{Weihrauch,brattka2021weihrauch} for more on these topics. Here computability of functions from $\Baire$ to $\Baire$ is intended as in TTE. A \emph{represented space} $\repspacex$ is a pair $(X,\repmap{X})$ where $X$ is a set and $\partialfunction{\repmap{X}}{\Baire}{X}$ is a (possibly partial) surjection. We say that $p$ is a $\repmap{X}$-\emph{name} for $x$ if $\repmap{X}(p)=x$. A computational problem $f$ between represented spaces $\textbf{X}$ and $\textbf{Y}$ is formalized as a \emph{partial multi-valued function} $\partialmultifunction{f}{\textbf{X}}{\textbf{Y}}$. A  \emph{realizer} for $\partialmultifunction{f}{\textbf{X}}{\textbf{Y}}$ is a (possibly partial) function $\partialfunction{F}{\Baire}{\Baire}$ such that for every $p \in \dom(f\circ \delta_X)$, $\delta_Y(F(p)) \in f(\delta_X(p))$. Realizers allow us to transfer properties of functions on the Baire space (such as computability or continuity) to multi-valued functions  on represented spaces in general. Hence, from now on, whenever we say that a multi-valued function between represented spaces is computable we mean that it has a computable realizer.

Recall that a  \emph{computable metric space} $\X=(X,d,\alpha)$, is a separable metric space $(X,d)$ and a dense sequence $\function{\alpha}{\nats}{X}$ such that $\function{d\circ (\alpha \times \alpha)}{\nats^2}{\mathbb{R}}$ is a computable double sequence of real numbers. For such spaces and $k>0$ we can use the inductive definition of Borel sets, and define the represented spaces $\boldfaceSigma_k^0(\X)$, $\boldfacePi_k^0(\X)$ and $\boldfaceDelta_k^0(\X)$ and this allows us to consider Borel classes as represented spaces (see \cite{effectiveborelmeasurability}).

We denote by $\tree$ and $\tree_2$ the represented spaces of trees on $\nats$ numbers and binary trees respectively: for both the representation map is the characteristic function and hence we identify $\tree$ and $\tree_2$ with closed subsets of $\Cantor$. That is, the negative representation of a closed subset $C$ of $\Baire$ (resp.\ $\Cantor$) is equivalent (in the sense of \cite[Definition 2.3.2]{Weihrauch}) to the one given by the characteristic function of a (resp.\ binary) tree $T$ such that $\body{T}=C$.  We refer to the latter representation as the \emph{tree representation}.

To compare the uniform computational content of different problems, we use the framework of \emph{Weihrauch reducibility}.  We say that $f$ is \emph{Weihrauch reducible} (respectively, \emph{strongly Weihrauch reducible}) to  $g$, and we write $f\weireducible g$ ($f\strongweireducible g$) if there are computable maps $\partialfunction{\Phi,\Psi}{\Baire}{\Baire}$ such that
    \begin{itemize}
        \item for every name $p_x$ for some $x \in \dom(f)$, $\Phi(p_x)=p_z$, where $p_z$ is a name for some $z \in \dom(g)$  and,
        \item for every name $p_w$ for some $w \in g(z)$, $\Psi(p_x\join p_w)=p_y$ (or just $\Psi(p_w)=p_y$ in case it is a strong Weihrauch reduction) where $p_y$ is a name for $y \in f(x)$.
    \end{itemize}
Informally, if $f\weireducible g$ we are claiming the existence of a procedure for solving $f$ which is computable modulo a single invocation of $g$ as an oracle (in other words, this procedure transforms realizers for $g$ into realizers for $f$). In case $\Phi$ is as above and $\Psi$ is not allowed to use $p$ in its computation, we say that $f$ is \emph{strongly Weihrauch reducible} to a problem $g$, written $f \sweireducible g$.

Weihrauch reducibility and strong Weihrauch reducibility are reflexive and transitive hence they induce the equivalence relations
$\weiequiv$ and $\sweiequiv$: that is $f\weiequiv g$ iff $f \weireducible g$ and $g\weireducible f$ (similarly for $\sweireducible$).
The $\weiequiv$-equivalence classes are called \emph{Weihrauch degrees} (similarly the $\sweiequiv$-equivalence classes are called  \emph{strong Weihrauch degrees}). Both the Weihrauch degrees and the strong Weihrauch degrees form lattices (see \cite[Theorem\ 3.9\ and\ Theorem\ 3.10]{brattka2021weihrauch}).

There are several natural operations on problems which also lift to the $\weiequiv$-degrees and the $\sweiequiv$-degrees: we mention below the ones we need.
 
\begin{itemize}
    \item The \emph{parallel product} $f \times g$ is defined by $(f \times g)(x,y) \defas f(x) \times g(y)$. 
    \item The \emph{finite parallelization} is defined as $f^*((x_i)_{i<n})\defas\{(y_i)_{i<n}:(\forall i<n)(y_i \in f(x_i))\}$.
    \item The \emph{infinite parallelization} is defined as $\parallelization{f}((x_i)_{i \in \nats})\defas\{(y_i)_{i \in \nats}:(\forall i)(y_i \in f(x_i))\}$.
    
	\item   Given a family of problems $\{f_i:i \in \nats\}$ with $\partialmultifunction{f_i}{\bf{X}_i}{\bf{Y}_i}$, the \emph{countable co-product} $\partialmultifunction{\bigsqcup_{i \in \nats }f_i}{\bigcup_{i \in \nats} X_i}{\bigcup_{i \in \nats} Y_i}$ where 
	\[ \dom(\bigsqcup_{i \in \nats} f_i) \defas \bigcup_{i \in \nats} \{i\} \times X_i \text{ and } \big(\bigsqcup_{i \in \nats} f_i\big)(i,x)\defas \{i\} \times f_i(x).\] 
    \end{itemize}
    Informally, the first three operators defined above, capture respectively the idea of using $f$ and $g$ in parallel, using $f$ a finite (but given in the input) number of times in parallel, and using $f$ countably many times in parallel. The last operator captures the idea of computing exactly one $f_i$.

	The following definition, with a slightly different notation, was recently given by Sold\`a and Valenti.
    \begin{definition}[\cite{valentisolda}]
    \thlabel{ustar}
    For every $\partialmultifunction{f}{\mathbf{X}}{\mathbf{Y}}$, define the finite unbounded parallelization $\partialmultifunction{\ustar{f}}{\nats \times \Baire \times \mathbf{X}}{(\Baire)^{<\nats}}$ as follows:
    \begin{itemize}
        \item instances are triples $(e,w,(x_n)_{n \in \nats})$ such that $(x_n)_{n \in \nats} \in \dom(\parallelization{f})$ and for each sequence $(q_n)_{n \in \nats}$ with $\repmap{Y}(q_n) \in f(x_n)$, there is a $k\in \nats$ such that $\Phi_e(w,q_0*\dots * q_{k-1})(0)\downarrow$ in $k$ steps;
        \item a solution for $(e,w,(x_n)_{n \in \nats})$ is a finite sequence $(q_n)_{n<k}$ such that for every $n <k$, $\repmap{Y}(q_n)\in f(x_n)$ and  $\Phi_e(w,q_0*\dots * q_{k-1})(0)\downarrow$ in $k$ steps.
    \end{itemize}
    \end{definition}
    Informally, $\ustar{f}$ takes an input a Turing functional with a parameter and an input for $\parallelization{f}$ and outputs \lq\lq sufficiently many\rq\rq\ names for solutions where \lq\lq sufficiently many\rq\rq\ is determined by the convergence of $\Phi_e$.
    
    We call $f$ a \emph{cylinder} if $f \sweiequiv f \times \id$. If $f$ is a cylinder, then $g \weireducible f$ iff $g \sweireducible f$ (\cite[Cor.\ 3.6]{BG09}). This is useful for establishing nonreductions because, if $f$ is a cylinder, then it suffices to diagonalize against all strong Weihrauch reductions from $g$ to $f$ in order to show that $g \not\weireducible f$. Cylinders are also useful when working with compositional products (discussed below). Observe that for every problem $f$, $f \times \id$ is a cylinder which is Weihrauch equivalent to $f$.

The \emph{compositional product} $f * g$ captures the idea of what can be achieved by first applying $g$, possibly followed by some computation, and then applying $f$. Formally, $f*g$ is any function satisfying
\[ f * g \weiequiv \max_{\weireducible} \{f_1 \circ g_1 \st f_1 \weireducible f \land g_1 \weireducible g\}. \]
This operator was first introduced in \cite{BolWei11}, and proven to be well-defined in \cite{BP16}. For each problem $f$, we denote by $f^{[n]}$ the $n$-fold iteration of the compositional product of $f$ with itself, i.e., $f^{[1]} = f$, $f^{[2]} = f * f$, and so on.

We say that a computational problem $f$ is \emph{first-order} if its codomain is $\nats$. As we need only the following characterization of $\firstOrderPart{f}$ we omit the technical definition (see e.g.\ \cite[Definition 2.2]{pauly-valenti}). 

\begin{theorem}[\cite{dzafarovsolomonyokoyama}]
	\thlabel{firstorderpartcharacterization}
	For every problem $f$, $\firstOrderPart{f} \weiequiv \max_{\weireducible}\{ g\st g \text{ is first-order and } g \weireducible f\}$.
\end{theorem}

We continue defining the \emph{jump} of a problem: notice that such an operator does not lift to Weihrauch degrees: that is, the jump applied to two Weihrauch equivalent problems may produce problems of different strength.

First, we define the jump of a represented space $\mathbf{X}=(X,\repmap{X})$. This is $\mathbf{X}'=(X,\repmap{X}')$ where $\repmap{X}'$ takes in input a sequence of elements of $\Baire$ converging to some $p \in \dom(\delta_X)$ and returns $\repmap{X}(p)$. Then for a problem $\partialmultifunction{f}{\textbf{X}}{\textbf{Y}}$ its jump $ \partialmultifunction{f'}{\textbf{X}'}{\textbf{Y}}$ is defined as $f'(x) \defas f(x)$. In other words, $f'$ is the following task: given a sequence that converges to a name for an instance of $f$, produces a solution for that instance. The jump lifts to  strong Weihrauch degrees but not to Weihrauch degrees (see \cite[\S 5]{BolWei11}). We use $f^{(n)}$ to denote the $n$-th iterate of the jump applied to $f$.

We proceed defining some well-known problems in the Weihrauch lattice.
\begin{itemize}
\item Let $ \partialfunction{\mflim}{(\Baire)^{\mathbb{N}}}{\Baire}, \ (p_n)_{n \in \mathbb{N}}\mapsto \lim p_n$ be the single-valued function  whose domain consists of all converging sequences in $\Baire$.
\end{itemize}
Notice that have $f^{(n)}\weireducible f* \mflim^{[n]}$, but the converse reduction does not hold in general. 
\begin{itemize}
    \item the function $\function{\lpo}{\Cantor}{\{0,1\}}$ is defined as $\lpo(p)=1$ iff $(\forall n)(p(n)=0)$. It is convenient to think of $\lpo$ as the function answering yes or no to questions which are $\Pi_1^{0}$ or $\Sigma_1^{0}$ in the input.
    \item  Similarly, $\lpo^{(n)}$ can be seen as the function answering yes or no to questions which are $\Pi_{n+1}^0$ or $\Sigma_{n+1}^0$ in the input.
\end{itemize}
It is well-known that $\mflim \sweiequiv \parallelization{\lpo}$. Moreover, using \cite[\S 5]{BolWei11}, we obtain that for every $n$, $\mflim^{(n)} \sweiequiv \parallelization{\lpo^{(n)}}$.

\begin{itemize}
    \item the function $\function{\wf}{\tree}{\{0,1\}}$ is defined as $\wf(T)=1$ iff $T \in \wellfounded$. Analogously to $\lpo$, we can think of $\wf$ as the problem answering yes or not to questions which are $\Pi_1^{1}$ or $\Sigma_1^{1}$ in the input. 
\end{itemize}

We now move our attention to \emph{choice problems}. For a computable metric space $\X$ and a boldface class $\boldfaceGamma$, we define:
\begin{itemize}
     \item $\partialmultifunction{\choice{\boldfaceGamma}{\X}}{\boldfaceGamma(\X)}{\X}$  as the problem that given in input a nonempty set $A \in \boldfaceGamma(\X)$ outputs a member of $A$ and 
     \item $\uchoice{\boldfaceGamma}{\X}$ as above but with domain restricted to singletons.
\end{itemize}
When $\boldfaceGamma=\boldfacePi_1^0$ we just write $\mathsf{C}_{\X}$ and  $\mathsf{UC}_{\X}$.
  It is well-known that for every $n>0$, $(\choice{\boldfacePi_n^0}{\nats})'\weiequiv \choice{\boldfacePi_{n+1}^0}{\nats}$.
 Using the tree representation of closed sets, $\CBaire$ can be formulated as the problem of computing a path through some $T \in \illfounded$: notice that $\CBaire$ closed under compositional product by \cite[Theorem 7.3]{paulybrattka}. Notice that $\CBaire \weiequiv \choice{\boldfaceSigma_1^1}{\Baire}$.

\subsubsection*{The ($\mathbf{k}$-)finitary part of a problem}
\label{finitarypart}
For $k>0$, we denote with $\mathbf{k}$ the space consisting of $\{0,\dots,k-1\}$ endowed with the discrete topology: the $\mathbf{k}$-\emph{finitary part} and the \emph{finitary part} of a problem $f$ captures respectively the most complex problem  with codomain  $\mathbf{k}$  and the most complex problem with finite codomain. We start from the $\mathbf{k}$-{finitary part} of a problem, whose definition follows the same pattern of \cite[Definition 2.2]{valentisolda}.

\begin{definition}
\thlabel{finitaryk}
For every problem $\partialmultifunction{f}{\mathbf{X}}{\mathbf{Y}}$, the $\mathbf{k}$-finitary part of $f$ is the multi-valued function $\partialmultifunction{\mathrm{Fin}_{\mathbf{k}}}{\nats\times\Baire\times\bf{X}}{\mathbf{k}}$ defined as follows:
\begin{itemize}
\item instances are triples $(e,w,x)$ such that $x \in \dom(f)$ and for every $y \in f(x)$ and every name $p_y$ for $y$,
$\Phi_e(w \oplus p_y)(0)\downarrow < k$;
\item a solution for $(e,w,x)$ is any $n<k$ such that there is a name $p_y$ for a solution $y \in f(x)$ with $\Phi_e(w \oplus p_y)(0)\downarrow=n$.
\end{itemize}
\end{definition}

\begin{proposition}
\thlabel{kfintiarypartcharacterization}
For every problem $f$, $\mathrm{Fin}_{\mathbf{k}}\weiequiv \max_{\weireducible} \{g : \subseteq \Baire \rightrightarrows \mathbf{k} \mid g \weireducible f\}$.
\end{proposition}
\begin{proof}
The proof is analogous to \cite[Theorem 3.2]{pauly-valenti} (see also the comment in \cite{pauly-valenti} after Proposition 3.2).
\end{proof}
We now define the {finitary part} of a problem $f$.
\begin{definition}
\thlabel{def:finitepart}
For every problem $f$, we define the \emph{finitary part} of $f$ as  \[\mathrm{Fin}(f) \defas \bigsqcup_{k \geq 1} \mathrm{Fin}_{\mathbf{k}}(f).\]
\end{definition}
We highlight an important fact. Despite its name and its intuitive meaning being reminiscent of the first-order/deterministic/$\mathbf{k}$-finitary part of a problem, the definition of the finitary part of a problem is very different. Indeed, we do not have a similar characterization to the ones given in \thref{firstorderpartcharacterization,kfintiarypartcharacterization}: in particular, we do not even have the codomain of the finitary part is finite.

The following proposition is immediate.
\begin{proposition}
\thlabel{finpartbelowfirstorder}
For every problem $f$, $\mathrm{Fin}(f) \weireducible \firstOrderPart{f}$.
\end{proposition}
We want to show that, for some $f$'s, the reduction in \thref{finpartbelowfirstorder} can be strict (\thref{finitepartcnats}). Before doing so we define when a problem is \emph{join-irreducible} and the \emph{cardinality of a problem}.
\begin{definition}[{\cite[Definition 5.4]{paulybrattka}}]
A problem $f$  is called join-irreducible, if 
\[f\weiequiv \bigsqcup_{i \in \nats} f_i \implies(\exists n_0)(f\weiequiv f_{n_0}).\]
\end{definition}

\begin{definition}[{\cite[Definition 3.5]{brattka2015probabilistic}}]
For every problem $\partialmultifunction{f}{\bf{X}}{\bf{Y}}$ we denote by $\#f$ the maximal cardinality (if it exists) of a set $M \subseteq \dom(f)$ such that $\{f(x) : x \in M\}$ contains pairwise disjoint sets.
\end{definition}

It is easy to see that for every problem $f$ and $g$, $f \strongweireducible g \implies \#f \leq \#g$ \cite[Proposition 3.6]{brattka2015probabilistic}.
\begin{proposition}
	\thlabel{finitepartcnats}
	$\mathrm{Fin}(\cnats) \strictlyweireducible \cnats \weiequiv \firstOrderPart{\cnats}$.
	\end{proposition}
	\begin{proof}
		The reduction and the equivalence are immediate.

	For strictness, suppose $\cnats\weireducible \bigsqcup_{k \geq 1} \mathrm{Fin}_{\mathbf{k}}(\cnats)$. Since $\cnats$ is join-irreducible (\cite[Corollary 5.6]{paulybrattka}), we obtain that if $\cnats \weireducible \bigsqcup_{k \in \nats} \mathrm{Fin}_{\mathbf{k}}(\cnats)$, then there must exist some $k \in \nats$ with $\cnats \weireducible \mathrm{Fin}_{\mathbf{k}}(\cnats)$. Since $\#\mathrm{Fin}_{\mathbf{k}}(\cnats) < \#\cnats$, by \cite[Proposition 3.6]{brattka2015probabilistic} this cannot be the case.
		\end{proof}

\subsection{Computing with graphs}

We conclude this section with some considerations on the the two different graph representation we use: namely via their characteristic function and via an enumeration of the vertices and the edges.
\begin{itemize}
	\item We denote by $\repspacegraphs$ the represented space of graphs represented via characteristic functions, where the representation map $\repmap{\graph}$ has domain 
	\[\{p \in \Cantor: (\forall i,j\in \nats)(p(\str{i,j})=1 \implies p(\str{j,i})=p(\str{i,i})=p(\str{j,j})=1)\}.\]
	Any graph $G$ has a unique $\repmap{\graph}$-name $p \in \Cantor$ such that $ i \in V(G) \iff p(\langle i,i\rangle)=1$ and for $i \neq j$, $ (i,j) \in E(G) \iff p(\langle i,j\rangle)=1$.
	\item We denote by $\repspaceegraphs$ the represented space of graphs represented via enumerations, where the representation map $\repmap{\egraph}$ has domain 
	\[\{p \in \Baire :(\forall i\neq j \in \nats)(\exists k \in \nats)(p(k)=\str{i,j} \implies (\exists \ell_0,\ell_1< k)(p(\ell_0)=\str{i,i} \land p(\ell_1)=\str{j,j})) \}.\]
	A graph $G$ instead has multiple $\repmap{\egraph}$-names namely $\{p \in \dom(\repmap{\egraph}):(i \in V(G) \iff (\exists k)(p(k)=i)) \land ((i,j) \in E(G) \iff (\exists \ell)(p(\ell)=\str{i,j})) \}$.
\end{itemize}

Regarding the space $\repspaceegraphs$, the extra requirement of enumerating edges only after the involved vertices are enumerated is just to simplify some proofs in this paper.

It is trivial that given a $\repmap{\graph}$-name for a graph $G$ we can compute a $\repmap{\egraph}$-name for $G$: the converse is false, but the next lemma tells us that from a $\repmap{\egraph}$-name for $G$ we can compute a $\repmap{\graph}$-name for a graph $H$ \quot{close}\ to $G$, i.e.\ a graph $H$ containing a copy of $G$ plus vertices of finite degree.

\begin{lemma}
	\thlabel{fromcetocomputable}
	There exists a computable $\partialmultifunction{\mathbf{F}}{\repspaceegraphs}{\repspacegraphs}$ such that 
	\[(\forall G\in \repspaceegraphs )(\forall H \in \mathbf{F}(G))(\exists V \subseteq V(H))(\forall v \in V(H)\setminus V)(H\restriction_V\cong G \land \D{H}{v}<\aleph_0).\]
\end{lemma}
\begin{proof}
	Let $q$ be a name for $G \in \repspaceegraphs$. The function $\mathbf{F}$ computes from $q$ a name $p$ for a graph $H \in \repspacegraphs$ in stages. We write $p(\str{i,j})\uparrow_s$ to denote that, at stage $s$, we do not have decided yet if $(i,j) \in E(H)$ or, in case $i=j$, if $i \in V(H)$. To prove the claim, for any $s \in \nats$, we define the auxiliary maps $\function{\iota_s}{\nats}{\nats}$ 
 with the property that, 
 \begin{enumerate}[$(i)$]
		\item for every $s$, $\dom(\iota_s)\defas \{v:(\exists i< s)(q(i)=\str{v,v})\}$ and $\dom(\iota_s)\subseteq \dom(\iota_{s+1})$;
	\end{enumerate}
 and the function $\function{\iota}{\nats}{\nats}$ such that:
	\begin{enumerate}
		\item[$(ii)$] for every $v \in V(G)$, $\iota(v)\defas \lim_{s\rightarrow \infty} \iota_s(v)$ exists.
	\end{enumerate}
\textit{Construction.}\\
We now describe how to compute $p$, but before doing so we introduce the following notation: we say that we \quot{erase $p$ up to stage $s$}\ meaning that for all $i,j\leq s$, if $p(\str{i,j})\uparrow_s$, we set $p(\str{i,j})=0$. We perform this procedure whenever we need to decide if $p(i)=1$ or $p(i)=0$ for some $i \in \nats$, but we do not have received the information we wanted. This procedure does not affect our construction as, when the information arrives, if it tells us that the vertex/edge coded by $i$ is not in $H$ we have already settled $p(i)$ correctly, otherwise we can choose another vertex/edge coded by some $t>s$ that will play the role of vertex/edge coded by $i$. This informal description will be clearer during the construction.

At stage $0$, do nothing. At stage $s+1$, suppose $q(s)=\str{u,v}$. 
First notice that, since an edge $(u,v)$ is enumerated only after $u$ and $v$ have already been enumerated, in this construction it is never the case that $ u \neq v \land (u \notin \dom(\iota_s) \lor v \notin \dom(\iota_s))$. We have the following cases
\begin{itemize}
	\item \emph{Case} 1: $u,v \notin \dom(\iota_s)$ and $u=v$.  Let $m \defas \min \{ k \notin \range(\iota_s): k>s\}$. Choosing $m$ as such, we ensure that $p(\str{m,m})\uparrow_s$ and hence we can set $p(\str{m,m})=1$ and $\iota_{s+1}(u)=m$. Then, erase $p$ up to stage $s$.
	\item \emph{Case} 2:  $u,v \in \dom(\iota_s)$ and $u = v$. By \emph{Case} 1 $p(\str{\iota_s(u),\iota_s(u)})=1$. Then, erase $p$ up to stage $s$.
	\item \emph{Case} 3:  $u,v \in \dom(\iota_s)$ and $u\neq v$,
	\ \begin{itemize}
		\item[(a)] if $p(\str{\iota_s(u),\iota_s(v)}) = p(\str{\iota_s(v),\iota_s(u)})= 1$, then erase $p$ up to stage $s$;
		\item[(b)] if $p(\str{\iota_s(u),\iota_s(v)})\uparrow_s$ and $p(\str{\iota_s(v),\iota_s(u)})\uparrow_s$, let 
		\[p(\str{\iota_s(u),\iota_s(v)})= p(\str{\iota_s(v),\iota_s(u)})=1.\]
		Then, erase $p$ up to stage $s$;
		\item[(c)] if $p(\str{\iota_s(u),\iota_s(v)}) = 1$ and $p(\str{\iota_s(v),\iota_s(u)})=\uparrow_s$, let 
		$ p(\str{\iota_s(v),\iota_s(u)})=1$. Then, erase $p$ up to stage $s$.;
		\item[(d)] if $p(\str{\iota_s(u),\iota_s(v)}) \lor p(\str{\iota_s(v),\iota_s(u)})=0$, let $k_u\defas \min \{k: q(k)=\str{u,u}\}$ and $k_v\defas \min \{k: q(k)=\str{v,v}\}$: if $k_u<k_v$ declare $v$ injured, $u$ otherwise. Suppose $v$ is injured (the case for $u$ is the same). Then let $m \defas \min \{k \notin \range(\iota_s): k>s\}$, let $p(\str{m,m})=1$, $\iota_{s+1}(v)=m$ and for every $i \leq s$, if $q(i)=\str{v,w} \lor q(i)=\str{w,v}$ let $p(\str{m,\iota_s(w)})=p(\str{\iota_s(w),m})=1$. Then, erase $p$ up to stage $s$.;
	\end{itemize}
\end{itemize}
\textit{Verification.}\\
We first verify that $\iota$ and, for all $s$, $\iota_s$ have the desired properties.
To verify the desired properties of $\iota$, notice that $(i)$ is straightforward:  to prove $(ii)$, notice that  
\[(\forall v \in V(G))(|\{s : v \text{ is injured at stage } s\}| \leq |\{w: (v,w) \in E(G) \land w < v\}|).\]
 Let $s_v\defas \max\{t : v \text{ is injured at stage }t\}$: then, since the only case in which $\iota_s(v)\neq \iota_{s+1}(v)$  is in \emph{Case} 3(d), we have that for all $(\forall s\geq s_v)(\iota_{s_v}(v)=\iota_s(v))$ and hence $\iota(v)$ exists.

We now prove that $(\forall G\in \repspaceegraphs)(\forall H \in \mathbf{F}(G))(\exists V\subseteq V(H))(H\restriction_V\cong G)$. Given a name  $p$ for $H \in \repspacegraphs$, let $V\defas \{\iota(v): v \in V(G)\}$. We have to prove that for any $v,u \in V(G)$,  $(v,u) \in E(G) \iff p(\str{\iota(v),\iota(u)})=1$. For the left-to-right direction, suppose $(v,u) \in E(G)$ and notice that, by $(ii)$, $\iota(v)$ and $\iota(u)$ exists. Let $s_0 \defas \max\{t : u \text{ or } v \text{ has been injured at stage } t \}$. By $(ii)$ we get that $\iota_{s_0}(v)=\iota(v)$ and $\iota_{s_0}(u)=\iota(u)$ and by construction, $p(\str{\iota(u),\iota(v)})=1$. For the opposite direction, notice that $p(\str{\iota(v),\iota(u)})=1$ holds only if $(u,v)$ is enumerated in $E(G)$. 

To conclude the proof, we need to show that $(\forall v \in V(H)\setminus V)(\D{H}{v}<\aleph_0)$.
Notice that $v \in V(H)\setminus V$ if and only if $(\exists s,w)(v=\iota_s(w)\neq \iota(w))$ and, by construction, $(\forall t \geq s) (p(\str{v,t})=0)$, i.e.\ $\D{H}{v}<\aleph_0$.
\end{proof}

\section{Effective Wadge complexity of sets of graphs}
\label{wadgecomplexityofgraphs}
In this section, we provide examples of $\Gamma$-complete sets of (names of) graphs where $\Gamma$ is a lightface class. For a fixed graph $G$, we consider sets of the form
\begin{equation}
	\label{setsweconsider}
\setofgraphs{G}{\repspaceallgraphs}{\allgraphrelations} \defas \{p \in \dom(\repmap{(E)Gr}) : G \allgraphrelations \repmap{(E)Gr}(p)\}.
\end{equation}
 The next proposition, whose easy proof is omitted, is a direct consequence of the definitions of $\repspacegraphs$ and $\repspaceegraphs$.

 \begin{proposition}
 	\thlabel{enumerationabovecharacteristic}
	For any graph $G$, $\setofgraphs{G}{\repspacegraphs}{\allgraphrelations}\effectivewadge \setofgraphs{G}{\repspaceegraphs}{\allgraphrelations}$.
 \end{proposition}

 The following two operations on graphs and trees are fundamental in many constructions of this paper.
 Given a tree $T$ and a graph $G$ where $V(G)=\{v_i: i \in \nats\}$, let $\constructionone{T,G}$ and $\constructiontwo{T,G}$ be such that $V(\constructionone{T,G})=V(\constructiontwo{T,G})\defas T$ (i.e.\ any node in $T$ is a vertex of the resulting graph) and
 \begin{itemize}
 \item $E(\constructionone{T,G}) \defas \{(\sigma,\tau):(v_{\length{\sigma}}, v_{\length{\tau}}) \in E(G) \land (\sigma \sqsubset \tau \lor \tau \sqsubset \sigma)\}$. 
 \item $E(\constructiontwo{T,G}) \defas \{(\sigma,\tau):(v_{\length{\sigma}}, v_{\length{\tau}}) \in E(G) \lor (\sigma\incomparable \tau)\}$. 
 \end{itemize}
 Notice that both $\constructionone{T,G}$ and $\constructiontwo{T,G}$ are computable relative to $T$ and $G$. 

\begin{proposition}
\thlabel{Constructionongraphs}
Let $T \in \illfounded$ and let $G$ be an infinite graph. Then $G \inducedsubgraph \constructionone{T,G}$ and $G \inducedsubgraph \constructiontwo{T,G}$.
\end{proposition}
\begin{proof}
	In both cases, suppose that $V(G)=\{v_i : i \in \nats\}$, let $p \in \body{T}$ and consider $V \defas \{p[n]:n \in \nats\}$. By definitions of $E(\constructionone{T,G})$ and $E(\constructiontwo{T,G})$, it is easy to check that  $\constructionone{T,G}_{\restriction_V} \cong G$ and $\constructiontwo{T,G}_{\restriction_V} \cong G$ as well, and this proves the claim.
\end{proof}

\begin{remark}
	\thlabel{remarkccehyp}
	In this paper, most of the results are stated for graphs having a computable, c.e. or hyperarithmetical copy. In the literature, it is customary to say that a graph is computable if its vertex set is computably isomorphic to $\nats$ and it has a computable edge relation: similar definitions hold for c.e.\ and hyperarithmetical graphs. From now on, for notational simplicity, we modify these definitions saying e.g., that a graph is computable if it has a copy with vertex set computably isomorphic to $\nats$ and computable edge relation: in all proofs, when we consider a computable graph, we always consider a computable copy of it (similarly for c.e.\ and hyperarithmetical graphs).  For the computable case, this is equivalent to say that a computable graph has a computable $\repmap{\graph}$-name and, for the c.e.\ case, that a c.e.\ graph has a computable $\repmap{\egraph}$-name.
\end{remark}

In this section, we show that, if we restrict to a computable or c.e.\ graph $G$, except when $G$ is finite but not isomorphic to any $\complete{n}$ and we consider the induced subgraph relation, we obtain that both $\setofgraphs{G}{\repspaceallgraphs}{\allgraphrelations}$ are $\Gamma$-complete for some lightface class $\Gamma$, and hence, in these cases, the conclusion of \thref{enumerationabovecharacteristic} becomes an equivalence. To prove results of this kind, we usually perform the following steps:
\begin{itemize}
 	\item  show that $\setofgraphs{G}{\repspaceegraphs}{\allgraphrelations}$ is in $\Gamma$;
 	\item  prove that for any $\Gamma$-complete set $A$,  $A \effectivewadge\setofgraphs{G}{\repspacegraphs}{\allgraphrelations}$;
 	\item  apply \thref{enumerationabovecharacteristic} to obtain that both $\setofgraphs{G}{\repspacegraphs}{\allgraphrelations}$ and $\setofgraphs{G}{\repspaceegraphs}{\allgraphrelations}$ are $\Gamma$-complete (and hence effectively Wadge equivalent).
 \end{itemize}

 \begin{proposition}
 	\thlabel{upperboundandfinite}
 	For any hyperarithmetical graph $G$, the following sets are $\Sigma_1^{1}$: 
 	\[\text{(i)}\ \setofgraphs{G}{EGr}{\subseteq_\mathsf{s}}, \ \text{(ii)}\ \setofgraphs{G}{Gr}{\subseteq_\mathsf{s}}, \]
 \[\text{(iii)}\ \setofgraphs{G}{Gr}{\subseteq_\mathsf{is}} \ \text{(iv)}\ \setofgraphs{G}{EGr}{\subseteq_\mathsf{is}}.\]
 	If $G$ is finite, the first three sets are $\Sigma_1^{0}$-complete. For any $n \in \nats$, $\setofgraphs{\complete{n}}{EGr}{\subseteq_\mathsf{is}}$  is $\Sigma_1^{0}$-complete, otherwise, if $G \not\cong \complete{n}$, the set $\setofgraphs{G}{EGr}{\subseteq_\mathsf{is}}$ is $\Sigma_2^{0}$-complete.
 \end{proposition}
 \begin{proof}
	We start showing that all sets in $(i)-(iv)$ are $\Sigma_1^1$.
 	\begin{enumerate}[$(i)$]
 		\item Let $p$ be a name for $H \in \repspaceegraphs$. Then $G \subgraph H$ if and only if  
 		\[(\exists f)(\forall i, j \in V(G))((i=j \lor (i,j) \in E(G)) \implies (\exists k)(p(k)= \str{f(i),f(j)})).\]
 		\item Let $p$ be a name for $H \in \repspacegraphs$. Then  $G\subgraph H$ if and only if
 		\[(\exists f)(\forall i, j \in V(G))(((i=j \lor (i,j) \in E(G)) \implies p(\str{f(i),f(j)})=1)).\]
 		\item  Let $p$ be a name for $H \in \repspacegraphs$. Then  $G\inducedsubgraph H$ if and only if 
 		\[(\exists f)(\forall i, j \in V(G))(((i=j \lor (i,j) \in E(G)) \iff p(\str{f(i),f(j)})=1)).\]
 		\item  Let $p$ be a name for $H \in \repspaceegraphs$. Then $G \inducedsubgraph H$ if and only if  
 		\[(\exists f)(\forall i, j \in V(G)) ( (i=j \lor (i,j) \in E(G)) \iff (\exists k)(p(k)= \str{f(i),f(j)})).\]
 	\end{enumerate}
 	All the formulas are $\Sigma_1^1$ and this concludes the first part of the proof.

 In case $G$ is finite, the first existential quantification ranges over $\nats$ while the quantification over $V(G)$ ranges over a finite set. This shows that the sets  $(i)$, $(ii)$, and $(iii)$ are $\Sigma_1^0$ sets, and $(iv)$  is $\Sigma_2^0$. On the other hand,  for any $n \in \nats$,  given a name $p$ for $H \in \repspaceegraphs$, $\complete{n} \inducedsubgraph H$ if and only if 
 \[(\exists \sigma \in \nats^{\length{E(\complete{n})}+\length{V(\complete{n})}})(\forall i, j \in V(\complete{n}))(\exists k)(p(k)=\str{\sigma(i),\sigma(j)}),\]
 and hence the formula defining the set in $(iv)$ is $\Sigma_1^0$. 
 
 We now show that, for finite graphs, the sets above are complete in the corresponding lightface classes. We prove that the set in $(iii)$ is $\Sigma_1^0$-complete and that if $G\not\cong\complete{n}$ for some $n>0$ the set in $(iv)$ is $\Sigma_2^0$-complete: we skip the proofs for the sets in $(i)$ and $(ii)$ as they follow the same pattern. To do so, we prove that  $\{p \in \Cantor : (\exists i)(p(i)=1)\} \effectivewadge \setofgraphs{G}{\repspacegraphs}{\subgraph}$ (recall that, by \thref{firstcompletesets}, the left-hand-side set is $\Sigma_1^0$-complete). Given $p\in \Cantor$ we define a computable $\function{f}{\Cantor}{\repspacegraphs}$ such that
 	\begin{itemize}
 		\item if $p(s)=0$, then let $f(p[s])$ be the empty graph;
 		\item if $p(s)=1$, then let $f(p) \cong G$, i.e.\ add a fresh copy of $G$ to the graph we are computing and stop the construction.
 	\end{itemize}
To show that $f$ is an effective Wadge reduction it suffices to notice that, if $(\exists i)(p(i)=1)$ then clearly  $G \inducedsubgraph f(p)\cong G$ while, if $(\forall i)(p(i)=1)$ then $f(p)$ is the empty graph and no nonempty graph is an induced subgraph of the empty graph.

  To conclude the proof, it remains to prove that $\setofgraphs{G}{EGr}{\subseteq_\mathsf{is}}$ for $G\not\cong \complete{n}$ for some $n \in \nats$ is $\Sigma_2^0$-complete and to do so we show that $\{p : (\forall^\infty i)(p(i)=0)\} \effectivewadge \setofgraphs{G}{EGr}{\subseteq_\mathsf{is}}$ (recall that, by \thref{firstcompletesets}, the left-hand-side set is $\Sigma_2^0$-complete). Given $p\in \Cantor$, we define a computable $\function{f}{\Cantor}{\repspaceegraphs}$ such that
 	\begin{itemize}
 		\item if $p(s)=0$, then let $f(p[s]) \defas f(p[s-1]) \disconnectedunion{}{} G'$, where $G'$ is a fresh copy of $G$;
 		\item if $p(s)=1$, then let $f(p[s])  \cong \complete{n_s}$ where  $n_s \defas \length{\{i: i \in V(f(p[s-1])) \}}$, i.e. enumerate enough edges in the graph we are computing to make it isomorphic to a complete finite graph.
 	\end{itemize}
 	Finally, if $(\forall^\infty i)(p(i)=0)$ then $f(p)\cong \complete{n} \disconnectedunion{}{} \infinitelycopies{G}$ for some $n \in \nats$, and clearly $G \inducedsubgraph f(p)$. Otherwise, if $(\exists^\infty i)(p(i)=1)$ then $f(p) \cong \complete{\omega}$. Combining the fact that $f(p)$ is an induced subgraph of $\complete{\omega}$ if and only if $f(p) \cong \complete{n}$  for some $n \in \nats$ or $f(p) \cong \complete{\omega}$ and the fact that $G\not\cong\complete{\omega}$ and $G\not\cong \complete{n}$ for any $n \in \nats$, we conclude that $G \not \inducedsubgraph f(p)$.
 \end{proof}

 Before moving to the case when $G$ is infinite, we give the following lemma, which is a particular case of the chain antichain principle (CAC). This principle asserts that each partial order on $\nats$ contains an infinite chain or an infinite antichain (notice that a tree, in particular, is a partial order, where the order is given by the prefix order of the tree). This principle is well-studied in reverse mathematics and its proof is an easy application of Ramsey theorem for pairs.

 \begin{lemma}
 	\thlabel{inducedsubgraphcacprinciple}
 	Let $T$ be a well-founded tree and $S$ an infinite subset of $T$. Then $S$ contains an infinite anti-chain, i.e.\ a set of nodes that are pairwise incomparable with respect to the prefix order of $T$.
 \end{lemma}

 In the next three proofs, the function $\mathbf{F}$ is the one defined in \thref{fromcetocomputable}.
 \begin{theorem}
 	\thlabel{upperboundsinducedkomega}
 	Let $G$ be a c.e.\ graph such that $\complete{\omega}\not\inducedsubgraph G$. Then $\setofgraphs{G}{\repspaceallgraphs}{\subseteq_\mathsf{is}}$ are $\Sigma_1^1$-complete.
 	\end{theorem}
 	\begin{proof}
 		By \thref{upperboundandfinite} both sets are $\Sigma_1^1$ and by \thref{enumerationabovecharacteristic}, to prove the claim it suffices to prove that $\setofgraphs{G}{\repspacegraphs}{\inducedsubgraph}$ is $\Sigma_1^1$-complete. To do so, we show that $\illfounded\effectivewadge\setofgraphs{G}{\repspacegraphs}{\inducedsubgraph}$ (recall that, by \thref{Complexityresults}, $\illfounded$ is $\Sigma_1^1$-complete).  Let $T \in \tree$, consider $G' \in \mathbf{F}(G)$ and, given $V(G')=\{v_i : i \in \nats\}$, compute $\constructiontwo{T,G}$. Notice that
 \begin{itemize}
 \item if $T \in \illfounded$, by \thref{Constructionongraphs}, $G' \inducedsubgraph \constructiontwo{T,G'}$ and hence, since $G\inducedsubgraph G'$, we obtain that $G \inducedsubgraph \constructiontwo{T,G'}$;
     \item if  $T \in \wellfounded$, we claim that $G \not\inducedsubgraph\constructiontwo{T,G'}$. If $T$ is finite (i.e.\ $T$ has finitely many nodes), since $G$ is infinite, the claim follows trivially. If $T$ is infinite, by \thref{inducedsubgraphcacprinciple}, for any infinite $H$ that is an induced subgraph of  $\constructiontwo{T,G'}$ there exists $H'$ that is an induced subgraph of $H$ such that $V(H')$ is an anti-chain in $T$. By definition of $E(\constructiontwo{T,G'})$, all incomparable nodes in $T$ are connected by an edge, and hence  $H'\cong \complete{\omega}$ is a subgraph of $H$.  Since $\complete{\omega}\not\subgraph G$ this concludes the proof of the claim.
 \end{itemize}
 Hence, $T$ is ill-founded if and only if $G \inducedsubgraph\constructiontwo{T,G'}$, and this concludes the proof.
 \end{proof}

 	While the theorem above holds only for the induced subgraph relation, the following holds also for the subgraph one. 
 \begin{theorem}
 	\thlabel{upperboundsinducedrayomega}
 	Let $G$ be a c.e.\ graph such that $\R\subgraph G$. Then the sets $\setofgraphs{G}{\repspaceallgraphs}{\allgraphrelations}$ are $\Sigma_1^1$-complete. 
 	\end{theorem}
 	\begin{proof}
 		By \thref{upperboundandfinite} all four sets are $\Sigma_1^1$ and by \thref{enumerationabovecharacteristic}, it suffices to show that the sets $\setofgraphs{G}{\repspacegraphs}{\allgraphrelations}$ are $\Sigma_1^1$-complete. We only show it for the subgraph relation, as the same proof works also for the induced one. As in \thref{upperboundsinducedkomega}, we show that $\illfounded\effectivewadge \setofgraphs{G}{\repspacegraphs}{\subgraph}$. Let $T \in \tree$, consider $G' \in \mathbf{F}(G)$ and, given $V(G')=\{v_i : i \in \nats\}$, compute $\constructionone{T,G}$. Notice that:
 			\begin{itemize}
 			\item if $T \in \illfounded$, by \thref{Constructionongraphs}, $G' \inducedsubgraph \constructionone{T,G'}$ and hence, since $G\inducedsubgraph G'$, $G \inducedsubgraph \constructionone{T,G'}$;
 			 \item if  $T \in \wellfounded$, we claim that $G\not\subgraph \constructionone{T,G'}$. Since, $\R\subgraph G$ by hypothesis, it suffices to show that $\R\not\subgraph \constructionone{T,G'}$. We show the contrapositive: suppose that $\R\subgraph \constructionone{T,G'}$ and let $\{ \sigma_i : i \in \nats\}$ such that for every $i$, $(\sigma_i,\sigma_{i+1}) \in E(\constructionone{T,G'})$ be the vertices of the copy of $\R$ in $\constructionone{T,G'}$. By definition of $E(\constructionone{T,G'})$, $(\sigma,\tau) \in E(\constructionone{T,G'})\implies \sigma \sqsubset \tau \lor \tau \sqsubset \sigma$. Hence, $T$ contains infinitely many nodes comparable to each other, i.e.\ $T \in \illfounded$ and this proves the claim.
 			\end{itemize}
 			To conclude the proof, notice that $T \in \illfounded$ if and only if $G \inducedsubgraph \constructionone{T, G'}$.
 		\end{proof}

		The next proposition gives us a lower bound to the complexity of sets of the form $\{H \in \repspaceallgraphs : G \subgraph H\}$ where $G$ is c.e.\ and infinite.
 \begin{proposition}
	For any infinite c.e.\ graph $G$, the sets $\{H \in \repspaceallgraphs : G \subgraph H\}$ are $\Pi_2^0$-hard.
 \end{proposition}
 \begin{proof}
 	To prove the claim, recall that, by \thref{firstcompletesets} the set $\{p \in \Cantor: (\exists^\infty n)(p(n)=1)\}$ is $\Pi_2^0$-complete and by \thref{enumerationabovecharacteristic}, it suffices to show that $ \{p \in \Cantor: (\exists^\infty n)(p(n)=1)\}\effectivewadge \{H \in \repspacegraphs : G \subgraph H\}$. Let $p \in \Cantor$ and consider $G' \in \mathbf{F}(G)$. We compute $H$ letting $V(H)\defas\{n:p(n)=1\}$ and $E(H)\defas\{(n,m) : p(n)=p(m)=1\}$. Notice that:
 	\begin{itemize}		
 		\item if $(\exists^{\infty} i)(p(i)=1)$ then $\complete{\omega} \subgraph H$, and since any graph is a subgraph of $\complete{\omega}$ we obtain that $G \inducedsubgraph G' \subgraph H$;
 		\item if $(\forall^{\infty} i)(p(i)=0)$ then $H$ is finite and since $G$ is infinite, we obtain that $G\not\subgraph H$.
 	\end{itemize}
 	Hence, $(\exists^{\infty} i)(p(i)=1)$ if and only if $G\subgraph H$ and this concludes the proof.
 \end{proof}

 So far, the sets of graphs of the form $\setofgraphs{G}{\repspaceallgraphs}{\allgraphrelations}$ for a fixed graph $G$ we considered are $\Gamma$-complete for $\Gamma \in \{\Sigma_1^0,\Sigma_2^0,\Sigma_1^1\}$: by the end of this section, we define sets of the same form that are $\Gamma$-complete with $\Gamma$ varying along the lightface arithmetical hierarchy. To do so, for $k \in \nats$, consider the graphs $\mathsf{T}_{2k+1}$ and $\mathsf{F}_{2k+2}$ where:
 $$V(\mathsf{T}_{2k+1})\defas \{\sigma \in \baire:\length{\sigma}\leq k\} \text{, } V(\mathsf{F}_{2k+2}) \defas \{\sigma \in \baire : 0 < \length{\sigma} \leq k+1\} \text{ and, }$$
 for $G \in \{\mathsf{T}_{2k+1}, \mathsf{F}_{2k+2}\}$, 
	  \[E(G)\defas \{(\sigma,\tau) \in E(G): \sigma,\tau \in V(G) \land (\sigma=\tau[\length{\tau}-1]\lor \tau=\sigma[\length{\sigma}-1])\}.\]
	  To have an intuitive idea on the graphs' definition above, notice that $\mathsf{T}$ and $\mathsf{F}$ stand respectively for \quot{tree}\ and \quot{forest} where a forest is the disconnected union of countably many trees. 

	  Given trees $T$ and $S$, we define the \textit{disjoint union} of $T$ and $S$ as $T\sqcup S=\{\str{}\} \cup \{\str{0} \tau: \tau \in  T \} \cup \{\str{1} \tau: \tau \in  S \}$. Of course, this is still a tree, and the construction can be easily generalized to countably many trees letting $\disjointunion{i\in \nats}{T^i}\defas\{\str{}\} \cup \{\str{i} \tau:\tau \in T^i \land i \in \nats\}$.  Starting from $\mathsf{T}_1 \cong  (\{\str{}\}, \emptyset)$ (the graph/tree having a unique vertex), the following relations hold:
	  \[\mathsf{F}_{2k+2}\cong \infinitelycopies{ \mathsf{T}_{2k+1}} \text{ and } \mathsf{T}_{2k+3}\cong \disjointunion{i \in \nats}{T_i} \text{ where } T_i \text{ is a connected component of }\mathsf{F}_{2k+2}.\]
 
	  Notice that all $\mathsf{T}_i$'s for $i \in \nats$ are well-founded, and for a well-founded tree $T$ we can define its \emph{height} as $\max\{\length{\sigma} : \sigma \in T\}$: notice that the $k$ in $\mathsf{T}_{2k+1}$ (respectively $\mathsf{F}_{2k+2}$) corresponds to the height of $\mathsf{T}_{2k+1}$  (the connected components of  $\mathsf{F}_{2k+2}$). For example, $\mathsf{T}_1= \{\str{}\}$ is a tree of height $0$, $\mathsf{F}_{2}$ is the forest having infinitely many trees of height $0$, $\mathsf{T}_3$ is a tree of height $1$ with infinitely many children, $\mathsf{F}_4$ is the forest made by infinitely many copies of $\mathsf{T}_3$ and so on.

	  \begin{theorem}
		\thlabel{subgraphcomplexity}
		\
	  \begin{itemize}
		  \item  The sets $\setofgraphs{\mathsf{T}_{2k+1}}{\repspaceallgraphs}{\subgraph}$ are $\Sigma_{2k+1}^0$-complete;
		  \item The sets $\setofgraphs{\mathsf{F}_{2k+2}}{\repspaceallgraphs}{\subgraph}$ are $\Pi_{2k+2}^0$-complete.
	  \end{itemize}
	  \end{theorem}

The proof of this theorem is obtained combining  \thref{subgraphcondition,treesforests}: the first proves that the sets involved are respectively $\Sigma_{2k+1}^0$ and $\Pi_{2k+2}^0$. The second one that (subspaces of) $\repspaceallgraphs$ are hard (and hence by \thref{subgraphcondition} complete) for the corresponding lightface classes.

\begin{lemma}
	\thlabel{subgraphcondition}
	Given a graph $H$, 
  \begin{itemize}
	  \item $\mathsf{T}_{2k+1} \subgraph H$ if and only if $(\exists v_0 \in V(H))(\exists^{\infty}v_1\in V(H))\dots(\exists^{\infty} v_{k}\in V(H))(\forall j<k)((v_{j},v_{j+1})\in E(H))$.
	  \item $\mathsf{F}_{2k+2} \subgraph H$ if and only if $(\exists^\infty v_0 \in V(H))(\exists^{\infty}v_1\in V(H))\dots(\exists^{\infty} v_{k}\in V(H))(\forall j<k)((v_{j},v_{j+1})\in E(H))$.
  \end{itemize}
  Hence,  $\setofgraphs{\mathsf{T}_{2k+1}}{\repspaceallgraphs}{\subgraph}$ and $\setofgraphs{\mathsf{F}_{2k+2}}{\repspaceallgraphs}{\subgraph}$ are respectively  $\Sigma_{2k+1}^0$ and $\Pi_{2k+2}^0$.
  \end{lemma}
  \begin{proof}
	In both cases, the left-to-right direction is trivial. For the opposite direction, we just prove the second item, as the proof of the first one is similar. Given a graph $H$, fix $k \in \nats$ and, for every $j < k$, let $(v_j^m)_{m \in \nats}$ be an enumeration of the $v_j$'s in $V(H)$: notice that, for $i \neq l \leq j$  vertices that appear in $(v_i^m)_{m \in \nats}$  may appear also in $(v_l^m)_{m \in \nats}$. We define a function $\iota$ from $V(\mathsf{F}_{2k+2})\subseteq \baire$ to $V(H)$ letting $\iota(\sigma)\defas v_{\length{\sigma}}^k$ if and only if
 $$(\forall \tau < \sigma)(\iota(\tau) \neq v_{\length{\sigma}}^k) \text{ and }(\exists l) (\iota(\sigma[\length{\sigma}-2])= v_{\length{\sigma}-1}^l \land (v_{\length{\sigma}}^k,v_{\length{\sigma}-1}^l) \in E(H))$$
			 Notice that, by definition of $H$, for any $\sigma$,  $v_{\length{\sigma}}^k$ as above always exists. Indeed, for $\sigma=\str{}$ this is immediate. For $\sigma \neq \str{}$, it suffices to notice that at any stage  we have defined only finitely many vertices in the range of $\iota$ while, by hypothesis, $\D{H}{\iota(\sigma[\length{\sigma}-2])}=\aleph_0$ (and hence we can always find $v_j^k$ satisfying the conditions above).

 Let $\iota(V(\mathsf{F}_{2k+2}))\defas\{\iota(\sigma) : \sigma \in V(\mathsf{F}_{2k+2})\} $ and consider the graph $H_{\restriction V}$: clearly, $H_{\restriction V} \cong \mathsf{F}_{2k+2}$ and this concludes the proof of the lemma.
  \end{proof}

\begin{remark}
	\thref{treesforests} shows that the complexity of such sets remain the same even if we restrict the graphs $H$ to trees/forests of length at most $k$: this comes in handy in \S \ref{othersubgraphproblems}.
\end{remark}

Given a subspace $\mathcal{G} \subseteq \repspaceallgraphs$, we define $\bigotimes \mathcal{G} \defas \{\disconnectedunion{n \in \nats}{G_n} : (\forall n)(G_n \in \mathcal{G})\}$. We denote by $\tree_{\leq k}$ the represented spaces of trees of height at most $k$: that is, 
  $$\tree_{\leq k} \defas \{T \in \tree: (\forall \sigma \in T)(\length{\sigma}\leq k) \}.$$

  Clearly $\tree_{\leq k}$ is a subspace of $\repspacegraphs$, and in the next lemma we consider the space $\bigotimes\tree_{\leq_k}$.

\begin{lemma}
	\thlabel{treesforests}
	For $k \in \nats$:
	\begin{itemize}
		\item ${Trees}_{2k+1}\defas\{G \in \bigotimes\tree_{\leq_k}:  \mathsf{T}_{2k+1} \subgraph G \}$ is $\Sigma_{2k+1}^0$-hard;
		\item 	${Forests}_{2k+2}\defas\{G \in \bigotimes\tree_{\leq_k}: \mathsf{F}_{2k+2} \subgraph G\}$ is $\Pi_{2k+2}^0$-hard.
	\end{itemize}
\end{lemma}
\begin{proof}
	By \thref{subgraphcondition}, ${Trees}_{2k+1}$ and ${Forests}_{2k+2}$ are respectively $\Sigma_{2k+1}^0$ and $\Pi_{2k+2}^0$ sets.

	To show that the sets above are complete for the corresponding lightface classes, we first prove by induction on $k$ that ${Forests}_{2k+2}$ is $\Pi_{2k+2}^0$-complete. For the base case, since $\{p \in \Cantor : (\forall^\infty i)(p(i)=0)\}$ is $\Pi_2^0$-complete (\thref{firstcompletesets}), it suffices to show that $\{p \in \Cantor : (\forall^\infty i)(p(i)=0)\}\effectivewadge{Forests}_{2}$. Given $p \in \Cantor$, compute $T\defas \{\str{}\} \cup \{\str{n}: p(n)=0\}$ and notice that
	$$p \in \{p \in \Cantor : (\exists^\infty i)(p(i)=0)\} \iff \D{T}{\str{}}=\aleph_0 \iff \mathsf{F}_2 \subgraph T.$$
	Assuming that ${Forests}_{2k+2}$ is $\Pi_{2k+2}^0$-complete we aim to show that ${Forests}_{2(k+1)+2}$ is $\Pi_{2k+4}^0$-complete.
	By \thref{Aomega},
	$${Forests}_{2k+2}^{0}\defas \{ (G_n)_{n \in \nats} \in (\bigotimes\tree_{\leq_k})^\nats : (\exists^\infty n)(G_n \in {Forests}_{2k+2})\} \text{ is }\Pi_{2k+4}^0\text{-complete}.$$
	Notice that, given $(G_n)_{n \in \nats} \in {Forests}_{2k+2}^{0}$ for every $n$, $G_n= \disconnectedunion{i \in \nats}{(T^{n,i})_{i \in \nats}}$ where $T^{n,i} \in \tree_{\leq k}$ : given $(G_n)_{n \in \nats} \in (\bigotimes\tree_{\leq_k})^\nats$, compute $F \defas \disconnectedunion{n \in \nats}{\big(\disjointunion{i \in \nats}{T^{n,i}}\big)}$. Informally, $F^n$ is obtained first connecting, for every $n \in \nats$, to a new root the roots of the $\{T^{n,i} : i \in \nats\}$ and then considering the disconnected union of the resulting trees: notice that, by hypothesis all $T^{n,i}$ are in $\tree_{\leq k}$ and, by construction all connected components of $F$ are in $\tree_{\leq k+1}$, and hence $F \in \bigotimes\tree_{\leq k+1}$. It is easy to check that  $ (G_n)_{n \in \nats} \in {Forests}_{2k+2}^{0} \iff F \in {Forests}_{2k+4}$,	and hence  ${Forests}_{2k+4}$ is $\Pi_{2k+4}^0$-complete.

	To conclude the proof, it suffices to prove for any $k \in \nats$, ${Trees}_{2k+1}$ is $\Sigma_{2k+1}^0$-complete. For $k=0$ we can prove the claim showing that $P_1 \effectivewadge Trees_1$: the proof is an easy adaptation of the one given to prove  $P_2\effectivewadge{Forests}_{2}$. For $k>0$, the proof is almost the same of the one showing that ${Forests}_{2k+2}$ is $\Pi_{2k+2}^0$-complete. Indeed, by \thref{Aomega}
	$${Trees}_{2k+2}^{0}\defas \{ (G_n)_{n \in \nats} \in (\bigotimes\tree_{\leq_k})^\nats : (\exists n)(G_n \in {Forests}_{2k+2})\} \text{ is }\Sigma_{2k+3}^0\text{-complete},$$
and we compute $F$ as above. It is easy to check that  $ F \in {Forests}_{2k+2}^{1} \iff F \in {Tree}_{2k+3}$,	and hence  ${Tree}_{2k+3}$ is $\Sigma_{2k+3}^0$-complete and this concludes the proof.
\end{proof}

Due to the fact that we can always think an element of (products) of $\tree_{\leq k}$ as an element of $\repspacegraphs$, the previous lemma, together with \thref{subgraphcondition}, proves \thref{subgraphcomplexity}.

We leave open whether there exists a graph $G$ such that $\setofgraphs{G}{\repspaceallgraphs}{\subgraph} \in \Gamma$ for $\Gamma \notin \{\Sigma_{2k+1}^0,\Pi_{2k+2}^0, \Sigma_1^1\}$ for $k \in \nats$.

\begin{remark}
Notice that from the results presented in this section  we also obtain the complexity of the same sets of graphs in the boldface hierarchy: in this context, graphs are given via their characteristic functions, i.e.\ as element of $\repspacegraphs$, and completeness is defined with respect to Wadge reducibility. The Wadge reductions are the effective ones we used to prove the statements for the lightface case. We do not restate all the results for the boldface case, except for the next one that gives a sort of dichotomy for $\setofgraphs{G}{\repspacegraphs}{\inducedsubgraph}$: this kind of results can be found for different structures for example in \cite{camerlo2005continua}.
\end{remark}

\begin{theorem}
	For any finite graph $G$, $\setofgraphs{G}{\repspacegraphs}{\inducedsubgraph}$ is $\boldfaceSigma_1^0$-complete, otherwise it is $\boldfaceSigma_1^1$-complete.
\end{theorem}

\begin{proof}
	For finite $G$'s the proof is the same of the one in \thref{upperboundandfinite}$(iii)$. 
	
	For infinite $G$'s, the proof that $\setofgraphs{G}{\repspacegraphs}{\inducedsubgraph}$ is $\boldfaceSigma_1^1$ is the same of \thref{upperboundandfinite}. For completeness, we show that $\illfounded \wadgereducible \setofgraphs{G}{\repspacegraphs}{\inducedsubgraph}$. To do so, we combine the effective Wadge reductions used in the proofs of \thref{upperboundsinducedkomega,upperboundsinducedrayomega},  noticing that if $\complete{\omega} \inducedsubgraph G$ then $\R\subgraph G$. Then it is easy to check that the following function is the desired Wadge reduction $$f_G(T) \defas \begin{cases}
		\constructionone{T,G}&\text{ if } \complete{\omega}\inducedsubgraph G,\\
		\constructiontwo{T,G}&\text{ if } \complete{\omega}\not\inducedsubgraph G,
	\end{cases}$$
	i.e., $T \in \illfounded \iff G \inducedsubgraph f_G(T)$.
\end{proof}

\subsection{Deciding  (induced) subgraphs problems in the Weihrauch lattice}
\label{DecisionProblemsGraphs}
In this section we consider problems of the following form:
	\begin{definition}
	\thlabel{def:isg}
	For a computable graph $G$, we define the functions $\function{\isg_G}{\repspacegraphs}{2}$ and $\function{\sg_G}{\repspacegraphs}{2}$ by
	$$\isg_G(H)=1 \iff G\inducedsubgraph H \text{ and } \sg_G(H)=1 \iff G\subgraph H.$$
	The same functions having domain $\repspaceegraphs$ are denoted respectively with $e\isg_G$ and $e\sg_G$.
	\end{definition}
	Notice that in \cite{bement2021reverse}, the authors considered only $\sg_G$ and $\isg_G$, and they denote these problems respectively with $\mathsf{SE}_G$ and $\mathsf{S}_G$.
   
   All the results in this section, are immediate consequences of  \S \ref{wadgecomplexityofgraphs}: indeed, in the following proofs, the forward and backward functionals witnessing the Weihrauch reduction are respectively the effective Wadge reduction coming from the corresponding result in \S \ref{wadgecomplexityofgraphs} and the identity. 
	\begin{proposition}
	 \thlabel{easycorollary}
	For any hyperarithmetical graph $G$,
	\[\isg_G \sweireducible e\isg_G \sweireducible \wf \text{ and } \sg_G \sweireducible e\sg_G \sweireducible \wf.\]
	 \end{proposition}
	\begin{proof}
		It follows by \thref{enumerationabovecharacteristic,upperboundandfinite} and the fact that $\wf$ answers any $\Sigma_1^1$-complete question relative to the input (see \S \ref{subsec:Weihrauch}).
	\end{proof}
The following theorem discusses the problems $\isg_G$ and $e\isg_G$ for finite graphs. 
\begin{theorem}
\thlabel{finiteinduced}
For any finite graph $G$, $\lpo \sweiequiv (e)\sg_G \sweiequiv \isg_G$.
If $G \cong \complete{n}$ for some $n>0$, then $\lpo \sweiequiv e\isg_G$, otherwise $\lpo' \sweiequiv e\isg_G$.
	 \end{theorem}
\begin{proof}
	The fact that $\lpo \sweiequiv \isg_G$ is from \cite[Theorem 27]{bement2021reverse}. Essentially the same proof also gives us that $\lpo \sweiequiv \sg_G$. Anyway, all results are direct consequences of \thref{upperboundandfinite} and the fact that $\lpo$ and $\lpo'$ answer respectively any $\Sigma_1^0$.complete and $\Sigma_2^0$-complete questions relative to the input (see \S \ref{subsec:Weihrauch}).
\end{proof}
In \cite[Theorem 24]{bement2021reverse} the authors showed that $\wf  \sweiequiv (\mathsf{i})\sg_{\R}  \sweiequiv (\mathsf{i})\sg_{ \R\connectedunion{}{\cycle{n}}}$ for some $n>2$ and left open the following question:
\begin{equation}
\label{questionngraph}
\text{is there a computable graph } G \text{ such that } \lpo \strictlyweireducible \isg_G \strictlyweireducible \wf\text{?}
\end{equation}
By \thref{finiteinduced} we know that if such a $G$ exists, it must be infinite.

\begin{lemma}
   \thlabel{lemma1question}
   Let $G$ be a c.e.\ graph such that $\R\subgraph G$. Then, $\wf \sweiequiv (e)\sg_G \sweiequiv (e)\isg_G$.
\end{lemma}
\begin{proof}
   It follows immediately from \thref{upperboundsinducedrayomega,easycorollary}.
\end{proof}

\begin{lemma}
   \thlabel{lemma2question}
   Let $G$ be a c.e.\ infinite graph such that  $\complete{\omega} \not\inducedsubgraph G$. Then $\wf \sweiequiv (e)\isg_G$.
\end{lemma}
\begin{proof}
   It follows immediately from \thref{upperboundsinducedkomega,easycorollary}.
\end{proof}
 We finally obtain a negative answer to question (\ref{questionngraph}) showing that the claim does not hold for even a larger class of graphs and problems.
	\begin{theorem}
		There is no c.e.\ graph $G$ such that $\lpo \strictlyweireducible (e)\isg_G \strictlyweireducible \wf$.
	\end{theorem}
	\begin{proof}
	   If $G$ is finite then $\isg_G\sweiequiv \lpo$ (\thref{finiteinduced}).   If $G$ is an infinite graph then either $\complete{\omega} \not\inducedsubgraph G$ or $\complete{\omega} \inducedsubgraph G$ (and hence $\R \subgraph G$): by \thref{easycorollary}, in both cases we obtain $\isg_G\sweiequiv \wf$.
	\end{proof}

	In the last years researchers have focused on the program connecting reverse mathematics and computable analysis via the framework of Weihrauch reducibility. The natural problem representing $\PiCA$ is $\parallelization{\wf}$: this problem has been first considered by Hirst in \cite{leafmanaegement} and by the second author, Kihara and Marcone in \cite{kihara_marcone_pauly_2020}. A systematic study of other representatives of $\PiCA$ has been carried out by the first author, Marcone and Valenti in \cite{CB} and the following corollary gives a  \quot{graph-theoretic}\ one .
	\begin{corollary}
	   Let $G$ and $H$ be infinite c.e.\ graphs such that $\R\subgraph G$ and $\complete{\omega}\not\inducedsubgraph H$. Then
	   $\parallelization{\wf} \sweiequiv \parallelization{(e)\sg_G}   \sweiequiv \parallelization{(e)\isg_G}\sweiequiv \parallelization{(e)\isg_H}$.
	\end{corollary}

	The following theorem shows that the situation for $\sg_G$ is quite different. 
	\begin{theorem}
 \thlabel{sgandlpon}
		For every $k \in \nats$, $(e)\sg_{\mathsf{T}_{2k+1}} \sweiequiv \jump{\lpo}{2k}$ and $(e)\sg_{\mathsf{F}_{2k+2}} \sweiequiv \jump{\lpo}{2k+1}$
	\end{theorem}
	\begin{proof}
		It follows from \thref{subgraphcomplexity} and the fact that $\jump{\lpo}{n}$ answers a binary $\Sigma_{n+1}^0$-complete or $\Pi_{n+1}^0$-complete question (see \S \ref{subsec:Weihrauch}).
	\end{proof}

   We do not know whether there exists a graph $G$ such that $\jump{\lpo}{n}\strictlyweireducible \sg_{G} \strictlyweireducible\jump{\lpo}{n+1}$: the answer to this question is clearly related to the one we left open at the end of \S \ref{wadgecomplexityofgraphs}.
   \section{Searching the (induced) subgraph}
   \label{findsubgraphsandinduced}
	In this section, we focus entirely on Weihrauch reducibility: notice that some results are still consequences of \S \ref{wadgecomplexityofgraphs}, but most of them require new proof techniques.
	\begin{definition}
		\thlabel{definitionisg}
	   Given a graph $G$, we define four multi-valued functions:
	   \[\partialmultifunction{\findisCC{G}}{\repspacegraphs}{\repspacegraphs}, \ \partialmultifunction{\findisCE{G}}{\repspacegraphs}{\repspaceegraphs},\]
	   \[\partialmultifunction{\findisEE{G}}{\repspaceegraphs}{\repspaceegraphs}, \ \partialmultifunction{\findisEC{G}}{\repspaceegraphs}{\repspacegraphs},\]
	   with domains $\{H:G \inducedsubgraph H\}$ and ranges $\{G': G' \cong G \land G' \text{ is an induced subgraph of } H\}$.
	   The corresponding functions for the subgraph case are denoted replacing $\findis{}{}$ with $\finds{}{}$ and have domains $\{H:G \subgraph H\}$ and ranges $\{G': G' \cong G \land G' \text{ is a subgraph of } H\}$.
	\end{definition}
   
   Informally, in the definition above, the $\chi$ (respectively $e$) at the first (second) position indicates that the domain (range) of the corresponding function is a subset of $\repspacegraphs$ ($\repspaceegraphs$). We adopt the following convention: whenever we write $\findis{}{G}$ we are referring to all four functions in \thref{definitionisg}, similarly for $\finds{}{G}$.

	The following relations hold, also replacing $\findis{}{}$ with $\finds{}{}$.
	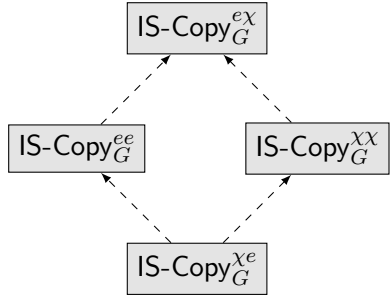
\begin{figure}[H]
	\centering
	 \tikzstyle{every picture}=[tikzfig]
	 \begin{tikzpicture}[scale=0.8]

		\begin{pgfonlayer}{nodelayer}
		\node [style=box] (cte) at (0,0) {$\findisCE{G}$};
		\node [style=box] (ete) at (-4,4) {$\findisEE{G}$ };
		\node [style=box] (ctc) at (4,4) {$\findisCC{G}$ };
		\node [style=box] (etc) at (0,8) {$\findisEC{G}$ };
	
		\end{pgfonlayer}
		\begin{pgfonlayer}{edgelayer}
			\draw [style=reducible] (cte) to (ete);
			\draw [style=reducible] (cte) to (ctc);
			\draw [style=reducible] (ctc) to (etc);
			\draw [style=reducible] (ete) to (etc);

		\end{pgfonlayer}
	\end{tikzpicture}
			\caption{Dashed arrows represent Weihrauch reducibility in the direction of the arrow, leaving open whether the reduction is strict.}
		\label{findisdefinitionsFigure}
   
	\end{figure}
   
   The next proposition shows that the main functions we consider in this section are cylinders, which implies that most reductions we obtain in this section are actually strong ones.
	\begin{proposition}
	   \thlabel{cylindersgraphone}
	   For any infinite graph $G$, $\finds{}{G}$ and  $\findis{}{G}$ are cylinders
	\end{proposition}
	\begin{proof}
	   We only prove the statement for $\findis{}{G}$ as the same proof shows that $\finds{}{G}$ are cylinders as well.
   Let $p \in \Baire$ and let $H \in \dom(\finds{}{G})$: we compute the graph $H'$ such that $V(H')\defas \{\str{v,p[v]}:v \in V(H)\}$ and $E(H')\defas \{(\str{v,p[v]},\str{w,p[w]}): (v,w) \in E(H)\}$.  Clearly $H'\cong H$ and the isomorphism between the two graph is computable. From $G' \in \findis{}{G}(H')$ compute the graph $S$ where $V(S)=\{v: \str{v,p[v]} \in V(G')\}$  and $E(S)=\{(v,w): (\str{v,p[v]},\str{w,p[w]}) \in E(G')\}$. Clearly  $S \in \findis{}{G}(H)$, and from $\{p[v]: \str{v,p[v]} \in V(S)\}$ (that by hypothesis is infinite) we can recover longer and longer prefixes of $p$. This concludes the proof
	\end{proof}

	The following proposition, together with Figure \ref{findisdefinitionsFigure}, shows that $\CBaire$ is an upper bound for all of them (notice that $\CBaire$ is a cylinder as well).
		\begin{proposition}
			\thlabel{findinducedsubgraphupperbound}
			For any infinite hyperarithmetical graph $G$, $\findis{}{G},\finds{}{G} \weireducible \CBaire$.
		\end{proposition}
		\begin{proof}
			By Figure \ref{findisdefinitionsFigure}, it suffices to show that $\findisEC{G},\findsEC{G}\weireducible \CBaire$.
		We only show that $\findisEC{G}\weireducible \CBaire$ as the same proof works also for showing that $\findsCE{G} \weireducible \CBaire$. Since $\choice{\boldfaceSigma_1^1}{\Baire}\weiequiv \CBaire$ (see \S \ref{subsec:Weihrauch}) we show that  $\findisEC{G}\weireducible \choice{\boldfaceSigma_1^1}{\Baire}$. Given in input $H$, by \thref{upperboundandfinite}, notice that $\{G' \in \repspaceegraphs : G' \cong G \land G' \text{ is an induced subgraph of } H\}$ is a nonempty $\Sigma_1^{1}$ (relative to $H$) subset of $\Baire$ and hence a suitable input for $\choice{\boldfaceSigma_1^1}{\Baire}$: clearly any  $G \in \choice{\boldfaceSigma_1^1}{\Baire}(H)$ is such that $G \in \findisEC{G}(H)$.
		\end{proof}

		We first consider the case of $G$ being finite, and notice that the case distinction, in some sense, is close to the one in \thref{upperboundandfinite}.
	\begin{theorem}
		\thlabel{findisthesameforfinite}
		For any finite graph $G$ the following holds:
		\begin{enumerate}[$(i)$]
			\item the problems $\finds{}{G}$ are computable;
			\item $\findisCE{G}$ and $\findisCC{G}$ are computable;
			\item $\findisEE{\complete{n}}$ and  $\findisEC{\complete{n}}$ are computable;
			\item if $G \not\cong \complete{n}$, $\findisEE{G}\sweiequiv \findisEC{G}\sweiequiv \cnats$.
		\end{enumerate}
	\end{theorem}
	\begin{proof}
	To prove $(i)$, by Figure \ref{findisdefinitionsFigure} it suffices to show that $\findsEC{G}$ is computable.  Given a $\repmap{\egraph}$-name $h$ for an input of $\findsEC{G}$ we compute a $\repmap{\graph}$-name $p$ for a solution of $\findsEC{G}$ as follows. At any stage $s$, we check  whether there exists a subgraph isomorphic to $G$ in the finite graph determined by $h[s]$. Formally, for every $s>0$, we check whether there exists an injective $\function{f}{V(G)}{V(\repmap{\egraph}(h[s]h(0)h(0)\dots))}$ such that
	\begin{itemize}
		\item $i \in V(G) \implies (\exists k_i<s)(h(k_i)=\str{f(i),f(i)})$ and
		\item  $(i,j) \in E(G)\implies (\exists \ell_i<s)(h(\ell_i)=\str{f(i),f(j)})$.
	\end{itemize}
	This can be done computably as both $\repmap{\egraph}(h[s]h(0)^\nats)$ and $G$ are finite and hence, at any stage, there are only finitely many $f$'s to check. If $G \not\subgraph \repmap{\egraph}(h[s]h(0)^\nats)$ do nothing. Otherwise, let $\function{\bar{f}}{V(G)}{V(\repmap{\egraph}(h[s]h(0)^\nats))}$ be a function witnessing $G\subgraph \repmap{\egraph}(h[s]h(0)^\nats)$ and let $p$ be such that 
		\[p(\str{i,j})\defas\begin{cases}
			1 & \text{if } (i=j \land \bar{f}^{-1}(i)\in V(G)) \lor ((\bar{f}^{-1}(i),\bar{f}^{-1}(j)) \in E(G)),\\
			0 & \text{otherwise.}
		\end{cases}\]
	It is straightforward to check that $p$ is a $\repmap{\graph}$-name for a solution of $\findsEC{G}$.
   
		For $(ii)$, by Figure \ref{findisdefinitionsFigure}, it suffices to prove that $\findisCC{G}$ is computable. The proof is the same of $(i)$, except for the fact that the injective $f$ is such that $i \in V(G)$ implies $h(\str{f(i),f(i)})=1$ and $(i,j) \in E(G)$ if and only if $h(\str{f(i),f(j)})=1$.
   
		The claim in $(iii)$ follows from $(i)$ as, for any graph $H$, $\complete{n}\subgraph H$ if and only if $\complete{n} \inducedsubgraph H$.
   
		To prove $(iv)$, by Figure \ref{findisdefinitionsFigure}, it suffices to show that $\findisEC{G} \sweireducible \cnats$ and $\cnats \sweireducible \findisEE{G}$. For the first reduction, given a name $p$ for an input $H$ of $\findisEC{G}$, let $A$ be the set of $(\sigma,\tau) \in (\nats\times\nats)^{V(G)}$ such that $(\forall i, j \in V(G))(\forall z \in \nats)$
		\begin{itemize}
			\item $p(\str{\tau(i),\tau(i)})=\str{\sigma(i),\sigma(i)}$;
			\item $(i,j) \in E(G) \implies p(\str{\tau(i),\tau(j)})= \str{\sigma(i),\sigma(j)}$ and
			\item $(i,j)\notin E(G) \implies p(z) \neq \str{\sigma(i),\sigma(j)}$.
		\end{itemize}
	Clearly, $A$ is nonempty, and since any $(\sigma,\tau)$ can be coded as a natural number, $A$ is a suitable input for $\cnats$. Let $n \in \cnats(A)$ be the code of some $(\sigma,\tau)$. We compute a name $q$ for a copy of $G' \in \findisEC{G}(H)$ letting $q(\str{\sigma(i),\sigma(j)})=1$ for all $i,j <\length{\sigma}$.
   
		To prove that $\cnats \sweireducible \findisEE{G}$, let $A \in \mathcal{A}(\nats)$ be an input of $\cnats$. We denote by $A^c[s]$ the enumeration of the complement of $A$ up to stage $s$ and notice that, since by hypothesis, $G\not\cong \complete{n}$, we have that $V(G)^2\setminus E(G) \neq \emptyset$. We compute an input $H$ for $\findisEE{G}$ as follows. At stage $0$, for every $v \in V(G)$ and for every $(v,w) \in E(G)$, enumerate $\str{v,0}$ in $V(H)$ and $(\str{v,0},\str{w,0})$ in $E(H)$ and let $n_0\defas 0$. At stage $s+1$, let $n_{s+1}\defas \min\{n: n \notin A^c[s+1]\}$. If $n_{s+1}=n_{s}$, do nothing, otherwise, for every $v,w \in V(H)$ at stage $s+1$, enumerate $(v,w)$ in $E(H)$, i.e.\ modify $H$ so that is isomorphic to a complete graph. Then for every $v \in V(G)$ and for every $(v,w) \in E(G)$, enumerate $\str{v,n_{s+1}}$ in $V(H)$ and $(\str{v,n_{s+1}},\str{w,n_{s+1}})$ in $E(H)$. In the limit, we obtain that either $H$ is such that $V(H)=\{\str{v,0}:v \in V(G)\}$ and $E(H)=\{(\str{v,0}, \str{w,0}) : (v,w) \in E(G)\}$ (in case $0 \in A$), or $H \cong \complete{m} \disconnectedunion{} G'$ for some $m \in \nats$ and  $V(G')=\{\str{v,n_s}:v \in V(G)\}$ and $E(G')=\{(\str{v,n_s}, \str{w,n_s}) : (v,w) \in E(G)\}$ where $n_s = \min \{n : n \in A\}$. Since $G \not\cong \complete{n}$, we know that any $H' \in \findisEE{G}(H)$ is not contained in the copy of $\complete{m}$ in $H$	and hence, given $\str{v,n_s} \in V(H')$, it is easy to check that $n_s \in A$ and this concludes the proof.
	\end{proof}
   
	\subsection{The induced subgraph problem for infinite graphs}
	We now move our attention to the case when $G$ is infinite. The results are summarized in the next two theorems: the first concerns the problems $\findis{}{G}$, while the second the problems $\finds{}{G}$. The first theorem analyzes $\findis{}{G}$ for all infinite $G$'s, while the second one does not include certain graphs some of which are discussed in \S \ref{subgraphsubsection}. We highlight that some results are stated for c.e.\ graphs,  others just for computable graphs.

   \begin{theorem}
	   \thlabel{maintheoreminducedsubgraph}
	   For any infinite graph $G$, $\findis{}{G} \weiequiv \CBaire$ relative to some oracle. In particular, we have the following cases:
	   \begin{enumerate}[$(i)$]
		   \item if $|\{v \in V(G):\D{G}{v}<\aleph_0\}|<\aleph_0$ we have two cases:
		   \begin{itemize}
			   \item[(a)] if $G$ is computable then $\CBaire \weiequiv \findis{}{G}$; 
			   \item[(b)] if $G$ is c.e., then $\CBaire \weiequiv \findisEE{G} \weiequiv \findisEC{G}$.
		   \end{itemize}
		   \item  If $G$ is c.e.\ and $|\{v \in V(G):\D{G}{v}=\aleph_0\}|<\aleph_0$, then 
		   $\CBaire \weiequiv \findis{}{G}$.
		   \item Let $\function{\lambda}{\nats}{\nats}$ be such that $\lambda(n) \defas \min\{k:|v \in V(G):\D{G}{v} \leq k\}|\geq k$. If $\{v \in V(G): \D{G}{v}<\aleph_0\}=\aleph_0$ and $\{v \in V(G): \D{G}{v}=\aleph_0\}=\aleph_0$ we have two cases:
		   \begin{itemize}
			   \item[(a)] if $G$ is computable then $\CBaire \weiequiv \findis{}{G}$ relative to $\lambda$; 
			   \item[(b)] if $G$ is c.e., then $\CBaire \weiequiv \findisEE{G} \weiequiv \findisEC{G}$ relative to $\lambda$.
		   \end{itemize}
	   \end{enumerate}
   \end{theorem}
   
   \begin{theorem}
	   \thlabel{maintheoermsubgraphcase}
	   Let $G$ be an infinite graph such that $|\{v \in V(G):\D{G}{v}<\aleph_0\}|<\aleph_0$. If $G$ is computable then $\CBaire \weiequiv \finds{}{G}$ while if $G$ is c.e., then $\CBaire \weiequiv \findsEE{G}\weiequiv \findsEC{G}$. 
   \end{theorem}
   
   The next part of this subsection is devoted to prove \thref{maintheoreminducedsubgraph,maintheoermsubgraphcase}: \thref{findinducedsubgraphupperbound} implies all the reductions from $\findis{}{G}$  to $\CBaire$: Figure \ref{findisdefinitionsFigure} and the remaining lemmas of this subsection prove the converse directions.

   The following lemma proves \thref{maintheoreminducedsubgraph}$(i)$ and \thref{maintheoermsubgraphcase}.
		\begin{lemma}
			\thlabel{findinducedsubgraphfinManyVertFinManyNeigh}
			Let $G$ be an infinite graph such that $|\{v \in V(G):\D{G}{v}<\aleph_0\}|<\aleph_0$. If $G$ is computable then $\CBaire \weireducible \finds{}{G}$, $\findis{}{G}$. If $G$ is c.e.\ then $ \CBaire\weireducible \findisEE{G}, \findsEE{G}$.
		\end{lemma}
		\begin{proof}
			For the first part, by Figure \ref{findisdefinitionsFigure}, it suffices to show that $\CBaire \weireducible \findsCE{G}$, $\findisCE{G}$. We only prove the claim for $\findsCE{G}$ as the proof for $\findisCE{G}$ is the same. Let $T \in \tree$ be an input for $\CBaire$, let $V(G)=\{v_i : i \in \nats\}$ and compute $\constructionone{T,G}$ (see \S \ref{wadgecomplexityofgraphs} for its definition). Notice that, since $G$ is computable, $\constructionone{T,G}$ is computable with respect to $T$ and by \thref{Constructionongraphs}, $G \subgraph \constructionone{T,G}$: hence $\constructionone{T,G}$ is a suitable input for $\findsCE{G}$. Let $G' \in \findsCE{G}(\constructionone{T,G})$. We claim that
			$$(\forall \sigma \in V(G'))(\D{G'}{\sigma}=\aleph_0 \implies (\exists \tau \in V(G'))(\D{G'}{\tau}=\aleph_0 \land \sigma \sqsubset \tau)).$$
			To prove the claim notice that, by hypothesis, $(\forall^\infty \sigma \in V(G'))(\D{G'}{\sigma}=\aleph_0)$. Hence, given $\sigma$ such that $\D{G'}{\sigma}=\aleph_0$,  there exists another vertex $\tau \in V(G')$ such that $(\sigma,\tau) \in E(G')$ and $\D{G'}{\tau}=\aleph_0$: the definition of $E(\constructionone{T,G})$ implies that $\sigma\sqsubset \tau$. 
	   
			Let $N\defas\max \{\D{G}{v}:v \in A\}$. We compute a sequence $\{\sigma_s : s \in \nats\} \subseteq V(G')$ such that $\bigcup_s \sigma_s \in \body{T}$. At stage $0$, let $\sigma_0$ be the first vertex enumerated by the name of $G'$ satisfying $\D{G'}{\sigma_0}=\aleph_0$ (which exists by the previous claim): this is a computable process as it suffices to verify $\D{G'}{\sigma_0}> N$. Suppose we have computed the sequence up to $\sigma_{s}$. At stage $s+1$, let $\sigma_{s+1}$ be the first vertex enumerated by the name of $G'$, satisfying $\sigma_s\sqsubset \sigma_{s+1} \land \D{G'}{\sigma_{s+1}}>N$ (the existence of $\sigma_{s+1}$ is guaranteed by the previous claim). This proves when $G$ is computable.
   
			To show that if $G$ is c.e. then $\CBaire\weireducible \findisEE{G}, \findsEE{G}$ notice that the proof is exactly the same. In this case, $\constructionone{T, G}$ is c.e.\ with respect to $T$, but this is fine as the input for $\findisEE{G}$ and $\findisEC{G}$ is in $\repspaceegraphs$.
		\end{proof}

		The following lemma proves \thref{maintheoreminducedsubgraph}$(ii)$.
		\begin{lemma}
		\thlabel{finmanywithinfmany}
	Let $G$ be an infinite graph such that $\{v \in V(G):\D{G}{v}=\aleph_0\}|<\aleph_0$. If $G$ is c.e., then $\CBaire \weireducible \findis{}{G}$.	
	\end{lemma}
	\begin{proof}
		By Figure \ref{findisdefinitionsFigure}, it suffices to show that $\CBaire \weireducible \findisCE{G}$. Let  $T\in \tree$ be an input for $\CBaire$, and compute $G' \in \mathbf{F}(G)$ where $\mathbf{F}$ is the function of \thref{fromcetocomputable}. Let $\constructiontwo{T,G'}$ (see \S \ref{wadgecomplexityofgraphs} for its definition): since $G' \in \repspacegraphs$, $T \in \illfounded$ and $G \inducedsubgraph G'$ (\thref{fromcetocomputable}), by \thref{Constructionongraphs}, $G\inducedsubgraph\constructiontwo{T,G'}$ (i.e.\ $\constructiontwo{T,G'}$ is a suitable input for $\findisCE{G}$). Let $H \in \findisCE{G}(\constructiontwo{T,G'})$.
		  
	Notice that $(\forall^\infty \sigma, \tau \in V(H))(\sigma \sqsubset \tau \lor \tau \sqsubset \sigma)$. If not then there exists $\{\sigma_i : i \in \nats\}\subseteq V(H)$ such that $(\forall i \neq j)(\sigma_i \incomparable \sigma_j)$. By definition of $\constructiontwo{T,G'}$ and the fact that $H$ is an induced subgraph of $\constructiontwo{T,G'}$, we obtain that $(\forall i\neq j)((\sigma_i,\sigma_j) \in E(H))$, and hence $(\forall i)(\D{H}{\sigma_i}=\aleph_0)$.  Since $H\cong G$, this contradicts the hypothesis that $|\{v \in V(G):\D{G}{v}=\aleph_0\}|<\aleph_0$. In other words, we have just showed that $(\exists f \in \body{T})(\forall^\infty \sigma \in V(H))(\sigma\sqsubset f)$.
	   We now show that:
	   \begin{equation}
	   \label{condition}
		   (\forall \sigma \in V(H))(\D{H}{\sigma}<\aleph_0 \implies \sigma \sqsubset f).
	   \end{equation}
	  Otherwise, $(\exists \tau \in V(H))(\D{H}{\tau}<\aleph_0 \land \tau\not\sqsubset f)$ but we have just shown that $(\forall^\infty \sigma \in V(H))(\sigma \sqsubset f)$ and hence $(\forall^\infty \sigma \in V(H))(\sigma \incomparable \tau)$. By definition of $\constructionone{T, G'}$ and the fact that $H$ is an induced subgraph of $\constructionone{T, G'}$ we obtain that $\D{H}{\tau}=\aleph_0$, getting the desired contradiction.
	  
	   Now we compute a sequence of vertices $\{\sigma_i: i \in \nats\}\subseteq V(H)$ such that  $\bigcup_i\sigma_i= f$. Let $N\defas |\{\sigma \in V(G):\D{G}{\sigma}=\aleph_0\}|$. For every $s$, let $\sigma_s \in V(H)$ be such that $(\exists \tau_0,\dots,\tau_{N} \in V(H))(\forall i\leq N)(\sigma_s \sqsubset \tau_i)$ and $(\exists! \tau'_0,\dots,\tau'_{s-1} \in V(H))(\forall i< s)(\tau'_i\sqsubset \sigma_s)$ (the second condition ensures that $\length{\sigma_s}\geq s$). For any $s$ the existence of $\sigma_s$ is guaranteed by (\ref{condition}): just let $\sigma_s \in \{\sigma \in V(H) : \D{H}{\sigma}<\aleph_0 \land \length{\sigma}\geq s \}$. It remains to show that for every $s$, $\sigma_s\sqsubset f$. To do so, notice that any $\sigma_s$ has (at least) $N+1$ many extensions $\tau_0,\dots,\tau_N$ in $T$. By hypothesis there are only $N$ many vertices of infinite degree in $V(H)$, hence there exists an $i<N$ such that $\D{H}{\tau_i}<\aleph_0$, i.e.\ $\sigma_s \sqsubset \tau_i\sqsubset f$, and this concludes the proof.
	   \end{proof}
	   
	Notice that \thref{findinducedsubgraphfinManyVertFinManyNeigh} and \thref{finmanywithinfmany} do not exhaust all the possible cases: it may be the case that $G$ is such that $\{v \in V(G): \D{G}{v}<\aleph_0\}=\aleph_0$ (as in \thref{finmanywithinfmany}) but $\{v \in V(G): \D{G}{v}=\aleph_0\}=\aleph_0$ too. The following lemma shows \thref{maintheoreminducedsubgraph}($iii$) and concludes the proof of \thref{maintheoreminducedsubgraph}. The proof of the next lemma is similar to the one of \thref{finmanywithinfmany}: the main difference is that here $ |\{\sigma \in V(G):\D{G}{\sigma}=\aleph_0\}| \notin \nats$ and hence the reduction is given relative to an oracle.
   
		\begin{lemma}
		\thlabel{oracle}
		Let $\function{\lambda}{\nats}{\nats}$ be such that $\lambda(n) \defas \min\{k:|v \in V(G):\D{G}{v} \leq k|\}\geq k$ and let $G$ be an infinite graph such that
				\[\{v \in V(G): \D{G}{v}<\aleph_0\}=\aleph_0 \text{ and } \{v \in V(G): \D{G}{v}=\aleph_0\}=\aleph_0.\]
				If $G$ is computable then $\CBaire \weireducible \finds{}{G}$, $\findis{}{G}$ relative to $\lambda$. If $G$ is c.e.\ then $ \CBaire\weireducible \findisEE{G}, \findsEE{G}$ relative to $\lambda$.
		\end{lemma}
		\begin{proof}
			By Figure \ref{findisdefinitionsFigure}, it suffices to show that $\CBaire \weireducible \findisCE{G}$. Let $T \in \tree$ be an input for $\CBaire$, assume $V(G)=\{ \sigma_i : i \in \nats\}$ and compute $\constructiontwo{T,G}$. Let $G'\in \findisCE{G}(\constructiontwo{T,G})$ and notice that the same proof of \thref{finmanywithinfmany} gives us that cofinitely vertices in $G'$ belong to the same path $f \in \body{T}$ and
	   \begin{equation}
	   \label{condition2}
		   (\forall \sigma \in V(G'))(\D{G'}{\sigma}<\aleph_0 \implies \sigma \sqsubset f).
	   \end{equation}
	  
		Now we compute a sequence of vertices $\{\sigma_s: s \in \nats\}\subseteq V(G')$ such that  $\bigcup_s\sigma_s= f$. For any $s$, let $\sigma_s \in V(G')$ be such that:
	\begin{enumerate}[$(i)$]
		\item $(\exists \tau_0,\dots,\tau_{\lambda(s+1)} \in V(G'))(\forall i\leq \lambda(s+1))(\sigma_s \sqsubset \tau_i)$;
		\item if $s>0$, $(\exists! \tau'_0,\dots,\tau'_{s-1} \in V(G'))(\forall i< s-1)(\tau'_i\sqsubset \tau'_{i+1} \sqsubset \sigma_s)$.
	\end{enumerate}
	Condition $(ii)$ ensures that $\length{\sigma_s}\geq s$ and for any $s$ the existence of $\sigma_s$ is guaranteed by (\ref{condition2}): just let $\sigma_s \in \{\sigma \in V(G') : \D{G'}{\sigma}<\aleph_0 \land \length{\sigma}\geq s\}$. It remains to show that for every $s$, $\sigma_s\sqsubset f$. To do so, notice that by hypothesis $(\exists v_0,\dots, v_s \in V(G))(\D{G}{v_s}\leq \lambda(s+1))$. Any isomorphism from $G$ to $G'$, for any $i\leq s$, must map $v_i$ to some $\tau \in V(G')$ such that $\tau \sqsubseteq \sigma_s \lor \sigma_s \sqsubseteq \tau$. Suppose it is not the case: if $\tau\incomparable \sigma_s$ then, since $(\forall i< \lambda(s+1))(\sigma_s \sqsubset \tau_i)$, we have that $(\forall i< \lambda(s+1))(\tau \incomparable \tau_i)$ and so $\D{G'}{\tau}\geq \lambda(s+1)+1$. Hence, any isomorphism from $G$ to $G'$ maps one between $v_0,\dots, v_s$ in some $\tau \in V(G')$ such that $\tau \sqsubseteq \sigma_s$ (this is guaranteed by $(ii)$): indeed, there are only $s$ many vertices in $G'$ that are prefixes of $\sigma_s$, hence $\sigma_s\sqsubset f$ and this concludes the proof.
   
	To show that if $G$ is c.e. then $\CBaire\weireducible \findisEE{G}, \findsEE{G}$ notice that the proof is exactly the same. In this case, $\constructiontwo{T, G}$ is c.e.\ with respect to $T$, but this is fine as the input for $\findisEE{G}$ and $\findisEC{G}$ is in $\repspaceegraphs$.
	   \end{proof}
   
	Notice that for some graphs, the $\lambda$ defined in the previous theorem is computable. For example, this is the case for \emph{highly recursive} graphs, i.e.\ graphs in which for every $v \in V(G)$, we can compute $\D{G}{v}$. These particular graphs have been considered, for different problems, for example in \cite{manaster1972effective}. We also mention that \thref{oracle} holds even if we replace $\lambda$ with any $\gamma$ bounding $\lambda$. 
   
	We leave open whether it is possible to \quot{get rid of}\ the oracle in \thref{oracle}, obtaining a Weihrauch reduction like in \thref{findinducedsubgraphfinManyVertFinManyNeigh,finmanywithinfmany} that would give us the result that for any computable/c.e.\ graph $G$, $\findis{}{G}\weiequiv \CBaire$.
   
   \subsection{The subgraph problem: when $\R\subgraph G$}
   \label{subgraphsubsection}
	\thref{maintheoermsubgraphcase} shows that $\finds{}{G}\weiequiv \CBaire$ where $G$ is an infinite c.e.\ graph such that $|\{v \in V(G): \D{G}{v}<\aleph_0\}|<\aleph_0$. To study the Weihrauch degree of $\finds{}{G}$ for c.e.\ graphs not satisfying such a condition, and in particular to understand which of them satisfy $\finds{}{G}\weiequiv \CBaire$, we start from those graphs that, intuitively, are  \quot{ill-founded}, i.e.\ those graphs $G$ such that $\R \subgraph G$. A piece of evidence supporting the claim that $\finds{}{G}\weiequiv \CBaire$ for any $G$ such that $\R\subgraph G$ is \thref{lemma1question}: in that case, the fact that $\R\subgraph G$ implied that $(e)\sg_G \weiequiv \wf$. In this section we show this intuition is not entirely true: before doing so, we define the following multi-valued functions.
	
	\begin{definition}
		Let $G$ be a graph such that $\R\subgraph G$. The multi-valued functions $\partialmultifunction{\embeddingray{G}}{\repspacegraphs}{\Baire}$ and $\partialmultifunction{\eembeddingray{G}}{\repspaceegraphs}{\Baire}$ with domains respectively $\{H \in \repspacegraphs:  G \cong H\}$ and $\{H \in \repspaceegraphs:  G \cong H\}$ are defined as
	$$(e)\embeddingray{G}(H) \defas \{p: (\forall i)((p(i),p(i+1)) \in E(H))\}.$$
	  \end{definition}
	  
	  The next proposition implies that most reductions we obtain in this section are actually strong ones: the proof is omitted as it is similar to the one proving \thref{cylindersgraphone}.
	  \begin{proposition}
		 \thlabel{cylindersgraphtwo}
		  For any graph $G$, $\embeddingray{G}$ and $\eembeddingray{G}$ are cylinders.
	  \end{proposition}
   
	 \begin{lemma}
	\thlabel{embeddabilityCBaire}
	Given a hyperarithmetical graph $G$ such that $\R \subgraph G$, $\embeddingray{G} \weireducible \eembeddingray{G} \weireducible \CBaire$.
	\end{lemma}
	\begin{proof}
	The first reduction directly follows from the fact that from the characteristic function of a graph we can compute an enumeration of it, while for the second one it suffices to notice that if $H \in \dom(e\embeddingray{G})$ then $A\defas\{p: (\forall i)((p(i),p(i+1)) \in E(H))\}\in \dom(\CBaire)$ and so any $p \in \CBaire(A)$ is a solution for $e\embeddingray{G}(H)$.
	\end{proof}
   
		The following proposition tells us that $\finds{}{G}$ composed with $(e)\embeddingray{G}$ computes $\CBaire$.

   \begin{proposition}
		\thlabel{cbairecomputedbyisol}
		Given a c.e.\ graph $G$ such that $\R \subgraph G$,
		\[ \eembeddingray{G}*\findsEE{\R}\weiequiv  \embeddingray{G}*\findsEC{G} \weiequiv \CBaire.\]
		Given a computable graph $G$ such that $\R \subgraph G$,
		\[\eembeddingray{G}*\findsCE{\R} \weiequiv \embeddingray{G}*\findsCC{\R} \weiequiv \CBaire.\]
   
	\end{proposition}
	\begin{proof}
		We only prove that $\eembeddingray{G}*\findsCE{\R} \weiequiv \CBaire$ as the same proof works also for the other equivalences. \thref{findinducedsubgraphupperbound,embeddabilityCBaire} and the fact that $\CBaire$ is closed under compositional product imply that $e\embeddingray{G} *\findisCE{G} \weireducible \CBaire$.
	   
		For the opposite direction, given an input $T $ for $\CBaire$, compute $\constructionone{T,G}$ and notice that since $T \in \illfounded$, $\constructionone{T,G} \in \dom(\findisCE{G})$ (\thref{Constructionongraphs}). Let $G' \in \findisCE{G}(\constructionone{T,G})$ and let $p \in e\embeddingray{G}(G')$. Notice that $i_0 \defas \min \{i : (\forall j \geq i )(p[j]\sqsubset p[j+1]) \}$ exists and it is computable: clearly $\bigcup_{i > i_0}p[i] \in \body{T}$.
	\end{proof}
	By \thref{cbairecomputedbyisol}, in order to show that a graph $G$ is such that $\finds{}{G}\weiequiv \CBaire$, it suffices to show that $(e)\embeddingray{G}$ is computable. The next result gives some examples of graphs satisfying this condition; for the next proposition, $\cantor$ denotes the full binary tree $\{\sigma: \sigma \in \cantor\}$.
	\begin{proposition}
		\thlabel{computableembedding}
		 Let $n>2$ and $m>0$: if $G \in \{\elle, \cycle{n} \connectedunion{}\R, \complete{m}\connectedunion{} \R, \cantor\}$ then $\finds{}{G}\weiequiv \CBaire$. 
		 \end{proposition}
		 \begin{proof}
		   By \thref{embeddabilityCBaire,cbairecomputedbyisol}, it suffices to show that $\eembeddingray{G}$ is computable.
			Let $H$ be an in input for $\eembeddingray{\elle}$, and notice that for any $v \in V(H)$, $\D{H}{v}=2$. To compute $p \in \eembeddingray{\elle}(L)$ let $p(0)$ be such that $p(0) \in V(H)$ and for every $i>0$ choose $p(i)$ such $(p(i),p(i+1)) \in E(H)$ since, at each stage, $p(i)$ exists and is unique, this shows that $\eembeddingray{\elle}$ is computable.
		   
			Let $H$  be an in input for $\eembeddingray{(\cycle{n} \connectedunion{} \R)}$. We compute $p \in \eembeddingray{(\cycle{n} \connectedunion{} \R)}$ as follows. Wait for some finite stage witnessing that $C_n \connectedunion{} \ray{1}\subgraph H$ and denote by $H'$ the copy of $C_n \connectedunion{} \ray{1}$ in $H$. Then let $p(0)$ be the unique $v \in V(H')$ such that $\D{H'}{v}=1$: clearly, all vertices in $H\setminus H'$ \lq\lq continue\rq\rq\ as a copy of $\R$, hence for $s>0$ just let $p(s)$ be such that $(p(s-1),p(s)) \in E(H\setminus H')$. A similar proof holds for $\complete{m} \connectedunion{}  \R$.
		   
			Let $H$ be an input for $\eembeddingray{\cantor}$. We compute $p \in \eembeddingray{\cantor}(H)$ as follows. Let $p(0)$ be any node in $V(H)$ and for any $s$ let $p(s+1)\defas \sigma$ where $\length{\sigma}=s$ and $\sigma \sqsupset p[s-1]$: by definition of $\cantor$, we can always find $\sigma$ as such and this concludes the proof.
			\end{proof}
		
			\subsection{A particular case:  $\finds{}{\R}$}
		   \label{particularcase}
		   It is natural to ask whether \thref{computableembedding} holds for $\R$ and so,  by \thref{embeddabilityCBaire}, that $\finds{}{\R}\weiequiv \CBaire$. This result would to be coherent with the definition of $\finds{}{\R}$ as its task is very similar to the one of $\CBaire$.

		   this would  together with  the following proposition shows that   $(e)\embeddingray{\R}$ is not computable, and hence we cannot apply the same strategy used in the proof of \thref{computableembedding}.
	   
			\begin{proposition}
			\thlabel{lim2}
			$\mflim_2 \weiequiv \embeddingray{\R}\weiequiv \eembeddingray{\R}$.
			\end{proposition}
			\begin{proof}
			\thref{embeddabilityCBaire} implies that $\embeddingray{\R}\weireducible \eembeddingray{\R}$. 
		   
		   We now show that $e\embeddingray{\R}\weireducible \mflim_2$. Let $G\cong \R$ be an input for $e\embeddingray{\R}$: we compute an input $q \in \Cantor$ for $\mflim_2$ in stages as follows. At stage $0$, let $v,w\in V(G)$ be such that $(v,w) \in E(G)$ and let $q(0)\defas 0$. At stage $s+1$, take some fresh $u_{s+1} \in V(G)$ (i.e.\ one that has not been considered yet) and let
			\[q(s+1)\defas \begin{cases}
			 0 & \text{if } \pathconnected{v}{u_{s+1}}{w}{G},\\
			 1 & \text{if }  \pathconnected{v}{u_{s+1}}{\lnot w}{G}.
			\end{cases}\]
			Informally, we are computing $q$ considering a copy $R$ of $\R$ starting from $v$ and checking whether $R$ continues \quot{after}\ $w$ (in which case $\mflim(q)=0$) or not (in which case $\mflim(q)=1$).
			 If $\mflim_2(q)=0$, we compute $p \in e\embeddingray{\R}(G)$ letting $p(0)\defas v$ and, for $i>0$, $p(i)\defas u_i$ where $u_i$ is the unique vertex connected to $v$, by a line segment of length $i$ passing via $w$: the fact that  $\mflim_2(q)=0$ implies that  $(\exists^\infty s)( \pathconnected{v}{u_s}{w}{G})$, and hence we can always find $u_i$. The case in which $\mflim(q)=1$ is held similarly letting $p(0)\defas v$ and, for $i>0$, $p(i)\defas u_i$ where $u_i$ is the unique vertex connected to $v$, by a line segment of length $i$ not passing via $w$.
	   
			To conclude the proof it suffices show that $\mflim_2\weireducible \embeddingray{\R}$. Let $q \in \Cantor$ be an input for $\mflim_2$. We compute an input for $\embeddingray{\R}$ in stages: if at stage $s$ $q(s)=0$, we extend the line segment computed so far to the left, otherwise to the right. More formally, we compute a name $g$ for an input of $\embeddingray{\R}$ as follows: at stage $0$ let $g(\str{0,0})=1$ and if $q(0)=0$ let $g(\str{2,2})=g(\str{0,2})=1$, otherwise let $g(\str{1,1})=g(\str{0,1})=1$. At stage $s+1$, if $q(s+1)=0$ let $g(\str{2s+2,2s+2})=1$ and $g(\str{2s+2,x})=1$ where $x\defas \max\{n = 2t+2 : t<s \land g(\str{n,n})=1\}$. Similarly, if $q(s+1)=1$ let $g(\str{2s+1,2s+1})=1$ and $g(\str{2s+1,x})=1 \in E(G)$ where $x\defas \max\{n = 2t+1 : t<s \land g(\str{n,n})=1\}$. Since $q$ converges either to $0$ or $1$, $\repmap{Gr}(g)\cong \R$ and so it is a suitable input for $\embeddingray{\R}$. Let $p \in \embeddingray{\R}(\repmap{\repspacegraphs}(g))$: 
				\[\mflim_2(p)=\begin{cases}
					 0&\text{if } (p(0)<p(1) \land p(1) \text{ is even}) \lor (p(0)>p(1) \land (p(1)\text{ is odd}\lor p(1)=0),\\
					 1&\text{if } (p(0)<p(1)\land p(1) \text{ is odd}) \lor (p(0)>p(1) \land p(1)\text{ is even}),
					\end{cases}\]
					and this concludes the proof.
				\end{proof}
				Combining the proposition above with \thref{cbairecomputedbyisol,lim2} we get the following corollary.
				\begin{corollary}
					\thlabel{corollarylim2}
						$\CBaire \weiequiv \mflim_2*\finds{}{\R}$. 
					\end{corollary}
				   
				   It is natural to ask whether we really need $\mflim_2$, i.e.\ does $\CBaire \weiequiv \finds{}{\R}$? The next proposition tells us that the first-order parts of the two problems coincide.
	\begin{proposition}
	\thlabel{propositionfoprayomega}
		$\firstOrderPart{\finds{}{\R}}\weiequiv \firstOrderPart{\CBaire} \weiequiv  \choice{\boldfaceSigma_1^1}{\nats}$.
	\end{proposition}
	\begin{proof}
		The fact that $ \firstOrderPart{\CBaire} \weiequiv  \choice{\boldfaceSigma_1^1}{\nats}$ is from \cite[Proposition 2.4]{pauly-valenti} and together with \thref{findinducedsubgraphupperbound} implies that $\firstOrderPart{\finds{}{\R}} \weireducible \firstOrderPart{\CBaire} \weiequiv  \choice{\boldfaceSigma_1^1}{\nats}$. To conclude the proof, by Figure \ref{findisdefinitionsFigure}, it suffices to show that $\choice{\boldfaceSigma_1^1}{\nats}\weireducible \findsCE{\R}$. We can think of an input for $\choice{\boldfaceSigma_1^1}{\nats}$ as a sequence $(T^i)_{i \in \nats} \in \tree^\nats$ such that  $(\exists i)(T^i \in \illfounded)$. Let $T\defas \disconnectedunion{i \in \nats} T^i$ (notice that $\otimes$ is the disconnected union of trees): since at least one $T^i \in \illfounded$, $S \in \dom(\findsCE{\R})$. Given $G \in \findsCE{\R}(S)$, we have that $G$ is a subgraph of $T^i$ where $T^i \in \illfounded$ and since any vertex in $V(G)$ is of the form $\pair{i,\sigma}$,  we can easily compute $i$.
	\end{proof}
   
	The following theorem surprisingly shows that the problems $\finds{}{\R}$ are significantly weaker than $\CBaire$.
   
	\begin{theorem}
	   \thlabel{maintheoremray}
	   $\finds{}{\R}\strictlyweireducible\CBaire$. In particular, 
	   \[(i) \ \findsCE{\R},\findsEE{\R}{\weiincomparable} \mflim, \ (ii) \ \findsCC{\R} \weiincomparable \mflim' \text{ and }(iii) \ \findsEC{\R} \weiincomparable \mflim''.\]
   \end{theorem}
   The proof of the theorem above is given throughout the remaining part of this subsection: notice that the right-to-left nonreductions are almost immediate. Indeed, it suffices to notice that $\finds{}{\R}\weiequiv \mflim_2*\CBaire$ (\thref{corollarylim2}), while, since for every $n$, $\mflim^{(n)}$ is closed under compositional product and $\mflim_2 \strictlyweireducible \mflim^{(n)}$ we have that $\mflim^{(n)}*\mflim_2 \weiequiv \mflim^{(n)} \strictlyweireducible \CBaire$. Hence, the remaining part of this subsection is devoted to prove the left-to-right nonreductions.
   
	\subsubsection*{\underline{Step 1} of \thref{maintheoremray}' proof: restricting to connected graphs}
	
	To prove the theorem above, it is convenient to consider $\connected \finds{}{\R}$, which is the same problem as $\finds{}{\R}$ with the domain restricted to connected graphs. The reductions in Figure \ref{definitionisg}, \thref{corollarylim2} and \thref{cylindersgraphone} hold also for this restricted version as summarized in the next proposition.
   
	\begin{proposition}
		\thlabel{summaryForConnectedness}
	   \
	   \begin{itemize}
		   \item For any graph $G$, $\connected \findsCE{G} \weireducible \connected\findsCC{G}, \connected\findsEE{G} \weireducible \connected \findsEC{G}$;
		   \item  $\CBaire \weiequiv \mflim_2*\connected\findsCE{\R}$; 
		   \item for any infinite graph $G$, $\connected \finds{}{G}$ are cylinders.
	   \end{itemize}
	\end{proposition}

	Regarding the connectedness of graphs, we consider the multi-valued function $\mathsf{D}$ introduced in \cite[\S 6]{gura2015existence}.
   
	\begin{definition}
		We define the multi-valued function $\multifunction{\mathsf{D}}{\repspacegraphs}{\Baire}$ defined as  $\mathsf{D}(G)\defas f$ where $\function{f}{V(G)}{\nats}$ is such that $(\forall v_1 \neq v_2 \in V(G))(f(v_1)=f(v_2) \iff\pathconnected{v_1}{v_2}{}{G})$.
		\end{definition}
   
		\begin{lemma}[{\cite[Theorem 6.4]{gura2015existence}}]
			\thlabel{findconcomponent}
			$\mathsf{D} \weiequiv \mflim$.
		\end{lemma}
   
   We first need the following technical lemma.
   
	\begin{lemma} 
		\thlabel{subgraphsigmacontinuous1}
		If $\mflim\weireducible f*g$ where $\multifunction{g}{\repspacex}{\nats}$, then $\lpo \weireducible f$ relative to some oracle.
		\end{lemma}
		\begin{proof}
		   Since $\mflim \weiequiv \mathsf{J}$, where $\mathsf{J}$ is the Turing jump operator (see \S \ref{subsec:Weihrauch}), suppose that $\mathsf{J}\weireducible f*g$. Let $B_n \subseteq X$ be the set of inputs to $g$ where $n \in \nats$ is a valid output. Let $\function{\Phi}{\Cantor}{\repspacex}$ be the forward functional witnessing the reduction $\mathsf{J} \weireducible f * g$ and let $A_n \defas H^{-1}(B_n)$: notice that $\Cantor = \bigcup_{n \in \nats} A_n$. If we restrict $\mathsf{J}$ to $A_n$ (denoted by $\mathsf{J}|_{A_n}$), we can safely replace $g$ with the constant function $n$. As a consequence, we conclude that $J|_{A_n} \weireducible f$ for all $n \in \nats$.
	  
	   Let $C \subseteq \nats$ be the set of all $n \in \nats$ such that $\mathsf{J}|_{A_n}$ (that, $\mathsf{J}$ with domain restricted to $A_n$) is continuous. If $\mathsf{J}|_{A_n}$ is continuous, it is computable relative to some oracle, say $p_n$. Now consider $q \defas \oplus_{n \in C} p_n$. If $q \in A_n$ for some $n \in C$ were true, then $\mathsf{J}(q) \leq_{\mathrm{T}} q$ would follow, a contradiction. Thus, there exists some $d \in \nats \setminus C$, i.e.\ some $\mathsf{J}|_{A_d}$ is discontinuous.
	  
	   As $\partialfunction{\mathsf{J}|_{A_d}}{\Cantor}{\Cantor}$ is a discontinuous function between admissible represented space, it follows that $\lpo \weireducible \mathsf{J}|_{A_d}$ relative to some oracle. We already established $\mathsf{J}|_{A_d} \weireducible f$, so $\lpo \weireducible f$ relative to some oracle follows.
	   \end{proof}

	\begin{lemma}
		\thlabel{findssigma11}
		\
		\begin{enumerate}[$(i)$]
			\item 	$\findsEE{\R} \weireducible \connected\findsEE{\R} * \choice{\boldfaceSigma_1^1}{\nats}$;
		   \item 	$\findsCC{\R} \weireducible (\connected\findsCC{\R} * \choice{\boldfaceSigma_1^1}{\nats})'$.
		\end{enumerate}
	\end{lemma}
	\begin{proof}
		Let $q$ be a name for an input $H \in \repspaceegraphs$ of $\findsEE{\R}$ and let $A \defas \{v : (\exists G_0 \cong G)(G_0 \text{ is a subgraph of } H \land v \in V(G_0))\}$. This is a valid input for $\choice{\boldfaceSigma_1^1}{\nats}$ and, given $v \in \choice{\boldfaceSigma_1^1}{\nats}(A)$, we can compute a name for the graph $H_0 \defas H\restriction_{\{w \in V(H): \pathconnected{v}{w}{}{H}\}} \in \repspaceegraphs$ (notice that $H_0$ is connected and it is easy to check that $\pathconnected{v}{w}{}{H}$ is a $\Sigma_1^0$ property).
	   To do so, we enumerate a vertex $w$ in $V(H_0)$ only when $(\exists s)(\pathconnected{v}{w}{}{\repmap{EGr}(q[s]q(0)^\nats)})$. Then we enumerate all the vertices/edges in the path from $v$ to $w$. By definition of $v$, $\R\subgraph H_0$ and any $R \in \connected\findsEE{\R}(H_0) $ is a valid solution for $\findsEE{\R}$.
   
		To prove the second item, by the fact that for any multi-valued function $f$, $f * \lim \weiequiv f'$ (see \S \ref{subsec:Weihrauch}) and \thref{findconcomponent}, the right-hand-side of the reduction is equivalent to $\connected\findsCC{\R} * \choice{\boldfaceSigma_1^1}{\nats}* \mathsf{D}$ and hence it suffices to show that $\findsCC{\R} \weireducible \connected\findsCC{\R} * \choice{\boldfaceSigma_1^1}{\nats}* \mathsf{D}$. Let $q$ be a name for  an input $H \in \repspacegraphs$ of $\findsCC{G}$, and 	let $f\in  \mathsf{D}(H)$. Then, given $A$ as above, let $v \in \choice{\boldfaceSigma_1^1}{\nats}(A)$. Now consider the graph $H\restriction_{\{w: f(w)=f(v)\}}$: this is a suitable input for $\connected\findsCC{\R}$ any $R \in \connected \findsCC{\R}(H\restriction_{\{w: f(w)=f(v)\}})$ is a valid solution for $\findsCC{\R}(H)$.
	\end{proof}
	\subsubsection*{\underline{Step 2} of \thref{maintheoremray}' proof: proving the theorem for $\findsCE{\R}$ and $\findsEE{\R}$}
	The finitary part of a problem was defined in \S \ref{finitarypart}.
   \begin{theorem}
		\thlabel{finitepartiscomputable}
	   For every $k \in \nats$, $\mathrm{Fin}_\mathbf{k}(\connected\findsEE{\R})$ is computable, and hence we obtain that $\mathrm{Fin}(\connected\findsEE{\R})$ is computable as well.
		\end{theorem}
		\begin{proof}
	Let $\partialmultifunction{f}{\Baire}{\mathbf{k}}$ for some $k \in \nats$ and suppose  that $f\weireducible \connected\findsEE{\R}$ as witnessed by the computable maps $\Phi$ and $\Psi$.

	Let $r$ be a name for $x \in \dom(f)$ and notice that $\R\subgraph\Phi(r)$. The proof strategy is the following: 
	\begin{enumerate}[$(i)$]
		\item proving that $\Phi(r)$ contains $k+1$-many distinct line segment of finite length $\ray{\ell_0},\dots,\ray{\ell_{k}}$ with names $p_{\ell_i}$ such that the following hold
		\begin{itemize}
		\item $\connectedunion{i \leq k} \ray{\ell_i}\cong  \ray{N}$ where $N\defas \sum_{i\leq k} \ray{\ell_i}$,
		\item $\connectedunion{i \leq  k} \ray{\ell_i}$ is a subgraph of $\Phi(r)$  and
		\item  for every $i \leq k$, $\Psi( r[\length{p_{\ell_i}}], p_{\ell_i})\downarrow= j$ for $j<k$.
		\end{itemize}
		\item Once we have proven $(i)$, the fact that there are only $k$-many solutions for $f(x)$ implies that
		\[(\exists m)(\exists i \neq j)(\Psi( p[\length{q_{\ell_i}}], q_{\ell_i})\downarrow=\Psi( r[\length{q_{\ell_j}}], q_{\ell_j})\downarrow=m).\]
		To conclude the proof, we first show that finding $m$ is a computable process, and then that it must be the case that $m \in f(x)$.
		\end{enumerate}

 	To prove $(i)$, let $p_0$ be a name for $S_0 \in \connected\findsEE{\R}(r,\Phi(r))$ and, for readability, assume $V(S_0)=\{v_i: i \in \nats\}$ and $E(S_0)=\{(v_i,v_{i+1}): i \in \nats\}$. Let $s_0 \defas \min \{ t>0 : \Psi(r[t],p_0[t])\downarrow  \}$ and let $\ell_0 \defas \max\{i: (\exists j\leq s_0)(v_i=p_0(\str{j,j})) \}$. Without loss of generality, we can assume that $\repmap{\egraph}(p_0[s_0]p_0(0)^\nats) \cong \ray{\ell_0}$. Indeed, if $\repmap{\egraph}(p_0[s_0]p_0(0)^\nats) \not\cong \ray{\ell_0}$ we can extend the graph determined by $p_0[s_0]$ to a line segment of finite length: to do so, since $\Psi(r[s_0],p_0[s_0])\downarrow$ and $S_0\cong\R$ we can extend $p_0[s_0]$ to $p_0[s_0]\concat \sigma$ where $\sigma$ is the (finite string) having digits $\str{v_i,v_i}$ for $i < \ell_0$ and $(v_j,v_j+1)$ for $j < \ell_0-1$. Then, we obtain that $\repmap{\egraph}(p_0[s_0]\concat \sigma\sigma(0)^\nats) \cong \ray{\ell_0}$ and $\Psi(r[s_0+\length{\sigma}],p_0[s_0]\concat \sigma)\downarrow$. For $0 < m \leq k$, let $p_m$ be a name for $S_m$ where $V(S_m) = V(S_0) \setminus \{v_i : i < \ell_{m-1}\}$ and $E(S_m)=E(S_0)\setminus \{(v_i,v_{i+1}) : i <\ell_{m-1}-1\}$.  Notice that $S_m \cong \R$, hence  $S_m \in \connected\findsCE{\R}(p,\Phi(r))$. Then, given $s_m \defas \min \{ t>0 : \Psi(p[t],p_m[t]) \downarrow  \}$, let $\ell_m \defas \max\{i : (\exists j \leq s_k)(v_i=p_m(\str{j,j}))\}$. Again, without loss of generality, we can assume that $\repmap{\egraph}(p_m[s_m]p_m(0)^\nats) \cong \ray{\ell_m}$. Since for every $i<k$, $\max\{v: v \in V(\repmap{\egraph}(p_{\ell_i}[s_{\ell_i}]p_{\ell_i}(0)^\nats))\}=\min\{v:v \in V(\repmap{\egraph}(p_{\ell_{i-1}}[s_{\ell_{i-1}}]p_{\ell_{i-1}}(0)^\nats))\}$, we obtain that $\connectedunion{i \leq k} \ray{\ell_i}\cong  \connectedunion{i \leq k} \repmap{\egraph}(p_{\ell_i}[s_{\ell_i}]p_{\ell_i}(0)^\nats)$  and the proof of $(i)$.
	
 	To prove $(ii)$, it suffices to show that 
	\begin{equation}
		\label{equationconnectedproof}
		(\exists S \in \connected\findsEE{\R}(\Phi(r)))(\ray{\ell_i}\text{ is a subgraph of }S \lor \ray{\ell_j} \text{ is a subgraph of } S).
	\end{equation}
Indeed, suppose $\ray{\ell_i}$ is a subgraph of $S$ (the case for $ray{\ell_j}$ is analogous): then there is a name for $S$ that begins with $q_{\ell_i}$, i.e.\ an enumeration of $\ray{\ell_i}$. Since $\Psi(r[\length{q_{\ell_i}}],q_{\ell_i})\downarrow=m$ by hypothesis, we are done. So let $S' \in \connected\findsEE{\R} (\Phi(r)))$. We have the following cases:
 	\begin{itemize}
 	\item if (\ref{equationconnectedproof}) holds there is nothing to prove;
 	\item if $V(S') \cap (V(\ray{\ell_i}) \cup V(\ray{\ell_j}))=\emptyset$, notice that by hypothesis $\Phi(r)$ is connected, hence $(\exists v \in V(\ray{\ell_i}))(\exists w \in V(S))(\pathconnected{v}{w}{}{\Phi(r)})$. If $\D{\ray{\ell_i}}{v}=1$, let $S$ be the infinite ray starting with $\ray{\ell_i}$, passing through $w$ and continuing as $S'$. Otherwise, let $S$ be the infinite ray starting with $\ray{\ell_j}$, passing through $v$ and $w$ and continuing as $S'$.
 	\item if $V(S') \cap V(\ray{\ell_i}) \supseteq \{v\}$ let $S$ be any infinite ray containing $\ray{\ell_j}$, passing through $v$ and continuing as $S'$; the case $V(S) \cap V(\ray{\ell_j}) \sqsupseteq \{v\}$ is analogous.
 	\end{itemize}
 	Since the procedure to search $\connectedunion{i \leq k} \ray{\ell_i}$ is computable we have shown that $f$ is computable and this proves the theorem.
		\end{proof}
   
		The following  corollary is immediate combining the previous theorem and \thref{summaryForConnectedness}.
		\begin{corollary}
		   \thlabel{corollarylpo1}
			$\lpo \not\weireducible  \connected\findsCE{\R}, \ \connected\findsEE{\R}$ (the same holds relative to any oracle).
		\end{corollary}
   
   The next proposition proves \thref{maintheoremray}$(i)$ for $\findsEC{\R}$ and $\findsEE{\R}$.
   \begin{proposition}
	   \thlabel{firstpartofproofgraph}
	   $\mflim \not\weireducible \findsCE{\R}, \findsEE{\R}$. 
   \end{proposition}
   \begin{proof}
	   For the sake of contradiction, suppose that $\mflim \weireducible \findsEE{\R}$: by \thref{findssigma11}$(i)$ we obtain that $\mflim \weireducible \connected\findsEE{\R} * \choice{\boldfaceSigma_1^1}{\nats}$. \thref{subgraphsigmacontinuous1} implies that  $\lpo \weireducible \connected\findsEE{\R}$ relative to some oracle, contradicting \thref{corollarylpo1}.
   \end{proof}
   
	Before moving our attention to $\connected\findsCC{\R}$ and $\connected\findsEC{\R}$, we prove that the fact that $\mathrm{Fin}(\connected\findsEE{\R})$ is computable (\thref{finitepartiscomputable}) cannot be extended to first-order functions. Let $\mathrm{ACC}_\nats$ be the restriction of $\cnats$ to sets of the form $\nats$ or $\nats \setminus \{n\}$ for some $n \in \nats$: notice that this problem is not computable. 
		\begin{proposition}
			\thlabel{subgraphACCFindLine}
			$\mathrm{ACC}_\nats\strictlyweireducible \firstOrderPart{\connected\findsEE{\R}}$. 	
	   \end{proposition}
			\begin{proof}
			 Let $A$ be an input for $\mathrm{ACC}_\nats$, and let $A^c[s]$ denote the enumeration of the complement of $A$ up to stage $s$. We compute a graph $G$ as follows. At stage $0$, let $1 \in V(G)$ (we deliberately leave $0$ outside $V(G)$ for the moment). At stage $s+1$,
				\begin{itemize}
					\item if $A^c[s+1]=\emptyset$, let ${s+2} \in V(G)$ and  $({s+1},{s+2}) \in E(G)$.
					\item if $A^c[s+1]=\{n\}$
					\begin{itemize}
						\item if $n\leq s+1$, let $0 \in V(G)$, $(n,{0}) \in E(G)$ and add a copy of $\R$ starting from $0$ and end the construction.
						\item if $n>s+1$, add a line segment from $s+1$ to $n$ having vertices $\{i : s+1 \leq i \leq n\}$ and let $(n,{0}) \in E(G)$. Then add a copy of $\R$ starting from $0$  and end the construction.
					\end{itemize}			
				\end{itemize}	
				Let $R'	\in \connected\findsEE{\R}(G)$. Exactly one of the following holds:
				\begin{itemize}
					\item $({n},{n+1}), ({n+1},{n+2}) \in E(R')$, for $n>0$. In this case, $n \in \mathrm{ACC}_\nats(A)$;
					\item $({n},{0}) \in E(R')$ for $n>0$. Then any $m \neq n$ is in $\mathrm{ACC}_\nats(A)$.
				\end{itemize}
				Strictness follows from \thref{summaryForConnectedness}.
			\end{proof}
		   Combining \thref{finitepartiscomputable,subgraphACCFindLine,finpartbelowfirstorder} and Figure \ref{definitionisg} we obtain the following corollary.
		   \begin{corollary}
			   $\mathrm{Fin}(\findsEE{\R}) \strictlyweireducible \firstOrderPart{\findsEE{\R}}$.
		   \end{corollary}
	   \subsubsection*{\underline{Step 3}: proving \thref{maintheoremray} for $\findsCC{\R}$ and $\findsEC{\R}$}
		The next proposition, together with the next corollary, shows that $\mathrm{Fin}(\connected \findsCE{\R})$  is not computable, and hence the same proof technique used to prove \thref{maintheoremray}$(i)$ does not work here.
	   
		\begin{proposition}
			\thlabel{ccantorfindsubgraph}
			 $\CCantor\strictlyweireducible \connected\findsCC{\R}$.
			\end{proposition}
			\begin{proof}
			Let $T\in \tree_2$ be an input for $\CCantor$ and let $G \in \connected\findsCC{\R}(T)$. Notice that, the fact that $T \in \tree_2$ implies that for every $n$ $\length{\{\sigma \in V(T): \length{\sigma}=n\}}\leq 2^n$, and this combined with the fact that $G \in \repspacegraphs$ implies that we can compute $\length{\{\sigma \in V(G): \length{\sigma}=n\}}$. Finally, it is easy to verify that $\bigcup_{n \in \nats}\{\sigma:\length{\{\sigma \in V(G): \length{\sigma}=n\}}=1\} \in \body{T}$ and this proves the reduction.
		   
			Strictness follows from \thref{summaryForConnectedness} and the fact that $\CBaire\not\weireducible \mflim_2* \CCantor$.
			\end{proof}
			Notice that $\mathsf{C}_2^*\weiequiv \firstOrderPart{\CCantor}$ (\cite[Corollary 7.6]{valentisolda}).
			\begin{corollary}
				\thlabel{firstorderpartccantor}
				$\mathsf{C}_2^*\weiequiv \firstOrderPart{\CCantor}\strictlyweireducible \connected\findsCC{\R}$.
			\end{corollary}
   The corollary above implies that $\mathrm{Fin}(\findsCC{\R})$ is not computable: despite this, we can prove that $\lpo\not\weireducible \connected\findsCC{\R}$. To do so, we introduce the notion of \emph{promptly connected} graph. Recall that, for a  graph $G$ and $V \subseteq V(G)$ the graph induced by $V$ on $G$, denoted by $G_{\restriction V}$ is such that $V(G_{\restriction V})\defas V$ and $E(G_{\restriction V})\defas E(G)\cap (V \times V)$.

	\begin{definition}
	A graph $G$ is  \emph{promptly connected} if for every $n$, the graph $G_{\restriction V \cap \{0,\dots,n\}}$ is connected.
	\end{definition}
   
	\begin{proposition}
	\thlabel{propositionincreasinglyconnected}
	Let $G$ be a promptly connected graph and let $v_0\defas \min\{v:v \in V(G)\}$. Then, for every $v \in V(G)\setminus \{v_0\}$, $v_0$ and $v$ are \quot{increasingly connected}, i.e.,
		\[(\exists \sigma)(\forall i< |\sigma|-1)(\sigma(0)=v_0 \land \sigma(\length{\sigma}-1)=v \land (\sigma(i),\sigma(i+1))\in E(G) \land \sigma(i)<\sigma(i+1)).\]
	\end{proposition}
	\begin{proof}
   Let $\{v_i: i \in \nats \land (\forall i)(v_i<v_{i+1}) \}$ be an enumeration of $V(G)$. The proof goes by an easy induction on the $v_i$'s. The claim for $v_1$ holds trivially: indeed, $v_0<v_1$ and, by definition of promptly connected graph, $(v_0,v_1) \in E(G)$. Suppose that $v_0$ and $v_s$ for $s>1$ are increasingly connected: we show that $v_0,\dots,v_{s+1}$ are increasingly connected as well. By definition of promptly connected graph, $G_{\restriction V \cap \{0,\dots,v_{s+1}\}}$ is a connected graph, and hence $(\exists i< s+1)((v_i,v_{s+1}) \in E(G))$. By inductive hypothesis, $v_0$ and $v_i$ are increasingly connected by some path $\sigma$ and hence, since $v_i<v_{s+1}$, the path $\sigma\concat v_{s+1}$ witnesses that $v_0$ and $v_{s+1}$ are increasingly connected.
	\end{proof}
   
	\begin{lemma}
	\thlabel{prefixproperty}
	There exists a computable function $\partialfunction{\mathbf{PC}}{\repspacegraphs}{\repspacegraphs}$ that, given in input a connected graph $G \in \repspacegraphs$, is such that $\mathbf{PC}(G) \cong G$, $\mathbf{PC}(G)$ is promptly connected and $0 \in V(\mathbf{PC}(G))$.
	\end{lemma}
	\begin{proof}
	Given a name $p$ for $G \in \repspacegraphs$, the function $\mathbf{PC}$ computes a name for a graph in $\repspacegraphs$ as follows. We define an auxiliary map $\function{\iota}{V(G)}{\nats}$ that arranges the vertices of $G$ so that $\mathbf{PC}(G)$ satisfies the properties of the lemma. Let $(v_i)_{i \in \nats}$ be an enumeration of $V(G)$. At stage $0$, let $\iota_0(v_0)=0$. At stage $s+1$, let $n_{s+1} \defas \max\{n : n \in \range(\iota) \text{ at stage } s\}$ and, if $(\exists i< s+1)((v_i,v_{s+1}) \in E(G))$, let $\iota(v_{s+1})=n_{s+1}+1$. Otherwise, wait for a stage $t_{s+1}\defas \min \{t: \pathconnected{v_0}{v_{s+1}}{}{G} \text{ at stage }t\}$ and let $\sigma_{s+1}$ be the path connecting $v_0$ and $v_{s+1}$ in $G$ at stage $t_{s+1}$. For every $i<\length{\sigma}$, if $\sigma(i) \notin \dom(\iota)$ at stage $s$ let $\iota(\sigma(i)) \defas n_{s+1}+i$. This ends the construction.
   
   Let $\mathbf{PC}(G)$ be the graph having $\{\iota(v_i): v_i \in V(G)\}$ and $\{(\iota(v_i),\iota(v_j)) : (v_i,v_j) \in E(G)\}$  as vertex set and edge set respectively. It is straightforward to verify that $\mathbf{PC}(G)$ satisfies the properties of the lemma.
				\end{proof}
			   
	The following lemma gives us a useful property of promptly connected graphs. First, given a connected graph $G$ we define the \emph{distance} between $v,w \in V(G)$ as 
	$$ d^G(v,w)\defas \min \{n : (\exists \sigma \in \nats^n) (\forall i < \length{\sigma}-1)(\sigma(0)=v\land \sigma(\length{\sigma}-1)=w \land (\sigma(i),\sigma(i+1)) \in E(G))\}.$$
   
	\begin{lemma}
	\thlabel{decreasingunionofcompactsets}
	Suppose $G$ is a promptly connected graph such that $(\forall v \in V(G))(\D{G}{v}<\aleph_0)$ and $\R\subgraph G$. then, 
	\[(\exists R\cong \R)(R\subseteq G \land (\forall v,w \in V(R))(d^G(0,v)<d^G(0,w)\implies v<w)),\]
		\end{lemma}
   
		\begin{proof}
		   Consider the tree:
		   \[T\defas \{\sigma: \sigma(0)=\min\{v: v \in V(G)\} \land (\forall i<\length{\sigma}-1)((\sigma(i),\sigma(i+1)) \in E(G) \land \sigma(i)<\sigma(i+1))\}.\]
		  To prove the lemma, it suffices to show that $\body{T}\neq \emptyset$. Indeed, suppose $p \in \body{T}$: then the graph having $ \{p(i):i \in \nats\}$ and $\{(p(i),p(i+1)):i \in \nats\}$ as vertex set and edge set respectively is the copy $R\cong \R$ contained in $G$ we were looking for.
		  
		   To show that $\body{T}\neq \emptyset$, first notice that, by hypothesis, $(\forall v \in V(G))(\D{G}{v}<\aleph_0)$ and in particular $(\forall \sigma \in T)(\forall i <\length{\sigma})(\length{\{v : (\sigma(i),v) \in E(G)\}}< \aleph_0)$, implying that $T$ is finitely branching. Hence, by K{\"o}nig lemma, it suffices to show that $T$ is infinite. By \thref{propositionincreasinglyconnected}, for every $v \in V(G)\setminus \min\{v: v \in V(G)\}$,  $\min\{v: v \in V(G)\}$ and $v$ are {increasingly connected}, and hence, for any $v \in V(G)$, there exists  $\sigma \in T$ such that $\sigma(\length{\sigma}-1)=v$ and since $G$ is infinite, $T$ is infinite as well, and this concludes the proof.
			   \end{proof}

		   We are now ready to show that  $\lpo \not\weireducible \connected\findsCC{\R}$.
				\begin{proposition}
				\thlabel{lpofindsCC}
				$\lpo \not\weireducible \connected\findsCC{\R}$ relative to any oracle $\mu$.
				\end{proposition}
				\begin{proof}
				   Let $\Phi,\Psi$ be the computable (relative to $\mu$) forward and backward functionals witnessing that $\lpo \weireducible \connected\findsCC{\R}$. Without loss of generality, we can restrict the domain of $\lpo$ to $A\defas \{0^\nats\} \cup \{0^i10^\nats: i \in \nats\}$: doing so, we ensure that $\Phi$ produces only countably many graphs. Let $\mathbf{PC}$ be the function defined in \thref{prefixproperty}, and compute the graphs $G_\infty\defas \mathbf{PC}(\Phi(0^\nats))$ and, for every $i \in \nats$, $G_i\defas \mathbf{PC}(\Phi(0^i10^\nats))$. Notice that,  by \thref{propositionincreasinglyconnected,decreasingunionofcompactsets}, for every $x \in \nats \cup \{\infty\}$, $0 \in V(G_x)$ and, for any $v \in V(G_x)$, $0$ and $v$ are increasingly connected.
				   We claim that there is a $\function{\lambda}{\nats}{\nats}$ such that for every $x \in \nats \cup \{\infty\}$ there exists $R \cong \R$ in $G_x$ (with  $V(R)=\{v_i :i \in \nats\}$, $E(R)\defas \{(v_i,v_{i+1}) : i \in \nats\}$ and  $v_0=0$) with the property that, for every $n \in \nats$, $v_n \leq \lambda(v_{n+1})$. The existence of such a $\lambda$ would witness that $\lpo\weireducible \CCantor$ relative to $\mu \oplus \lambda$, giving rise to a contradiction. Indeed,  given an input $p \in A$ for (the restricted version of) $\lpo$, via the $\mathbf{PC}$ defined in \thref{prefixproperty} we can compute $\mathbf{PC}(\Phi(p))$ being promptly connected. Then, consider the graph $G$ such that $V(G)\defas V(\mathbf{PC}(\Phi(p)))$ and $E(G) \defas \{(v,w): v,w \in V(\mathbf{PC}(\Phi(p)))\land \max\{v,w\}< \lambda(\min\{v,w\})\}$. Clearly, $G$ satisfies the conditions of \thref{decreasingunionofcompactsets}. Hence, compute the tree $T$ of the proof of \thref{decreasingunionofcompactsets} (a valid input for $\CCantor$) and, given $f \in \CCantor(\body{T})$ (as we have done in the same proof)  we compute a $\repmap{\graph}$-name $q$ for some $R\cong \R$ in $G$. Then, $\Psi(p,R)$ is a correct answer for $\lpo$, showing that $\lpo\weireducible \CCantor$ relative to the oracle $\lambda \oplus \mu$, obtaining the desired contradiction. 
		 
		We define $\lambda$ as follows. Let $p_\infty$ be the name for $G_\infty \in \repspacegraphs$ and, for every $i \in \nats$ let $p_i$ be the name for $G_i \in \repspacegraphs$. Let $R \in \connected\findsCC{\R}(G_\infty)$ with name $r$, and let $v$ be such that $\D{R}{v}=1$. Since $R$ is a solution for $\connected\findsCC{\R}(G_\infty)$, there exists a stage $\ell$ such that $v<\ell$ and $\Psi(0^\ell,\repmap{\graph}(r)[\ell])\downarrow=1$. Without loss of generality, since $\Phi$ is computable and in particular continuous, we can assume that for every $k \geq \ell$, $\repmap{\graph}(p_k[\ell]p_k(0)^\nats)\cong \repmap{\graph}(p_\infty[\ell]p_\infty(0)^\nats)$. Notice that, for the reduction to work correctly, it must be the case that, 
	   \begin{equation}
	   \label{equationfindline}
	   (\forall k \geq \ell)(\forall S\cong\R)(S \text{ is a subgraph of } G_k \implies \Psi(0^\nats, R)\downarrow\neq \Psi(0^k10^\nats,S)\downarrow).
	   \end{equation}
	   We claim that (\ref{equationfindline}) implies
	   \[(\forall k \geq \ell)(\forall S\cong\R)(\exists u \neq v <\ell)((S \text{ is a subgraph of } G_k \land \D{S}{v}=1) \implies u \in V(S)).\]
	   Suppose not: then, for $k\geq \ell$, let $s'$ be a name for $S' \in \connected\findsCC{\R}(G_k)$ be such that $\D{S}{v}=1$ and suppose that for every $u\neq v <\ell$, $u \notin V(R)$, i.e.\ vertices after $v$ in  $V(S')$ are greater than $\ell$. Notice that $R[\ell] \connectedunion{} S' \in \connected\findsCC{\R}(G_k)$ (where the connection involves the vertex $v$) and let $s''$ be a $\repmap{Gr}$-name for such a solution. Clearly, $\repmap{\graph}(s''[\ell]s''(0)^\nats) \cong \repmap{\graph}(r[\ell]r(0)^\nats)$ and so, by hypothesis, $\Psi(0^\ell,s''[l])\downarrow=1$, a contradiction. 
	  
	   Furthermore, since $G_k$ is promptly connected, by \thref{propositionincreasinglyconnected} $0$ and $v$ are increasingly connected. Then, given
	   \[a_k \defas  \min\{w: (\forall i<k)(\exists R \cong \R)(R \subseteq G_i \land 0,v,w \in V(R) \land d^R(0,v) < d^R(0,w))\}.\]
	   let  $\lambda(v)\defas \max \{\ell, a_k\}$: the definition of $\lambda(v)$ ensures that in every $G_x$, for $x \in \nats \cup \{\infty\}$, we can find a copy $R$ of $\R$ such that $\D{R}{0}=1$,  passes through $v$ and \lq\lq continues\rq\rq\ with some $w \leq \lambda(v)$. 	This concludes the proof. 
	   \end{proof}
	  
			   So far we have shown that $\lpo\not\weireducible \connected\finds{a}{\R}$ for $a \in \{\chi e, ee, \chi\chi\}$. The next proposition shows that instead, $\lpo\strictlyweireducible \connected\findsEC{\R}$. 
   
				\begin{proposition}
					\thlabel{eccomputeslpo}
					$\lpo\strictlyweireducible \connected\findsEC{\R}$. 
				\end{proposition}
				\begin{proof}		
					For the reduction, let $p \in \Cantor$ be an input for $\lpo$: we compute a tree $T \in \tree_2$ as follows. At stage $s$, if $p(s)=0$, let $0^s \in T$. If $p(s)=1$, we stop inspecting $p$ at later stages and for every $t$, let $1^t \in T$. It is clear that if $p=0^\nats$ then $\body{T}=\{0^\nats\}$, while if $(\exists i)(p(i)=1)$, $\body{T}=\{1^\nats\}$. The fact that $T\in \illfounded_2$ implies that $T \in \dom(\findsEC{\R})$. Let $R \in  \findsEC{\R}(T)$, and notice that by definition of $T$ there exists an $s$ such that for all $t>s$, either $0^t \in V(R)$ and $1^t \notin V(R)$ or vice versa. Hence, it suffices to search for such an $s$, and in the first case $\lpo(p)=1$ while in the second one  $\lpo(p)=0$.
   
						The fact that $\connected\findsEC{\R} \not\weireducible \lpo$ is straightforward and this concludes the proof.
   \end{proof}		
				Combining \thref{finitepartiscomputable,lpofindsCC,eccomputeslpo} we obtain the following corollary.
				\begin{corollary}
					$\connected\findsCC{\R}, \ \connected\findsEE{\R} \strictlyweireducible \connected\findsEC{\R}$.
				\end{corollary}
		   Despite \thref{eccomputeslpo} we are still able to prove that $\CBaire\not\weireducible \findsEC{\R}$: before concluding the proof of \thref{maintheoremray}, we recall the \emph{jump inversion theorem}, stating that  $f' \weireducible g' \implies f \weireducible g$ relative to the halting problem (see \cite[Theorem 11]{montecarlo}).

	\begin{proposition}
	   \thlabel{secondpartofproofgraph}
	$\mflim' \not\weireducible  \findsCC{\R}$ and $\mflim'' \not\weireducible  \findsEC{\R}$.  
	\end{proposition}
	\begin{proof}

 For the first nonreduction, assume for the sake of contradiction, that $\mflim' \weireducible  \findsCC{\R}$. By \thref{findssigma11}$(ii)$, $\findsCC{G} \weireducible (\connected\findsCC{G} * \choice{\boldfaceSigma_1^1}{\nats})'$. By the jump inversion theorem, if $\mflim' \weireducible (\connected\findsCC{G} * \choice{\boldfaceSigma_1^1}{\nats})'$ then $\mflim \weireducible \connected\findsCC{G} * \choice{\boldfaceSigma_1^1}{\nats}$ relative to $\emptyset'$. If so, \thref{subgraphsigmacontinuous1} implies that $\lpo \weireducible \connected\findsCC{G}$, contradicting \thref{lpofindsCC}. 
  
 For the second nonreduction, assume again for the sake of contradiction, that $\mflim'' \weireducible  \findsEC{\R}$. By \cite[Lemma 6.3]{BG09}, we get that computing the characteristic function of a set from its enumeration is Weihrauch equivalent to $\mflim$. Hence, $\findsEC{\R}\weireducible {\findsCC{\R}}'$ and from this we obtain that $\mflim''\weireducible{\findsCC{\R}}'$. The jump inversion theorem implies that $\mflim'\weireducible \findsCC{\R}$ relative to $\emptyset'$ and combined with the fact that $\findsCC{\R} \weireducible (\connected\findsCC{G} * \choice{\boldfaceSigma_1^1}{\nats})'$ we would obtain  $\mflim' \weireducible (\connected\findsCC{G} * \choice{\boldfaceSigma_1^1}{\nats})'$. Applying again the jump inversion theorem we would finally get that $\mflim \weireducible \connected\findsCC{G} * \choice{\boldfaceSigma_1^1}{\nats}$ relative to $\emptyset''$, contradicting \thref{lpofindsCC}.
	\end{proof}
   
	Notice that the strongest result we have for non first-order problems is $\CCantor\strictlyweireducible\finds{}{\R}$. It is not clear to us what $\finds{}{\R}$ compute: it would be also interesting to characterize its first-order/deterministic/finite part, in terms of some already known problem in the Weihrauch lattice.
   \subsection*{The subgraph problem: when $G$ has only finite components}
   
   Now we deal with the problems $\finds{}{G}$ where $G\defas \disconnectedunion{i \in \nats} F_i$, with $F_i$ a finite graph. For these problems, it seems to be harder to  reach the equivalence with $\CBaire$: indeed, we have already discussed that $\CBaire$ can be stated as the task of finding a path through an ill-founded tree, and graphs of this kind  are far from being intuitively \quot{ill-founded}, i.e.\ from having $\R$ as subgraph. This intuition is actually wrong: even if for many graphs $G$ the problems $\finds{}{G}$ are computable, for others we have that $\finds{}{G}\weiequiv \CBaire$. The important distinction is the following. Given $G\defas \disconnectedunion{i \in \nats} F_i$ we distinguish whether 
	\begin{equation}
	\label{Equationfinmanycomponents}
	(\forall^{\infty} i)(\exists^\infty j)(F_i \subgraph F_j).
	\end{equation}
	\begin{theorem}
	   \thlabel{computablefinitecomponents}
	Let $G$ be an infinite computable graph such that $G=\disconnectedunion{i \in \nats} F_i$, where $F_i$ is a finite graph and $(\forall^{\infty} i)(\exists^\infty j)(F_i \subgraph F_j)$. Then the problems $\finds{}{G}$ are computable.
	\end{theorem}
	\begin{proof}
	   By Figure \ref{findisdefinitionsFigure}, it suffices to show that $\findsEC{G}$ is computable.
	By definition there exists $k \in \nats$ such that there are $k$-many graphs $F_{n_0},\dots,F_{n_{k-1}}$ that are subgraphs of just finitely many $F_i$'s. Let $A\defas \{F_i: (\exists l<k)(F_{n_l} \subgraph F_i )\}$ and notice that $\length{A}<\aleph_0$. Given $H \in \dom(\findisEC{G})$ in input, we compute $G'$ being a subgraph of $H$ such that $G'\cong G$ with the two following procedures that can be performed in parallel.
			   
	\begin{itemize}
	\item The first procedure waits for a finite stage witnessing that $(\forall F \in A)(F \subgraph H)$ and adds to $G'$ the corresponding copy of $F$ in $H$. Since all $F$'s are finite and $A$ is finite as well, such an $s$ exists, and we can computably find it.
	\item The second procedure takes care of all the $F_s \in \{F_i : i \in \nats\} \setminus A$. For every $s>k$, it adds to $G'$ the first copy of $F_s$ in $H$ that it finds. We claim that this procedure eventually adds to $G'$ a copy for every $F_s$. Recall that, by hypothesis, $(\forall s>k)(\exists^\infty j)(F_i \subgraph F_j)$. Suppose that at stage $s$ there exists an $m>k$ such that $F_m$ has not been added to $G'$ yet. Since $(\exists^\infty j)(F_m \subgraph F_j)$, and we have seen only a finite portion of $H$, we can wait for a finite stage greater than $ s$ such that $F_m'$ is a subgraph of  $H$, $F_m'\cong F_m$ and $F_m'$ has not been added to $G'$ yet: hence we can add $F_m'$ to $G'$.
	\end{itemize}
	This completes the proof.
	\end{proof}
	The graph $\infinitelycycles$ does not satisfy (\ref{Equationfinmanycomponents}): the next theorem shows that  $\findisCE{\infinitelycycles}\weiequiv \CBaire$.
	\begin{theorem}
	   \thlabel{infinitelycyclescbaire}
	$\findis{}{\infinitelycycles} \weiequiv \CBaire$.
	\end{theorem}
	\begin{proof}
	The left-to-right direction is \thref{findinducedsubgraphupperbound}. For the opposite direction, by Figure \ref{findisdefinitionsFigure}, it suffices to show that $\CBaire\weireducible\findsEC{G}$. We begin partitioning the $\cycle{i}$'s in  three infinite disjoint sets: i.e., $\{\cycle{i}:i \geq 3\}=\{P_n:n \in \nats\}\cup\{F_n:n \in \nats\}\cup\{G_n:n \in \nats\}$. Let $T \in \tree$ be an input for $\CBaire$: we compute a graph $H \in \dom(\findisCE{\infinitelycycles})$ in stages as follows. First, for every $n \in \nats$, we add to $H$
	\begin{itemize}
	\item infinitely many copies of $P_n$, that we denote by $\{P_n^i: i \in \nats \land  P_n^i \cong P_n\}$,			\item two copies of $G_n$ denoted by $G_n^0$, $G_n^1$, and
	\item a copy of $F_n$.
	\end{itemize}
	For every $i,n \in \nats$, we associate to $P_n^i$ a \emph{box} containing the graphs
	\[ P_n^i\odot  G_{\pair{n,i,0}}^1 \text{ and }  \disconnectedunion{k \in \nats} \big ( G_{\pair{n,i,k}}^0 \odot  G_{\pair{n,i,k+1}}^1 \big ).\]
   
	Notice that for any $k>0$, $G_{\pair{n,i,k}}^0$ has a designed docking vertex (different from the ones involved in $G_{\pair{n,i,k}}^1\connectedunion{}  G_{\pair{n,i,k+1}}^0$) to which, in later stages, we may attach a graph from $\{F_n:n \in \nats\}$, i.e.\ $G_{\pair{n,i,k}}^1\connectedunion{}  G_{\pair{n,i,k+1}}^0$ may become $G_{\pair{n,i,k}}^1\connectedunion{}  G_{\pair{n,i,k+1}}^0 \connectedunion{} F_n$. We say that $G_{\pair{n,i,k}}^0$ is \emph{free} if no $F_n$ is attached to the designed docking vertex.
			   
	Without loss of generality, we assume that $\length{\{\sigma:\sigma \notin T\}}=\aleph_0$ is infinite: indeed, if such a set is finite, instead of considering $T$, we consider the tree $T'\defas \str{} \cup \{1\tau:\tau \in T\}$:  clearly, $\{\sigma:\sigma \notin T'\}$ is infinite and from $T'$ we can easily compute $T$. Let $(\sigma_s)_{s \in \nats}$ be a computable (with respect to $T$) enumeration  of $\{\sigma:\sigma \notin T\}$. At stage $\sigma_s$, let $k\defas\min\{j:G_{\pair{n,i,j}}^0 \text{ is free}\}$: for every $n<\length{\sigma_s}$ attach $F_{s}$ to $G_{\pair{n,i,k}}^0$ if and only if $\sigma_s(n)=i$ (recall that $G_{\pair{n,i,j}}^0$ is only in $P_n^i$'s box). This ends the construction.
   
	We first claim that $G\subgraph H$. Let $q \in \body{T}$,  and consider the following copy of $G$ in $H$. For every $n \in \nats$, consider the graph $G'$ containing,
	\begin{itemize}
	\item  $P_n^{q(n)}$ and,
	\begin{itemize}
	\item if $P_n^i \in G'$ then, for every $k \in \nats$, we add to $G'$ the graphs $G_{\pair{n,i,k}}^0$,
	\item if $P_n^i \notin G'$ then, for every $k \in \nats$, we add to $G'$ the graphs $G_{\pair{n,i,k}}^1$ (this choice allow us to put in $G'$, if needed, a copy of $F_m$ contained in $P_n^i$'box).
	\end{itemize}
	\item the copy of $F_{n}$ belonging to $P_{\length{\sigma_n}-1}^{\sigma_n(\length{\sigma_n}-1)}$ box: since $\sigma_n\notin T$,  $P_{\length{\sigma_n}}^{\sigma_n(\length{\sigma_n}-1)} \notin G'$, hence by the previous point we can choose the copy of $F_n$ in this box.
	\end{itemize}
	Hence, for every $n$, we added in $G'$ a copy of $P_n$, $G_n$, and $F_n$ and this concludes the proof of the claim.
   
	To conclude the proof we need to show that from any $G' \in \findisCE{G}(H)$ we can compute some $q \in \body{T}$. First, notice the following useful fact. Suppose that $P_n^i \in G'$ and recall that, for every $x$, $G_x$ has only two copies in $H$, namely $G_x^0$ and $G_x^1$. Since $P_n^i \in G'$ and $P_n^i$ shares a vertex with $G_{\pair{n,i,0}}^1$ we are forced to add in $G'$ the copy $G_{\pair{n,i,0}}^0$. Similarly, for any $m \in \nats$, $G_{\pair{n,i,m}}^0$ shares a vertex with $G_{\pair{n,i,m+1}}^1$: the fact that $G_{\pair{n,i,0}}^0$ in $G'$ forces us, for every $m \in \nats$, to add $G_{\pair{n,i,m}}^0$ in $G'$. Another important consequence of this observation is that if $P_n^i \in G'$ we cannot put in $G'$ any copy of $F_s$ from $P_n^i$'s box, as $F_s$ shares a vertex with $G_{\pair{n,i,t}}^0$ for some $t \in \nats$, and we have just argued that $G_{\pair{n,i,t}}^0$ is in $G$. So let $(P_n^i)_{i,n \in \nats}$ be the copies of $P_n$ in $G'$: we claim that there exists a $q \in \body{T}$ such that for every $n,i \in \nats$, $q(n)=i$. Suppose not. This means that 
	\[(\exists \tau \in \baire)(\forall m<\length{\tau})(\tau[\length{\tau}-2] \in T \land \tau \notin T \land P_{m}^{\tau(m)} \in G').\]
	In other words, $\tau=\sigma_s$ for some $s$ (where $(\sigma_s)_{s \in \nats}$ is the computable enumeration of $\{\sigma: \sigma \notin T\}$). By construction, the only copies of $F_s$ are in $P_m^{\tau(m)}$'s box for $m<\length{\tau}$: since $P_m^{\tau(m)}$ are all in $G'$, from the observation above, we cannot put any copy of $F_s$ in $G'$, contradicting that $G' \in \findisCE{G}(H)$, and this concludes the proof.
	\end{proof}
	Notice that this construction heavily relies on the fact that the connected union of at most three connected and finite components of $G$ is not isomorphic to any connected component of $G$. This means that the theorem above holds for any graph having this property.
   
	\subsection{Other subgraph problems}
	\label{othersubgraphproblems}
	We have shown examples of graphs $G$ such that $\finds{}{G} \weiequiv \CBaire$, graphs $G$ such that the problems $\finds{}{G}$ are computable and others that are difficult to compute but weak when they have to compute other problems. In this section, we show that there are also graphs $G$ such that $\finds{}{G}$ occupy well known areas of the Weihrauch lattice. Before doing so,  we give a convenient characterization of $\jump{\mflim}{n}$ for every $n \in \nats$. Given a represented space $\repspacex$, we denote by $\mathcal{O}(\repspacex)$ the final topology on $X$ induced by $\repmap{X}$.
   
	\begin{definition}
		\thlabel{definitionenuminf}
	 Let $\mathrm{EnumInf}_{\Pi_n} :\subseteq \boldfacePi^0_n(\nats) \rightrightarrows \mathcal{O}(\nats)$ be defined by $U \in \mathrm{EnumInf}(A)$ if and only if $U \subseteq A \wedge \length{U} = \aleph_0$.
	 \end{definition}
	
	 \begin{lemma}
		\thlabel{enuminf}
		For $n>0$,  $\mathrm{EnumInf}_{\Pi_n} \weiequiv \jump{\mflim}{n-1}$.
	 \end{lemma}
	 \begin{proof}
		Notice that  $\widehat{\jump{\mflim}{n-1}}\weiequiv (\id : \boldfacePi^0_n(\nats) \to \Cantor)$, i.e.\ the function computing the characteristic function of a $\boldfacePi_n^0$ set. Hence, we show  $\mathrm{EnumInf}_{\Pi_n} \weiequiv (\id : \boldfacePi^0_n(\nats) \to \Cantor)$.
	
		The left-to-right direction is trivial and if $n = 0$, both sides are trivially computable.
	
	 For the opposite direction, we need to show that from a $\boldfacePi^0_n$-name of a set $A \subseteq \nats$ we can compute a $\boldfacePi^0_n$-name of an infinite set $B \subseteq \nats$, such that any enumeration of an infinite subset of $B$ allows us to recover the characteristic function of $A$.
	
	 The given $\boldfacePi^0_n$-name $A$ brings with it a sequence $(C_i)_{i \in \nats}$ of $\boldfacePi^0_{n-1}$-sets with $\nats \setminus A = \bigcup_{i \in \nats} C_i$. For $n \in \nats$, we let $\lambda_A(n) = 0$ if $n \in A$, and $\lambda_A(n) = \min \{i \mid n \in C_i\}+1$ if $n \notin A$. Let $p_0,p_1,\ldots$ be the increasing enumeration of the prime numbers. We now define:
	 $$B = \{\prod_{i \leq k} p_i^{\lambda_A(i)} \mid k \in \nats\}$$
	
	 The set $B$ has the desired properties: we can obtain its $\boldfacePi^0_n$-name from the name of $A$; the set $B$ is infinite, and an enumeration of any infinite subset of $B$ allows us to recover $B$, and then subsequently $A$.
	 \end{proof}
	
	The graphs we promised at the beginning of this section are $\mathsf{T}_{2k+1}$ and $\mathsf{F}_{2k+2}$, defined in \S \ref{wadgecomplexityofgraphs}. Before stating the main theorem of this section, we give a preparatory result.
   
		\begin{lemma}
	The problems $\finds{}{\mathsf{T}_1}$ and $\finds{}{\mathsf{F}_2}$ are computable and $\finds{}{\mathsf{T}_3}\weiequiv \choice{\boldfacePi_2^0}{\nats}$.
		\end{lemma}
		\begin{proof}
		The fact that $\finds{}{\mathsf{T}_1}$ and $\finds{}{\mathsf{F}_2}$ are computable is straightforward, regardless of whether graphs are given as input/output as elements of $\repspacegraphs$ or $\repspaceegraphs$: in the first case, given $H \in \dom(\finds{}{\mathsf{T}_1})$, it suffices to wait for a vertex $v$ to appear in $V(H)$: it is clear that $(\{v\},\{\emptyset\}) \in \finds{}{\mathsf{T}_1}(H)$. Similarly, given $H \in \dom(\finds{}{\mathsf{F}_2})$, we obtain that $(V(H),\{\emptyset\}) \in \finds{}{\mathsf{F}_2}(H)$.
   
		To show that $\finds{}{\mathsf{T}_3}\weireducible \choice{\boldfacePi_2^0}{\nats}$, by Figure \ref{findisdefinitionsFigure}, it suffices to show that $\findsEC{\mathsf{T}_3}\weireducible \choice{\boldfacePi_2^0}{\nats}$. Given in input $H \in \dom(\findsEC{\mathsf{T}_3})$, let $A\defas \{v \in V(H) : \D{H}{v}=\aleph_0\}$. It is clear that $A \in \boldfacePi_2^0(\nats)$ and $A$ is nonempty. Then, given $v_0 \in \choice{\boldfacePi_2^0}{\nats}(A)$ we compute the graph $G'$ such that $V(G')=\{v_0\} \cup \{v:(v_0,v) \in E(H)\} $ and $E(G') \defas \{(v_0,v): (v_0,v) \in E(H)\}$.
	   
		For the converse, by Figure \ref{findisdefinitionsFigure}, it suffices to show that $\choice{\boldfacePi_2^0}{\nats}\weireducible \findsCE{\mathsf{T}_3}$. An input for $ \choice{\boldfacePi_2^0}{\nats}$ is a nonempty set $A \in \boldfacePi_2^0(\nats)$. By \thref{firstcompletesets}, we can associate to every $n \in \nats$ an infinite sequence $p_n \in \Cantor$ such that $n \in \nats \iff (\exists^\infty i)(p_n(i)=1)$. Let $H$ be the graph such that $V(H) \defas \nats$ and $E(H)\defas \{(\str{n,0},\str{n,i+1}) : p_n(i)=1\}$. Since $A \neq \emptyset$ by hypothesis, we have that $(\exists m)(\exists^\infty i)(p_m(i)=1)$ and hence $\D{H}{\str{m,0}}=\aleph_0$, i.e.\ $\mathsf{T}_3 \subgraph H$: this show that $H$ is a suitable input for $\findsCE{\mathsf{T}_3}$. It is also clear that given $G' \in \findsCE{H'}$ we can compute $n \in A$ (just look at the projection on the first coordinate of a vertex in $V(G')$).
		\end{proof}
   
		We conclude this section with \thref{maintheorem_findubgraphgn}.
   
		\begin{theorem}
	\thlabel{maintheorem_findubgraphgn}
			 For $k>0$,
		   $$(i)\ \finds{}{\mathsf{T}_{2k+1}} \weireducible \choice{\boldfacePi_{2k}^0}{\nats} \times \parallelization{\jump{\lpo}{2k-3}}  \text{ and }(ii)\ \finds{}{\mathsf{F}_{2k+2}} \weiequiv \parallelization{\jump{\lpo}{2k-1}}.$$
		\end{theorem}
		\begin{proof}
		   We only prove $(ii)$ as the proof of $(i)$ is similar to the left-to-right direction of $(ii)$. 
		   
		   For the left-to-right direction, by Figure \ref{findisdefinitionsFigure}, it suffices to show that $\findsEC{\mathsf{F}_{2k+2}} \weireducible \parallelization{\jump{\lpo}{2k-1}}$. Since, for any $n$, $\parallelization{\jump{\lpo}{n}}$ is clearly closed under parallelization, it suffices to show that 
		   \[\findsEC{\mathsf{F}_{2k+2}}\weireducible   \bigtimes_{0< i \leq 2k-1}  \parallelization{\jump{\lpo}{i}}.\]
		   For every $i$ such that $0<i \leq 2k-1$, and for every $v \in V(H)$ we can uniformly compute a sequence $(p_v)_{v \in V(H)}$ such that $\jump{\lpo}{i}(p_v)=1 $ if and only if
		   \[(\exists^\infty v_0 \in V(H))\dots(\exists^{\infty} v_{j}\in V(H))(\forall  j< i)((v,v_0) \in E(H)\land (v_{j},v_{j+1})\in E(H)).\]
	   Now for every $i$ such that $0<i\leq 2k-1$, let $(v_i^m)_{m \in \nats}$ be an enumeration of $\{v \in V(H): \jump{\lpo}{i}(p_v)=1\}$. Following the same ideas of \thref{subgraphcondition}' proof we can compute a $\repmap{\graph}$-name for a copy of $\mathsf{F}_{2k+2}$ in $H$.
   
	   For the converse direction, notice that, by \thref{enuminf}, $\widehat{\jump{\lpo}{2k-1}} \weiequiv\jump{\mflim}{2k-1} \weiequiv  \mathrm{EnumInf}_{\Pi_{2k}}$: hence, by Figure \ref{findisdefinitionsFigure}, it suffices to show that $\mathrm{EnumInf}_{\Pi_{2k}} \weireducible\findsCE{\mathsf{F}_{2k+2}}$. Recall from \S \ref{wadgecomplexityofgraphs} that given a subspace $\mathcal{G} \subseteq \repspaceallgraphs$, we define $\bigotimes \tree_{\leq k} \defas \{
	   \disconnectedunion{n \in \nats}{G_n} : (\forall n)(G_n \in \tree_{\leq k})	\}$,
   where $\tree_{\leq k}$ is the space of trees with height at most $k$. \thref{treesforests} tells us that ${Forests}_{2k}\defas\{G \in \bigotimes \tree_{\leq k}:  \mathsf{F}_{2k} \subgraph G\}$ is $\Pi_{2k}^0$-complete: this allows us to think of an input $A$ for $\mathrm{EnumInf}_{\Pi_{2k}}$ as a sequence $(F_n)_{n \in \nats}$ such that $n \in A \iff F_n \in Forests_{2k}$. Since by \thref{definitionenuminf} $A$ is infinite, $\disconnectedunion{n \in \nats}{F_n}$ is a suitable input for $\finds{}{\mathsf{F}_{2k+2}}$ and from any $G \in \findsCE{\disconnectedunion{n \in \nats} F_n}$, checking the projection on the first coordinate of the nodes in $G$ we can compute an infinite subset of $A$.
		\end{proof}
   
		We conjecture that the reduction in \thref{maintheorem_findubgraphgn}$(ii)$ is actually an equivalence, but the details of the proof still need to be adjusted. We leave open whether there exists graphs $G$ such that the problem $\finds{}{G}\weiequiv f$ for some non computable $f \notin \{\CBaire, \choice{\boldfacePi_{2k}^0}{\nats} \times \parallelization{\jump{\lpo}{k}}, \parallelization{\jump{\lpo}{k}}\}$ for $k \in \nats$.

	 \section{Conclusions and open problems}
	 In this paper, we investigated the {subgraph problem} and the {induced subgraph problem} using tools from (effective) Wadge reducibility and Weihrauch reducibility. We studied the (effective) Wadge complexity of certain subsets of graphs, and we located decision problems and \quot{search}\ problems in the Weihrauch hierarchy. Regarding Weihrauch reducibility, we solved some questions left open in \cite{bement2021reverse}, but the exact (effective) Wadge complexity of some subsets of graphs we studied and the Weihrauch degree of certain problems remains open.

	 For the induced subgraph relation, in \S \ref{wadgecomplexityofgraphs} we showed that, at least when $G$ is computable, for any graph $G$, $\setofgraphs{G}{\repspaceallgraphs}{\inducedsubgraph} \in \Gamma$ for $\Gamma \in \{\Sigma_{1}^0, \Sigma_1^1\}$. \thref{subgraphcomplexity} witnesses that, for the subgraph case, there exists a graph $G_k$ such that $\setofgraphs{G_k}{\repspaceallgraphs}{\subgraph} \in \Gamma$ for $\Gamma \in \{\Sigma_{2k+1}^0,\Pi_{2k+2}^0, \Sigma_1^1\}$ for $k \in \nats$.

	 \begin{open}
		Is there a computable/c.e.\ graph $G$ and some $k \in \nats$ such that $\setofgraphs{G}{\repspaceallgraphs}{\subgraph} \in \Gamma$ for $\Gamma \notin \{\Sigma_{2k+1}^0,\Pi_{2k+2}^0, \Sigma_1^1\}$?
	 \end{open}

In terms of Weihrauch reducibility essentially the same question can be rephrased as follows.
\begin{open}
	Is there a computable/c.e.\ graph $G$ and some $k>0$ such that $\jump{\lpo}{k}\strictlyweireducible \sg_{G} \strictlyweireducible\jump{\lpo}{k+1}$?
 \end{open}

Regarding \quot{search}\ problems, \thref{maintheoreminducedsubgraph} shows that $\findis{}{G} \weiequiv \CBaire$ but, in case $G$ is such that
\[\{v \in V(G): \D{G}{v}<\aleph_0\}=\aleph_0 \text{ and }\{v \in V(G): \D{G}{v}=\aleph_0\}=\aleph_0,\]
the right-to-left reduction of the equivalence holds relative to an oracle.

\begin{open}
	Is it the case that for any computable/c.e.\ graph $G$,  $\findis{}{G}\weiequiv \CBaire$ (with no oracle involved)?
\end{open}

We have showed that the problems $\finds{}{\R}$ have the unusual property  of being hard to compute, but weak when they have to compute a problem on their own.  \thref{maintheoremray}, in particular, shows what $\finds{}{\R}$ cannot compute but, regarding what $\finds{}{\R}$ compute,  the only satisfactory result we have is that $\firstOrderPart{\finds{}{\R}}\weiequiv \firstOrderPart{\CBaire}$ (\thref{propositionfoprayomega}). For non first-order problems, the best we were able to show is  implied by \thref{ccantorfindsubgraph} and it shows that  $\CCantor\strictlyweireducible\findsCC{\R}$.

\begin{open}
	What can $\finds{}{\R}$ compute? 
\end{open}

\thref{corollarylpo1} shows that 
			$\mathrm{Fin}(\connected\findsEE{\R}) \strictlyweireducible \firstOrderPart{\connected\findsEE{\R}}$: it is natural to ask the following.

			\begin{open}
				Does	$\firstOrderPart{\connected\findsEE{\R}}\weiequiv f$ for some well-known $f$ in the Weihrauch lattice?
			\end{open}

We furthermore point out two open questions arising from \S \ref{othersubgraphproblems}. For the first one (asking whether the reduction in \thref{maintheorem_findubgraphgn}$(ii)$ is an equivalence), we conjecture a positive answer, but the details of the proof still need to be adjusted.

\begin{open}
	Does $\finds{}{\mathsf{T}_{2k+1}} \weiequiv \choice{\boldfacePi_{2k}^0}{\nats} \times \parallelization{\jump{\lpo}{2k-3}}$?
\end{open}

\begin{open}
	Is there a computable/c.e.\ a graph $G$ such that $\finds{}{G}\weiequiv f$ for some non computable $f \notin \{\CBaire, \choice{\boldfacePi_{2k}^0}{\nats} \times \parallelization{\jump{\lpo}{k}}, \parallelization{\jump{\lpo}{k}}\}$ for some $k \in \nats$?
\end{open}

We conclude by mentioning the dual problem: Rather than fixing the putative subgraph and take the putative supergraph as input, we can also consider the \emph{supergraph} problem, where the question is whether the input graph is isomorphic to a(n induced) subgraph of the fixed graph. We are exploring this problem in a companion paper \cite{cipriani2023complexity}.


\begin{thebibliography}{10}
	\providecommand{\bibitemdeclare}[2]{}
	\providecommand{\surnamestart}{}
	\providecommand{\surnameend}{}
	\providecommand{\urlprefix}{Available at }
	\providecommand{\url}[1]{\texttt{#1}}
	\providecommand{\href}[2]{\texttt{#2}}
	\providecommand{\urlalt}[2]{\href{#1}{#2}}
	\providecommand{\doi}[1]{doi:\urlalt{http://dx.doi.org/#1}{#1}}
	\providecommand{\bibinfo}[2]{#2}
	
	\bibitemdeclare{article}{bement2021reverse}
	\bibitem{bement2021reverse}
	\bibinfo{author}{Zach \surnamestart BeMent\surnameend},
	  \bibinfo{author}{Jeffry~L. \surnamestart Hirst\surnameend} \&
	  \bibinfo{author}{Asuka \surnamestart Wallace\surnameend}
	  (\bibinfo{year}{2021}): \emph{\bibinfo{title}{Reverse mathematics and
	  {W}eihrauch analysis motivated by finite complexity theory}}.
	\newblock {\sl \bibinfo{journal}{Computability}}
	  \bibinfo{volume}{10}(\bibinfo{number}{4}), pp. \bibinfo{pages}{343--354},
	  \doi{10.3233/COM-210310}.
	\newblock \urlprefix\url{https://doi.org/10.3233/COM-210310}.
	\newblock \bibinfo{note}{ArXiv 2105.01719}.
	
	\bibitemdeclare{article}{effectiveborelmeasurability}
	\bibitem{effectiveborelmeasurability}
	\bibinfo{author}{Vasco \surnamestart Brattka\surnameend}
	  (\bibinfo{year}{2004}): \emph{\bibinfo{title}{Effective Borel measurability
	  and reducibility of functions}}.
	\newblock {\sl \bibinfo{journal}{Mathematical Logic Quarterly}}
	  \bibinfo{volume}{51}(\bibinfo{number}{1}), pp. \bibinfo{pages}{19--44},
	  \doi{https://doi.org/10.1002/malq.200310125}.
	
	\bibitemdeclare{article}{paulybrattka}
	\bibitem{paulybrattka}
	\bibinfo{author}{Vasco \surnamestart Brattka\surnameend},
	  \bibinfo{author}{Matthew \surnamestart de~Brecht\surnameend} \&
	  \bibinfo{author}{Arno \surnamestart Pauly\surnameend} (\bibinfo{year}{2012}):
	  \emph{\bibinfo{title}{Closed Choice and a Uniform Low Basis Theorem}}.
	\newblock {\sl \bibinfo{journal}{Annals of Pure and Applied Logic}}
	  \bibinfo{volume}{163}(\bibinfo{number}{8}), pp. \bibinfo{pages}{968--1008},
	  \doi{10.1016/j.apal.2011.12.020}.
	
	\bibitemdeclare{article}{BG09}
	\bibitem{BG09}
	\bibinfo{author}{Vasco \surnamestart Brattka\surnameend} \&
	  \bibinfo{author}{Guido \surnamestart Gherardi\surnameend}
	  (\bibinfo{year}{2009}): \emph{\bibinfo{title}{Weihrauch Degrees, Omniscience
	  Principles and Weak Computability}}.
	\newblock {\sl \bibinfo{journal}{Journal of Symbolic Logic - JSYML}}
	  \bibinfo{volume}{76}, \doi{10.2178/jsl/1294170993}.
	
	\bibitemdeclare{article}{brattka2015probabilistic}
	\bibitem{brattka2015probabilistic}
	\bibinfo{author}{Vasco \surnamestart Brattka\surnameend},
	  \bibinfo{author}{Guido \surnamestart Gherardi\surnameend} \&
	  \bibinfo{author}{Rupert \surnamestart H{\"o}lzl\surnameend}
	  (\bibinfo{year}{2015}): \emph{\bibinfo{title}{Probabilistic computability and
	  choice}}.
	\newblock {\sl \bibinfo{journal}{Information and Computation}}
	  \bibinfo{volume}{242}, pp. \bibinfo{pages}{249--286}.
	
	\bibitemdeclare{article}{BolWei11}
	\bibitem{BolWei11}
	\bibinfo{author}{Vasco \surnamestart Brattka\surnameend},
	  \bibinfo{author}{Guido \surnamestart Gherardi\surnameend} \&
	  \bibinfo{author}{Alberto \surnamestart Marcone\surnameend}
	  (\bibinfo{year}{2012}): \emph{\bibinfo{title}{The {B}olzano-{W}eierstrass
	  Theorem is the jump of {W}eak {K}őnig's {L}emma}}.
	\newblock {\sl \bibinfo{journal}{Annals of Pure and Applied Logic}}
	  \bibinfo{volume}{163}, pp. \bibinfo{pages}{623--655}.
	
	\bibitemdeclare{incollection}{brattka2021weihrauch}
	\bibitem{brattka2021weihrauch}
	\bibinfo{author}{Vasco \surnamestart Brattka\surnameend},
	  \bibinfo{author}{Guido \surnamestart Gherardi\surnameend} \&
	  \bibinfo{author}{Arno \surnamestart Pauly\surnameend} (\bibinfo{year}{2021}):
	  \emph{\bibinfo{title}{Weihrauch complexity in computable analysis}}.
	\newblock In: {\sl \bibinfo{booktitle}{Handbook of {C}omputability and
	  {C}omplexity in {A}nalysis}}, \bibinfo{publisher}{Springer}, pp.
	  \bibinfo{pages}{367--417}.
	
	\bibitemdeclare{article}{montecarlo}
	\bibitem{montecarlo}
	\bibinfo{author}{Vasco \surnamestart Brattka\surnameend},
	  \bibinfo{author}{Rupert \surnamestart H{\"o}lzl\surnameend} \&
	  \bibinfo{author}{Rutger \surnamestart Kuyper\surnameend}
	  (\bibinfo{year}{2017}): \emph{\bibinfo{title}{{Monte Carlo Computability}}}
	  \bibinfo{volume}{66}, pp. \bibinfo{pages}{17:1--17:14}.
	\newblock \doi{10.4230/LIPIcs.STACS.2017.17}.
	\newblock \urlprefix\url{http://drops.dagstuhl.de/opus/volltexte/2017/7016}.
	
	\bibitemdeclare{article}{BP16}
	\bibitem{BP16}
	\bibinfo{author}{Vasco \surnamestart Brattka\surnameend} \&
	  \bibinfo{author}{Arno \surnamestart Pauly\surnameend} (\bibinfo{year}{2016}):
	  \emph{\bibinfo{title}{On the algebraic structure of {W}eihrauch degrees}}.
	\newblock {\sl \bibinfo{journal}{Logical Methods in Computer Science}}
	  \bibinfo{volume}{14}, \doi{10.23638/LMCS-14(4:4)2018}.
	
	\bibitemdeclare{article}{camerlo2005continua}
	\bibitem{camerlo2005continua}
	\bibinfo{author}{Riccardo \surnamestart Camerlo\surnameend}
	  (\bibinfo{year}{2005}): \emph{\bibinfo{title}{Continua and their
	  $\sigma$-ideals}}.
	\newblock {\sl \bibinfo{journal}{Topology and its Applications}}
	  \bibinfo{volume}{150}(\bibinfo{number}{1-3}), pp. \bibinfo{pages}{1--18}.
	
	\bibitemdeclare{article}{CB}
	\bibitem{CB}
	\bibinfo{author}{Vittorio \surnamestart Cipriani\surnameend},
	  \bibinfo{author}{Alberto \surnamestart Marcone\surnameend} \&
	  \bibinfo{author}{Manlio \surnamestart Valenti\surnameend}
	  (\bibinfo{year}{2022}): \emph{\bibinfo{title}{The {W}eihrauch lattice at the
	  level of $\mathbf{\Pi_1^1}\text{-}\mathsf{CA}_0$: the {C}antor-{B}endixson
	  theorem}}.
	\newblock {\sl \bibinfo{journal}{Submitted, available at
	  \url{https://arxiv.org/abs/2210.15556}}}, \doi{10.48550/ARXIV.2210.15556}.
	
	\bibitemdeclare{inproceedings}{cipriani2023complexity}
	\bibitem{cipriani2023complexity}
	\bibinfo{author}{Vittorio \surnamestart Cipriani\surnameend} \&
	  \bibinfo{author}{Arno \surnamestart Pauly\surnameend} (\bibinfo{year}{2023}):
	  \emph{\bibinfo{title}{The Complexity of Finding Supergraphs}}.
	\newblock In \bibinfo{editor}{Gianluca \surnamestart Della~Vedova\surnameend},
	  \bibinfo{editor}{Besik \surnamestart Dundua\surnameend},
	  \bibinfo{editor}{Steffen \surnamestart Lempp\surnameend} \&
	  \bibinfo{editor}{Florin \surnamestart Manea\surnameend}, editors: {\sl
	  \bibinfo{booktitle}{Unity of Logic and Computation}},
	  \bibinfo{publisher}{Springer Nature Switzerland}, \bibinfo{address}{Cham},
	  pp. \bibinfo{pages}{178--189}.
	
	\bibitemdeclare{article}{dzafarovsolomonyokoyama}
	\bibitem{dzafarovsolomonyokoyama}
	\bibinfo{author}{Damir~D. \surnamestart Dzhafarov\surnameend},
	  \bibinfo{author}{Reed \surnamestart Solomon\surnameend} \&
	  \bibinfo{author}{Keita \surnamestart Yokoyama\surnameend}
	  (\bibinfo{year}{2023}): \emph{\bibinfo{title}{On the first-order parts of
	  problems in the Weihrauch degrees}}.
	\newblock {\sl \bibinfo{journal}{Available at
	  \url{https://arxiv.org/abs/2301.12733}}}.
	
	\bibitemdeclare{article}{pauly-valenti}
	\bibitem{pauly-valenti}
	\bibinfo{author}{Jun~Le \surnamestart Goh\surnameend}, \bibinfo{author}{Arno
	  \surnamestart Pauly\surnameend} \& \bibinfo{author}{Manlio \surnamestart
	  Valenti\surnameend} (\bibinfo{year}{2021}): \emph{\bibinfo{title}{Finding
	  descending sequences through ill-founded linear orders}}.
	\newblock {\sl \bibinfo{journal}{Journal of Symbolic Logic}}
	  \bibinfo{volume}{86}(\bibinfo{number}{2}), \doi{10.1017/jsl.2021.15}.
	
	\bibitemdeclare{article}{gura2015existence}
	\bibitem{gura2015existence}
	\bibinfo{author}{Kirill \surnamestart Gura\surnameend},
	  \bibinfo{author}{Jeffry~L \surnamestart Hirst\surnameend} \&
	  \bibinfo{author}{Carl \surnamestart Mummert\surnameend}
	  (\bibinfo{year}{2015}): \emph{\bibinfo{title}{On the existence of a connected
	  component of a graph}}.
	\newblock {\sl \bibinfo{journal}{Computability}}
	  \bibinfo{volume}{4}(\bibinfo{number}{2}), pp. \bibinfo{pages}{103--117}.
	
	\bibitemdeclare{article}{leafmanaegement}
	\bibitem{leafmanaegement}
	\bibinfo{author}{Jeffry \surnamestart Hirst\surnameend} (\bibinfo{year}{2019}):
	  \emph{\bibinfo{title}{Leaf management}}.
	\newblock {\sl \bibinfo{journal}{Computability}} \bibinfo{volume}{9}, pp.
	  \bibinfo{pages}{1--6}, \doi{10.3233/COM-180243}.
	
	\bibitemdeclare{book}{kechris2012classical}
	\bibitem{kechris2012classical}
	\bibinfo{author}{A.~\surnamestart Kechris\surnameend} (\bibinfo{year}{2012}):
	  \emph{\bibinfo{title}{Classical Descriptive Set Theory}}.
	\newblock \bibinfo{series}{Graduate Texts in Mathematics},
	  \bibinfo{publisher}{Springer New York}.
	\newblock \urlprefix\url{https://books.google.co.uk/books?id=WR3SBwAAQBAJ}.
	
	\bibitemdeclare{article}{kihara_marcone_pauly_2020}
	\bibitem{kihara_marcone_pauly_2020}
	\bibinfo{author}{Takayuki \surnamestart Kihara\surnameend},
	  \bibinfo{author}{Alberto \surnamestart Marcone\surnameend} \&
	  \bibinfo{author}{Arno \surnamestart Pauly\surnameend} (\bibinfo{year}{2020}):
	  \emph{\bibinfo{title}{Searching for an analogue of $\mathsf{ATR}_0$ in the
	  {W}eihrauch lattice}}.
	\newblock {\sl \bibinfo{journal}{The Journal of Symbolic Logic}}
	  \bibinfo{volume}{85}(\bibinfo{number}{3}), p. \bibinfo{pages}{1006–1043},
	  \doi{10.1017/jsl.2020.12}.
	
	\bibitemdeclare{article}{manaster1972effective}
	\bibitem{manaster1972effective}
	\bibinfo{author}{Alfred~B \surnamestart Manaster\surnameend} \&
	  \bibinfo{author}{Joseph~G \surnamestart Rosenstein\surnameend}
	  (\bibinfo{year}{1972}): \emph{\bibinfo{title}{Effective matchmaking
	  (recursion theoretic aspects of a theorem of Philip Hall)}}.
	\newblock {\sl \bibinfo{journal}{Proceedings of the London Mathematical
	  Society}} \bibinfo{volume}{3}(\bibinfo{number}{4}), pp.
	  \bibinfo{pages}{615--654}.
	
	\bibitemdeclare{article}{moschovakis}
	\bibitem{moschovakis}
	\bibinfo{author}{Yiannis~Nicholas \surnamestart Moschovakis\surnameend}
	  (\bibinfo{year}{1982}): \emph{\bibinfo{title}{Descriptive Set Theory}}.
	\newblock {\sl \bibinfo{journal}{Studia Logica}}
	  \bibinfo{volume}{41}(\bibinfo{number}{4}), pp. \bibinfo{pages}{429--430}.
	
	\bibitemdeclare{article}{valentisolda}
	\bibitem{valentisolda}
	\bibinfo{author}{Giovanni \surnamestart Solda\surnameend} \&
	  \bibinfo{author}{Manlio \surnamestart Valenti\surnameend}
	  (\bibinfo{year}{2023}): \emph{\bibinfo{title}{Algebraic properties of the
	  first-order part of a problem}}.
	\newblock {\sl \bibinfo{journal}{Ann. Pure Appl. Log.}}
	  \bibinfo{volume}{174}(\bibinfo{number}{7}), p. \bibinfo{pages}{103270},
	  \doi{10.1016/j.apal.2023.103270}.
	\newblock \urlprefix\url{https://doi.org/10.1016/j.apal.2023.103270}.
	
	\bibitemdeclare{book}{Weihrauch}
	\bibitem{Weihrauch}
	\bibinfo{author}{Klaus \surnamestart Weihrauch\surnameend}
	  (\bibinfo{year}{2013}): \emph{\bibinfo{title}{Computable Analysis: An
	  Introduction}}, \bibinfo{edition}{1st} edition.
	\newblock \bibinfo{publisher}{Springer Publishing Company, Incorporated}.
	
	\end{thebibliography}
\end{document}